\documentclass[aos]{imsart}
\usepackage[dvipsnames]{xcolor}
 \usepackage{amsthm,amsmath}

\usepackage{amsbsy,verbatim}
 \usepackage[sort]{natbib}
 \usepackage{jbradic_definitions}
\RequirePackage[colorlinks,citecolor=blue,urlcolor=blue]{hyperref}

\usepackage{graphicx}
\allowdisplaybreaks[4]
\usepackage[final]{pdfpages}

\startlocaldefs

\global\long\def\RR{\mathbb{R}}

\global\long\def\trace{{\rm trace}}

\global\long\def\boldone{\mathbf{1}}

\global\long\def\tA{\tilde{A}}

\global\long\def\tmu{\tilde{\mu}}

\global\long\def\deltab{\bm{\delta}}

\global\long\def\Omegab{\bm{\Omega}}

\global\long\def\Ncal{\mathcal{N}}

\global\long\def\II{\mathbb{I}}

\global\long\def\Sigmab{\boldsymbol{\Sigma}}

\global\long\def\gammab{\boldsymbol{\gamma}}

\global\long\def\KL{{\rm KL}}

\global\long\def\TV{{\rm TV}}

\global\long\def\tgammab{\boldsymbol{\tilde{\gamma}}}

\global\long\def\yb{\mathbf{y}}

\global\long\def\Zb{\mathbf{Z}}

\global\long\def\Wb{\mathbf{W}}

\global\long\def\tB{\tilde{B}}

\global\long\def\xib{\boldsymbol{\xi}}

\endlocaldefs

  \usepackage{xr}
\begin{document}

\begin{frontmatter}

\title{
Testability of high-dimensional  linear models with non-sparse structures}
\runtitle{Testability of non-sparse models}


 \author{\fnms{Jelena} \snm{Bradic}\corref{}\thanksref{m1}\ead[label=e1]{jbradic@ucsd.edu}}
\address{Department of Mathematics, \\
Halicio\u{g}lu Data Science Institute, \\
University of California, San Diego,\\
\printead{e1}}

\author{\fnms{Jianqing } \snm{Fan}\thanksref{m2}\ead[label=e2]{jqfan@princeton.edu}}
\address{Department of Operations Research \\
\ \ \ \ and Financial Engineering, \\
Princeton University,\\
\printead{e2}}
\and
\author{\fnms{Yinchu} \snm{Zhu}\thanksref{m3}\ead[label=e3]{yzhu6@oregon.edu}}
\address{Lundquist College of Business, \\
University of Oregon,\\
\printead{e3}}

\affiliation{University of California, San Diego\thanksmark{m1}, Princeton University\thanksmark{m2} \\ and University of Oregon\thanksmark{m3}  }

\thankstext{m1}{Bradic's research is supported by NSF Grant DMS-1712481.}
\thankstext{m2}{Fan's research is supported by NSF Grants DMS-1662139 and  DMS-DMS-1712591 and NIH grants 5R01-GM072611-12.}
 
\runauthor{J. Bradic, J. Fan and Y. Zhu}

 \begin{abstract}

Understanding statistical inference under possibly non-sparse high-dimensional models has gained much interest recently. For a given component of the regression coefficient, we show that the difficulty of the problem depends on the sparsity of the corresponding row of the precision matrix of the covariates, not the sparsity of the regression coefficients. We develop new concepts of uniform and essentially uniform non-testability that allow the study of limitations of tests across a broad set of alternatives.
Uniform non-testability identifies a collection of alternatives such that the power of any test, against any alternative in the group, is asymptotically at most equal to the nominal size.
Implications of the new constructions include new minimax testability results that, in sharp contrast to the current results, do not depend on the sparsity of the regression parameters. We identify new tradeoffs between testability and feature correlation. In particular, we show that, in models with weak feature correlations, minimax lower bound can be attained by a test whose power has the $\sqrt{n}$ rate, regardless of the size of the model sparsity. 
\end{abstract}



\end{frontmatter}


\section{Introduction}
  Confidence intervals construction and hypothesis testing in high-dimensional studies arise in  almost all   modern application areas, ranging from biomedical imaging \citep{chalkidou2015false} or disease tracking,  to the discovery of  genetic variants associated with normal and disorder-related phenotypic variance in brain  function
\citep{krishnan2016genome,ganjgahi2018fast}, to the evaluation of policy and marketing strategies  \citep{verhoef2017consumer}, and many more.  There has been considerable interest in developing valid statistical methods for the construction of confidence intervals in high-dimensional problems. Some notable recent advances include proposals based on the ridge estimate \citep{buhlmann2013statistical,nickl2013confidence}, on the lasso estimate \citep{zhang2014confidence,van2014asymptotically},  score  and orthogonal moments methods \citep{belloni2014inference, goeman2006testing}, as well as combinations thereof  (see for example \cite{belloni2014uniform, javanmard2014confidence}).

This line of work has produced many promising methods. The literature, however, does not provide an answer as to how these methods should be adapted for the possible lack of sparse structures in the underlying models. First, there is no guidance on how to check whether a model is sparse or not. The majority of current approaches construct confidence intervals under a set of assumptions describing how sparse the underlying model is. The process of developing algorithms that detect model sparsity is still somewhat ``unattainable'', therefore in practice effectively rendering a priori belief in the sparsity.  Second, no formal guarantees have been provided, to either confirm or deny, the ability to perform a hypothesis test (or to construct optimal confidence intervals); not without imposing sparsity on model parameters.

In this paper, our primary goal is a theoretical understanding of the high-dimensional minimax theory that can address both of these concerns. Our framework allows for high dimensional linear models that are not necessarily sparse. We illustrate that moving away from assumptions on sparse parameters towards assumptions on the design matrix can allow for certain optimal inferences. Moreover, we show how the estimators and tests can be designed to achieve these new optimality results.

We formalize our results in terms of the high-dimensional linear regression:
  \begin{equation}\label{eq:model1}
  \yb = \Xb \betab + \varepsilonb, \qquad \varepsilonb \sim \mathcal{N}(0, \sigma^2 \mathbb{I}_n)
  \end{equation}
where $\yb=(y_1,\cdots,y_n)^{\top} \in \RR^n$, $\Xb$ is a collection of $n$ i.i.d. vectors, $\Xb= (\Xb_1, \dots, \Xb_n)^\top \in \RR^{n \times p}$ and $\betab \in \RR^p$ whose dimensionality $p$ can be much larger than the sample size $n$. Here, the covariance matrix of $\Xb$  is denoted by $\Sigmab= \EE (\Xb_i \Xb_i^\top) $, whereas its precision matrix is denoted by $\Omegab=\Sigmab^{-1}$. We denote with $k=\|\betab\|_{0} $. In this paper, our focus is on the problem of testing individual entries of $\betab$. Without loss of generality, we consider the first entry and denote $\betab=(\beta,\gammab^{\top})^{\top} \in \RR^{p} $.

We provide a motivating result first. Note that $\beta $ can be represented as a linear combination of easily estimable quantity, $\EE(\Xb_i y_i) $, with the weights being the first row of $\Omegab$. We investigate if particular structures in $\Omegab$ can be leveraged to remove sparsity assumptions on $k$. When $\Omegab$ is known, we show that a simple plug-in estimate achieves the parametric rate for a full range  $0 \leq k \leq p$. Hence, there is hope that strict sparsity requirements on $k$ are not necessary for valid inference. However, there are significant hurdles that need to be cleared before minimax results can be directly developed for inference on models that are not necessarily sparse.

An impediment to exploring high-dimensional models is the fear that the  researchers will search for essential variables, and then report only the results for variables with extreme effects, which in turn, are dependent on the existence of only a small number of significant signals; highlighting thus the signal that may be purely spurious. For this reason, such practices must specify in advance that only a few signals are ``real'' and then they proceed to find them.   However, such procedures can make it difficult to discover
strong but unexpected signals. In this paper, we seek to address
this challenge by developing a method that yields valid asymptotic confidence intervals for the real underlying
signal by moving away from conditions on the conditional expectation of $\yb|\bX$ to exploring structures in the distribution of $\bX$ such as the sparsity of   $\Omegab$. We showcase that sparsity in $\Omegab$ can allow for arbitrary growth of $k$.

 In GWAS studies, an agnostic approach to the conditional distribution of the response is especially valuable.
Since around 2006, the advent of GWAS, and more recently exome sequencing, has provided the first detailed understanding of the genetic basis of complex traits. To explain ``missing heritability,''  a new paradigm has emerged in which complex disease is driven by an accumulation of a large number of weak effects across all of the network of genetic pathways \citep{boyle2017expanded, furlong2013human,chakravarti2016revealing}. Similarly, it is deeply understood that microbial functional relationship to the host is highly complex, that microbial communities have highly complex structures and that small and numerous changes in the network affect the host adversely \citep{huttenhower2012structure}. At the same time, it is widely believed that features in many studies have a sparse correlation structure (providing evidence of sparse $\Omegab$). For example, only a certain number of genes functionally depend on one-another, clump together. Similarly, far apart, SNPs are very nearly independent \citep{janson2017eigenprism}, so we may expect that the true $\Omegab$ has nearly banded structure.

Therefore,  for many practically relevant examples, it is not necessary nor wise to impose a sparse structure on the conditional distribution of $\yb|\bX $; after all, if we are studying $\yb|\bX $, that typically means we do not know very much about it.

Our detection rates are stated in terms of $s$, the number of non-zero entries in the first row of $\Omegab$ as well as the size of the $\|\betab \|_2$.
Thanks to the newly defined optimality criterion,  the rate  $  n^{-1/2}+sn^{-1}\log p  $ is  identified as the minimax rate of   detection for the problem of identifying the null $\beta=\beta_0$ against the alternative $ \beta=\beta_0+h$, whenever $\|\betab \|_2 \leq \kappa $ for a fixed $\kappa >0$, regardless of the size of the model sparsity $k$. When $\Omegab$'s   first row is sparse enough, we provide a minimax optimal test and a confidence interval for $\beta$ without assuming an upper bound on $k$. We identify as well that even with knowledge of sparsity of $\Omegab$, the detection rate will not tend to zero if   $  \kappa \gtrsim \sqrt{n}$ and no constraint is imposed on $k$.

  
  We propose a novel framework to study the detection rates for  $\beta$ while allowing $k\lesssim p$. Impossibility results are established under the new concept of (essentially) uniform non-testability.
  We state that the null hypothesis is uniformly non-testable against the alternative if the power of any test of nominal level $\alpha$ against {\it any} point in the alternative is at most $\alpha$. The proposed uniform non-testability results also provide new insights. Under uniform non-testability, testing the null hypothesis against one (arbitrary) point is impossible for {\it any} test. Since any test that has size control is powerless against every point in the alternative, our work indicates that the difficulty in these testing problems is quite fundamental.
Besides, the new non-testability results allow for a characterization of non-adaptivity; in a certain sense, those two notions match. It will enable us to shed new light on the existing literature on the adaptivity of testing. Ideally, an adaptive
confidence interval should have its length automatically adjusted to the actual
sparsity of the unknown coefficient vector, while maintaining a pre-specified coverage
probability. We showcase that with known $\Omegab$ this can be done while for the unknown $\Omegab$, adaptivity requires $s \ll \sqrt{n}/\log p $; both results do not depend on the size of $k$.

\subsection{Existing literature}

Under the linear model above, the parameter of interest can be written as
\[
\beta = \Omegab_{,1} \EE(\Xb_i y_i)
\]
  where $\Omegab_{,1} \in \RR^p$ denotes the first row of $\bOmega$.
As a consequence of this representation, it may be tempting to first estimate $\Omegab_{,1} $ as well as $\EE(\Xb_i y_i) $ and then set
$\tilde \beta = \hat \Omegab_{,1} \hat \mu$,
where $\hat \mu =n^{-1}\sum_{i=1}^n \Xb_i y_i$.
This simple approach, however, is often not optimal: Because $n^{-1}\sum_{i=1}^n \Xb_i y_i$ is a $p$-dimensional vector, that does not have to have any sparse structures,  the product may be highly unstable. As an example, consider fitting the graphical lasso \citep{meinshausen2006high,wainwright2007high} to estimate $ \Omegab_{,1}$.
  A naive approach would make a product of such an estimate and $\hat \mu$ to construct $\tilde \beta$. However, since $\hat{\Omegab}_{,1}$ is regularized towards zero, the bias at estimation will propagate in all elements of $\hat \mu$ and, therefore, the product.

  The recent literature on high-dimensional inference has proposed several ideas on how to avoid such ``regularization bias''. In particular, several recent papers have proposed structural changes to various regularized methods, aimed at accurate estimation of $\beta$ \citep{buhlmann2013statistical,nickl2013confidence,goeman2006testing,belloni2014inference,belloni2014uniform,zhang2014confidence,van2014asymptotically,javanmard2014confidence}. These approaches always correctly de-bias the estimates for valid high-dimensional inference. However, they assume various sparsity structures in their analysis without which no guarantees are provided for validity.
In detail, their analysis relies on the assumption that the vector of the nuisance parameters belongs to the set of $k$-sparse regression vectors with $k \ll \sqrt{n} /\log p$. Such sparsity requirement recently raised considerable interest since it appears to be a  much stronger condition than that needed for consistent estimation, which only imposes $ k \ll n/\log p $; see, e.g., \cite{negahban2009unified,raskutti2011minimax}.

The natural question is whether the strong condition of $k\ll \sqrt{n}/\log p $ is needed. The pioneering   work   of   \cite{cai2017confidence} and \cite{javanmard2015biasing} aim   to address this question, where the former derives the minimax rate for the expected length of confidence intervals  assuming $k \lesssim n/\log p$ and the latter,   in a different context,   improves the condition $k\ll \sqrt{n}/\log p $ to $k \ll n/(\log p)^2$. This work provides a complementary study where we reveal an intricate relationship between sparsity (or the non-existence of thereof) and $\ell_2$-norm constraints.

Another line of work, closer to our paper, has focused on inference approaches not closely relying on sparsity assumptions; see e.g., \cite{shah2017goodness} and \cite{janson2017eigenprism}. The work of \cite{ZhuBradic2017a} and \cite{ZhuBradic2017b}  is particularly close to ours. There, the authors propose asymptotically exact confidence interval construction under no model sparsity assumption. However, therein, no formal optimality guarantees were derived beyond several specific examples in which the model parameters are restricted to be small or approximately sparse. 
Therefore, it is not apparent what the optimal detection rate is for general non-sparse models, and it is not expected that methods discussed therein can provide uniform guarantees for an ample parameter space.  Inspired by those findings, we asked whether any formal, minimax guarantees can be provided for a class of dense models? If so, what kind of estimates would be able to achieve the fundamental limits of detection? We identify that sample-splitting helps guarantee uniform detection rates. We discuss in detail sparsity in the precision matrix as being sufficient and necessary tools for this purpose. We also showcase an increase in the minimax (testing) rates whenever $\ell_2$-norm of the model parameters is not bounded, and the model is not necessarily sparse.

\subsection{Organization of the paper}

The rest of the paper is organized as follows: After basic notation is introduced,
Section \ref{sec:2} presents a precise formulation of the problem and some initial insights.
Section \ref{sec:3} establishes two impossibility results under the lack of sparsity in the first row of $\Omegab$. These results provide a lower bound on the detection rates.   Section \ref{sec:4} focuses on the upper bounds and the attainability of lower bounds.   Section \ref{sec:5} discusses connections to the minimax rates of detection and adaptivity of the confidence intervals.  Section \ref{sec:7} discusses minimax detection rates with growing $\ell_2$ balls. The proofs of all of the results are presented in the Appendices: A-C are collected in the main document whereas D-L are presented in the Supplement.

\section{Problem setup}\label{sec:2}

We present in
this section the   framework for  hypothesis in  high-dimensional models that are  not necessarily sparse. We begin with the notation that will be used throughout the manuscript.

\subsection{Notation}

For a matrix $\bX \in \RR^{n \times p}$, $\mathbf X_i$, $\mathbf X_{,j}$ and $X_{ij}$ denote respectively
the $i$-th row, $j$-th column and $(i,j)$ entry of the matrix $\bX$, $\mathbf X_{i, -j}$  denotes the $i$-th row
of $\bX$ excluding the $j$-th coordinate, and $\bX_{-j}$ denotes the submatrix of $\bX$ excluding
the $j$-th column. Let $[p]=\{1, 2,\dots,p\}$. For a subset $J \subseteq [p]$, $\bX_{J}$ denotes the
submatrix of $\bX$ consisting of columns $\mathbf X_{,j}$ with $j \in J$ and for a vector $\mathbf x \in \RR^p$,
$\mathbf  x_J$ is the $p$-dimensional vector that has the same coordinates as $\mathbf  x$ on $J$ and zero coordinates on the complement $J^c$ of $J$.
Let $\mathbf  x_{-J}$ denote the subvector with indices
in $J^c$. For a set $S$, $|S|$ denotes its cardinality. For a vector $\mathbf  x \in \RR^p$, $\mbox{supp}(\mathbf  x)$
denotes the support of $\mathbf  x$ and the $\ell_q$-norm of $\mathbf x$ is defined as $\| \mathbf  x \| _q^q  = \sum_{j \in [p]} |x_j|^q$ for $q \geq 0$, with $\| \mathbf  x \|_0 = |\mbox{supp}(\mathbf  x)|$ and $\|\mathbf  x\|_{\infty} = \max_{j \in [p]} |x_j|$.     For a matrix $\bA$ and $1 \leq q \leq \infty$, $\| \bA\|_q = \sup_{\mathbf x: \|\mathbf  x\|_q=1}\| \bA\|_q$.   For a symmetric matrix $\bA$, $\lambda_{\min}(\bA)$ and $\lambda_{\max}(\bA)$ denote respectively
the smallest and largest eigenvalue of $\bA$. $\mathbb{I}_q$ denotes the $q\times q$ identity matrix.   For two positive sequences
$a_n$ and $b_n$, $a_n \lesssim b_n$  means $a_n \leq C b_n$ for a positive constant $C$ independent of $n$. Moreover, we use $a_n \asymp b_n $ if $b_n \lesssim  a_n$ and $a_n \lesssim b_n$.
Lastly, $a _n \ll b_n$ is used to denote that  $ \lim_{n \to \infty} a_n /b_n =0$.  For a $\in \RR$, let $\left\lfloor a\right\rfloor $
denote the largest integer that is at most $a$.

\subsection{High-dimensional linear models that are not necessarily sparse}

We shall focus on the high-dimensional linear model \eqref{eq:model1} with the random design such that $\bX_i \sim \mathcal{N}_p(0, \Sigmab)$, $1 \leq i \leq n$ and are independent of the error $\varepsilon$. Note that both $\bSigma$ and the noise level $\sigma$   are  considered as unknown.
Since our problem is centered around the construction of confidence intervals for the univariate parameter $\beta$, we re-parametrize model \eqref{eq:model1} as
\begin{equation}\label{eq:model2}
  \yb = \Zb \beta + \Wb \gammab + \varepsilonb, \qquad \varepsilonb \sim \mathcal{N}(0, \sigma^2 \mathbb{I}),
  \end{equation}
where  $\beta \in \RR$, $\gammab \in \RR^{p-1}$, $\Zb = (Z_1,\dots, Z_n)^\top \in \RR^n$ and $\Wb =(\Wb_1,\dots, \Wb_{n})^\top \in \RR^{n \times (p-1)}$.
 The distribution of the data is now indexed by the parameter
 $$\theta = (\beta, \bgamma, \bSigma, \sigma),$$ which consists of parameter of interest $\beta$, the nuisance parameters $\gammab$,  the covariance matrix $\bSigma =   \EE[ \Xb_i \Xb_i^\top ]$ of
 the random design vector $\Xb_i=(Z_i, \Wb_i^\top)^\top$, and the variance of the noise $\sigma$.  Observed
data   $\Db=\{D_1,\ldots,D_n\} $ consists of i.i.d.  triplets $D_i=(y_i, Z_i, \bW_i)$,    for  $i=1,\dots, n$.
Note that $\beta$ in \eqref{eq:model2} can be represented as
\begin{equation} \label{fan1}
\beta = \bOmega_{1,}\EE(\bX^{\top}_i y_i).
\end{equation}
  Since each element of $\EE(\bX_i^\top y_i)$ can be easily estimated at a root-$n$ rate, the estimability of $\beta$ depends on $\bOmega$ only through its first row.
Hence, it seems prudent to define a parameter space that includes both the parameters of the model as well as  the matrix $\bOmega$,
\begin{align}
\widetilde{\Theta}= \biggl\{ \theta & =(\beta,\gammab,\Sigmab,\sigma) : \ \beta\in \RR,   \biggl.
\nonumber \\ \biggl.
 &  \  M^{-1}\leq\lambda_{\min}(\Sigmab)\leq\lambda_{\max}(\Sigmab)\leq M,  0\leq\sigma\leq M_{1},   \| \bbeta\|_2 \leq M_{2} \biggl\},\label{eq:thetatilde}
\end{align}
where $\bbeta = (\beta, \bgamma^{\top})^{\top}$ and  $M>1$, $M_{1}$  and $M_2$ are   positive constants.  Note that
$$
   \Var(y_{i}) = \bbeta^T \bSigma \bbeta + \sigma^2 \geq \lambda_{\min}(\bSigma) \|\bbeta\|_2^2.
$$
Thus, the constraint $\|\bbeta\|_2 \leq M_2$ can be dropped in the above definition if we only consider bounded $\Var(y_{i})$.

Observe that  whenever $\Omegab$ is known, a simple plug-in estimate
\begin{equation}\label{eq: plug-in estimator}
\hat{\beta}=\bOmega_{1,}\Xb^{\top}\yb/n
\end{equation}
  achieves the parametric rate without any assumption on $k$. Namely, we provide the following result.

\begin{theorem}\label{THM: THM KNOWN SIGMA}
For $\widetilde \Theta$ defined in \eqref{eq:thetatilde} we have
\begin{equation}\label{EQ:FIRST}
 \sup_{\theta\in \wtTheta }\EE_{\theta} |\hat{\beta}-\beta| \asymp n^{-1/2}.
 \end{equation}
\end{theorem}

Theorem \ref{THM: THM KNOWN SIGMA} is an oracle-like statement that holds for bounded constant $M_2$. It indicates that a parametric rate  of detection is possible for dense parameters (with $p \gg n$) with bounded $\ell_2$ norm.  Majority of the present paper focuses on the parameter space $\widetilde{\Theta}$.   In Section \ref{sec:7} we showcase minimax optimality rates that do not restrict the growth of $M_2$.

Observe that the above result   allows for $p \gg n$; in fact, it does not put any restrictions on the growth of $p$ or $k$.
Additionally, Theorem \ref{THM: THM KNOWN SIGMA} identifies that the inference for non-sparse high-dimensional models is possible as long as the precision matrix is known.  It indicates that the ability to decorrelate the features (i.e., to estimate  $\Omegab_{1,}$ well) is the key to efficient inference in high-dimensional non-sparse models.

To further study lower limits of detection of testing
\[
H_0: \beta =\beta_0
\]
our focus is on the parameter spaces defined by
 \begin{multline} \label{eq:6}
\Theta=\biggl\{\theta=(\beta,\boldsymbol{\gamma},\boldsymbol{\Sigma},\sigma):\ \beta\in \RR, \ M^{-1}\leq\lambda_{\min}(\boldsymbol{\Sigma})\leq\lambda_{\max}(\boldsymbol{\Sigma})\leq M,\\
 \boldsymbol{\Sigma}_{(-1),(-1)}=\mathbb{I}_{p-1},0\leq\sigma\leq M_{1}, \text{ and }  \| \bbeta\|_2 \leq M_{2}\biggr\},
\end{multline}
where $M>1$ and $ M_1,M_2>0$ are some universal constants. To study upper limits of detection we still analyze $\widetilde \Theta$ as defined in \eqref{eq:thetatilde}.
It is worth pointing that none of the parameter spaces, $\widetilde \Theta$ or $\Theta$, restricts $k$, the number of non-zero elements in $\betab$ of the linear model (\ref{eq:model1}), or the $\ell_1$-norm of $\betab$ (which can grow at a rate of $\sqrt{p}$). Our work is hence   very different from existing minimax studies.
We also define
 \[
 \Theta(s_{0},\beta_{0})=\left\{ \theta=(\beta,\gammab,\Sigmab,\sigma) \in  \Theta: \beta=\beta_0, \ \|\Omegab_{1,}\|_{0}\leq s_{0}\right\}
\]
and
 \[
 \Theta(s_{0})=\bigcup_{\beta\in\mathbb{R}}   \Theta(s_{0},\beta).
\]

The main goal of this paper is to address the following questions:
\begin{itemize}
\item[1.] {\it Is it possible to have accurate inference procedure about univariate parameters without requiring  the model parameter $\betab$ itself to be sparse?}
\item[2.] {\it Is the accuracy in terms of the detection rates uniform over the parameter space? }
   \end{itemize}

\section{Lower bound}\label{sec:3}

For $0 < \alpha <1$ and a given parameter space $\Theta_1 $,  the set of tests of nominal level $\alpha\in(0,1)$
regarding the null hypothesis $\theta\in\Theta_1$ is  denoted with
\[
\Psi_{\alpha}(\Theta_1)=\left\{ \psi:\ \Db \mapsto[0,1]\  : \ \sup_{\theta\in\Theta_1}\EE_{\theta}\psi\leq\alpha\right\},
\]
see, e.g., \cite{lehmann2006testing}.
Here, we allow for both random and non-random tests.


\begin{definition}[Uniform non-testability] \label{def:1}
Consider the hypothesis testing problem of $H_{0}:\ \theta\in\Theta^{(1)}$
versus $H_{1}:\ \theta\in\Theta^{(2)}$. We say that $\Theta^{(1)}$
is asymptotically uniformly non-testable against $\Theta^{(2)}$ at
size $\alpha\in(0,1)$ if $\limsup_{n\rightarrow\infty} \sup_{\theta\in\Theta^{(2)}} \EE_{\theta}\psi\leq\alpha$
for any test $\psi\in\Psi_{\alpha}(\Theta^{(1)})$.
\end{definition}

Above Definition \ref{def:1} introduces new concept of testability.
Per Definition \ref{def:1}  there does not exist a test that is better  than   a  simple coin toss. Since a simple coin toss  is   uniformly most powerful  asymptotically, the data cannot provide sufficient statistical evidence to distinguish the null from the alternative hypothesis.
This concept provides  an alternative to the widely known
  minimax-type results which state that for any test, there is one ``difficult'' point in the alternative for which this test has no power; therefore, it is possible that beyond this ``difficult'' point, there might exist a test that has good power against  all the other points. We could argue that we are proposing a different and not necessarily better characterization of optimality.

To characterize   alternative hypothesis we  introduce
\begin{multline*}
\Theta_{\zeta,\kappa }(s/2,\beta_{0}+h_{n})=\biggl\{  \theta \in\Theta(s/2, \beta_{0}+a):\  0\leq a \leq h_n,  \| \bbeta\|_2 \leq \zeta M_2,\\
(\zeta M)^{-1}\leq\lambda_{\min}(\Sigmab)\leq\lambda_{\max}(\Sigmab)\leq \zeta M,\  \kappa \leq\sigma\leq \zeta M_{1} \biggl\},
\end{multline*}
where $h_n$ is a sequence of positive numbers  and $\zeta\in (M^{-1},1) $ and $\kappa \in(0,\zeta M_{1}) $ are constants.
\begin{theorem}
\label{THM:4}Suppose that $sn^{-1}\log p\leq1/4$ and $2\leq s\leq p^{c}$
for some constant $c<1/2$. Then we have that for any $\beta_{0}$
\[
\limsup_{n\rightarrow\infty}\sup_{\psi\in\Psi_{\alpha}(\Theta(s,\beta_{0}))}\sup_{\theta\in\Theta_{\zeta,\kappa }(s/2,\beta_{0}+h_{n})}\EE_{\theta}\psi = \alpha,
\]
where $h_{n}=\rho sn^{-1}\log p$ and
\begin{multline} \label{eq: def rho}
  \rho=\min \biggl\{  4,\  \frac{1/2-c}{15(\kappa^{-2}M+1)},\ \frac{2\left(\zeta^{-1}-1\right)^{2}}{M^{3}(2M+1)},\ \frac{2M(1-\zeta)^{2}}{2M+1},\\
   \frac{(1-\zeta^{2})M_{2}}{8\zeta\sqrt{M}},\ \frac{\kappa^{2}(1-\zeta^{2})^{2}M_{2}^{2}}{64\zeta^{4}MM_{1}^{2}},\ \frac{M_{2}\sqrt{1-\zeta^{2}}}{2\sqrt{M}}
,\ \frac{\kappa^{2}(1-\zeta^{2})M_{2}^{2}}{4\zeta^{2}M_{1}^{2}M}
    \biggr\}.
\end{multline}
\end{theorem}

Theorem \ref{THM:4} establishes that
$\Theta(s,\beta_{0})$ is uniformly non-testable against all points in the alternative $\Theta_{\zeta,\kappa}(s/2,\beta_{0}+h_{n})$, i.e. {\it every } point in $\Theta_{\zeta,\kappa}(s/2,\beta_{0}+h_{n})$ is  difficult for {\it every} test. The distance to the alternatives, $\rho sn^{-1}\log p$, depends on the unknown constants $M,M_1$ and $M_2$ characterizing invertability of the covariance matrix $\mathbf \Sigma$, noise level $\sigma$ and the norm $\|\bbeta\|_2$, respectively;  see e.g., \eqref{eq: def rho}.

 This result is unique in its treatment of nuisance parameters $\gammab$, which are allowed to be fully dense. In the case of dense models, the lower bound for detection depends on how sparse $\Omegab_{1,}$ is: it is impossible to have power in testing $\beta=\beta_0$ against $\beta=\beta_0+h $ whenever $|h|\leq \rho sn^{-1} \log p $. One implication is that when $\Omegab_{1,} $ is not sparse enough (i.e.,  $s \gtrsim n/\log p $),  a detection of alternatives separated by a constant is not  guaranteed; that is, even deviation of non-vanishing magnitude cannot be detected.

The proof of Theorem \ref{THM:4} is formulated in a novel way.
For any point in the alternative hypothesis, we compute the $\chi^2$ distance between  that  alternative and a large collection of points in the null hypothesis. Whenever this distance is small, it indicates that the average rejection probability for that particular alternative is close to the average rejection probability for many of the nulls-- therefore indicating lack of power.
Although uniform non-testability explores a class of $\alpha$
 level tests (only), by inspecting the    proof of Theorem \ref{THM:4}, we see that the detection rate would not change even if we localize our problem and impose $ k\lesssim n/\log(p) $. Therefore, the rate is not really driven by some ultra-dense (and hence seemingly hopeless)  $k\gg n/\log(p) $  points in the parameter space.

Another novelty in the theoretical analysis is the construction of the prior. The prior used by   \cite{cai2017confidence}  can be adapted to the case of sparse $\Omegab_{1,}$ (instead of sparse $\gammab $ as in their paper). However, that adaption would assume $\Omegab_{1,-1} =0$ and thus would not be enough to show the \textit{uniformity} of non-testability. We compare this adaption with our construction in Appendix M. 

Next, we fine-tune the above result in search of a   parametric rate   of detection.
In view of that fact, we introduce a slightly weaker notion of essentially uniform non-testability.

\begin{definition}[Essentially uniform non-testability]
Consider the hypothesis testing problem of $H_{0}:\ \theta\in\Theta^{(1)}$
versus $H_{1}:\ \theta\in\Theta^{(2)}$. We say that $\Theta^{(1)}$
is asymptotically essentially uniformly non-testable against $\Theta^{(2)}$
at size $\alpha\in(0,1/2)$ if $\limsup_{n\rightarrow\infty} \sup_{\theta\in\Theta^{(2)}} \EE_{\theta}\psi\leq2\alpha$
for any test $\psi\in\Psi_{\alpha}(\Theta^{(1)})$.
\end{definition}

Essentially uniform non-testability implies
  \[
 \liminf_{n\rightarrow \infty} \biggl(   \alpha +\inf_{\psi\in\Psi_{\alpha}(\Theta^{(1)})}  \inf_{\theta \in \Theta^{(2)} } \EE_{\theta}(1-\psi)  \biggr) \geq 1-\alpha.
 \]
We note that this statement implies the following claim on the minimax total error probability (a notion discussed by \cite{ingster2010detection})
$$
 \liminf_{n\rightarrow \infty} \biggl(   \alpha +\inf_{\psi\in\Psi_{\alpha}(\Theta^{(1)})}  \sup_{\theta \in \Theta^{(2)} } \EE_{\theta}(1-\psi )  \biggr) \geq 1-\alpha.
$$

 We also denote
\begin{multline*}
\Theta_{\kappa }(s,\beta_{0}+h_{n})=\biggl\{  \theta\in\Theta(s, \beta_{0}+a):\  0\leq a \leq h_n, \| \betab\|_2\leq M_2, \\
M^{-1}\leq\lambda_{\min}(\Sigmab)\leq\lambda_{\max}(\Sigma)\leq  M,\  \kappa \leq\sigma\leq  M_{1} \biggl\}.
\end{multline*}

\begin{theorem} \label{THM:5}
Suppose that $sn^{-1}\log p\leq1/4$ and $2\leq s\leq p^{c}$ for
some constant $c<1/2$. Then for any constant $\kappa \in(0,M_{1}] $, we have that  for any $\beta_{0}$,
\[
\limsup_{n\rightarrow\infty}\sup_{\psi\in\Psi_{\alpha}(\Theta (s,\beta_{0}))}\sup_{\theta\in  \Theta_{\kappa }(s,\beta_{0}+h_n)}\EE_{\theta}\psi \leq 2\alpha,
\]
where $h_n=n^{-1/2}\tau $ and $\tau=\kappa \sqrt{M^{-1}\log(1+\alpha^{2})}$.
\end{theorem}

 This results implies that $ \Theta(s,\beta_{0})$ is essentially uniformly
non-testable against $\Theta_{\kappa}(s,\beta_{0}+n^{-1/2}\tau )$. This result confirms the intuition that parametric rate is a fundamental boundary for statistical inference, an insight from the classical results  of, for example, \cite{lehmann2006testing,van2000asymptotic}.
Let $c_0=\min\{\rho,\tau\}$. Then
\begin{multline*}
\Theta_{\zeta,\kappa }(s/2,\beta_{0}+c_0 (n^{-1/2} +sn^{-1}\log p)) \\
\subset \Theta_{\zeta,\kappa }(s/2,\beta_{0}+\rho sn^{-1}\log p ) \bigcap  \Theta_{\kappa }(s,\beta_{0}+\tau n^{-1/2} ).
\end{multline*}
Hence,   Theorems \ref{THM:4}  and \ref{THM:5} imply
\begin{equation}\label{eq: lower bound1}
\limsup_{n\rightarrow\infty}\sup_{\psi\in\Psi_{\alpha}(\Theta(s,\beta_{0}))}\sup_{\theta\in\Theta_{\zeta,\kappa }(s/2,\beta_{0}+c_0(n^{-1/2}+sn^{-1}\log p))}\EE_{\theta}\psi \leq  2 \alpha.
\end{equation}
Therefore, constructing a meaningful test with a detection rate smaller than  that of $n^{-1/2}+sn^{-1}\log p$ is indeed impossible.

\section{Upper bound} \label{sec:4}

In this section, we show that the lower limit of detection matches the upper limit of detection. In this section we focus our analysis on the  space  $\tilde{\Theta}(s) $. Formally, we define $\widetilde \Theta(s_{0})=\bigcup_{\beta_0\in \RR} \widetilde \Theta(s_{0},\beta_{0})$ and
\[
\widetilde \Theta(s_{0},\beta_{0})=\left\{ \theta=(\beta,\gammab,\Sigmab,\sigma) \in \widetilde \Theta: \beta=\beta_0, \ \|\Omegab_{1,}\|_{0}\leq s_{0}\right\}.
\]


We propose a  test that achieves the bounds of Section \ref{sec:3}. The newly proposed estimator $\hat \beta$ of $\beta$ utilizes the constants that define the parameter set of interest to us, $\tilde \Theta (s, \cdot)$, and is therefore   of pure theoretical interest. It
 is based on delicately  designed  high-dimensional estimators of the nuisance parameters: both of the model as well as that of the partial correlations of the features;  an $\ell_1$  consistent  in the big coordinates   while  $\ell_\infty$ consistent  in the small  coordinates.
Lastly, the new estimates are based on cross-fitting concepts enabling  adaptivity to the rates  of Section \ref{sec:3}.

We introduce notation that   helps with    our construction.
The constructed method will utilize a sample-splitting scheme. Let $b_{n}=\left\lfloor n/4\right\rfloor $. We consider four non-overlapping
subsets of the original sample $H_{1}=\{1,...,b_{n}\}$, $H_{2}=\{b_{n}+1,...,2b_{n}\}$,
$H_{3}=\{2b_{n}+1,...,3b_{n}\}$ and $H_{4}=\{3b_{n}+1,...,4b_{n}\}$.

Next, we observe that  the first row $\bOmega_{1,} $ takes the form $(1,-\pib^{\top})/\sigma_{\Vb}^{2} $, where $\pib$ and $\sigma_{\Vb}^{2}$  are from the regression 
\begin{equation}\label{eq:p}
\Zb=\Wb \pib + \bV,
\end{equation}
where the vector $\bV$ is independent of $\Wb$ with $\sigma_{\Vb}^2 = \EE(\Vb^\top \Vb)/n$. 
Moreover, observe that
\begin{equation}\label{eq:y}
y_i = \Wb_i^\top (\boldsymbol{\pi}\beta+\boldsymbol{\gamma}) + \eta_i
\end{equation}
for $\eta_i = \beta v_i + \varepsilon_i$.
Then, we notice that the parameter of interest, $\beta$, can be defined through a moment condition
\[
\EE[v_i y_i]= \beta \sigma_{\Vb}^2.
\]
Therefore, for a suitably chosen estimator ${\breve \pib}$ of $\pib$, let
  \[
  \hat{v}_{i}=Z_{i}-\mathbf{W}_{i}^{\top}{\breve \pib}
  \] denote the estimated residuals of the model \eqref{eq:p} and consider a natural estimator of $\beta$ arising from the above moment condition
\begin{equation}\label{eq: optimal estimator beta}
\hat{\beta}=\frac{ \sum_{i\in H_{4}}\hat{v}_{i}y_{i}}{ \sum_{i\in H_{4}}\hat{v}_{i}^{2}}.
\end{equation}
Observe that this estimator is computed on the last fold, $H_4$ of the data; the remaining three folds are used to construct the  estimator ${\breve \pib}$.
Note that the numerator in (\ref{eq: optimal estimator beta}) is estimating
$$
   \EE[v_iy_i] =\EE[(Z_i-\Wb_i^{\top}\pib)y_i]=\EE(Z_i y_i)-\pib^{\top}\xib , \qquad
    \xib = \EE[\Wb_iy_i].
$$
 Although the estimation of $\pib$ is a sparse high-dimensional regression problem, existing estimators, such as Lasso, Dantzig selector or their debiased version, do not possess the theoretical properties we need for inference on $\beta$. Therefore, we construct a new estimator that is suitable for the purpose of inference. This  new  projected de-biased estimator ${\breve \pib}$ of $\pib$  aims to balance the good properties of both Lasso as well as de-biased Lasso estimator; balancing $\ell_1$ with $\ell_\infty$ estimation quality.
\vskip 10pt

We use the second and fourth fold of the data to construct cross-validated de-biased estimator of  $\pib$ in the following way. On the second fold compute 
a simple $\ell_1$-regularized estimator
 $\hat{\pib}$,
$$
\hat \pib=\underset{\boldsymbol{q}\in\mathbb{R}^{p-1}}{\arg\min}\quad b_{n}^{-1}\sum_{i\in H_{2}}(Z_{i}-\mathbf{W}_{i}^{\top}\boldsymbol{q})^{2}+\lambda_{\boldsymbol{\pi}}\|\boldsymbol{q}\|_{1},
$$
with $\lambda_{\boldsymbol{\pi}}=24M\sqrt{b_{n}^{-1}\log p}$.

To shrink the bias in estimated large coefficients of $\pib$,
we define a cross-fitted  estimator  as
\[
{\tilde \pib}=\hat{\boldsymbol{\pi}}+b_{n}^{-1}\sum_{i\in H_{4}}\hat{\boldsymbol{\Omega}}_{\mathbf{W}}\mathbf{W}_{i}(Z_{i}-\mathbf{W}_{i}^{\top}\hat \pib),
\]
In the above, $\hat \bOmega_{\bW}$ is a carefully designed candidate estimate of $\bOmega_{\bW}=\bSigma_{\bW}^{-1}$, that utilizes model \eqref{eq:y} while   ensuring that $\hat \bOmega_{\bW}$ is close to $\bSigma_{\bW}^{-1}$.
We propose the following  cross-fitted spectral  estimate


\begin{eqnarray}
\hat{\boldsymbol{\Omega}}_{\mathbf{W}}
 & =  \underset{\boldsymbol{Q} \in \mathbb{R}^{(p-1)\times(p-1)} }{\arg\min} 
\lambda_{\max}(\boldsymbol{Q})
 & \label{eq: constraint OmegaW} \\[1.2 ex]
 &   \mbox{s.t.}   \qquad  \qquad  \qquad \bQ \qquad & =\bQ^\top  \nonumber \\[1.2 ex]
 &  \Bigl\| \bigl\{ \mathbb{I}_{p-1}-\hat \bSigma_{\bW}\boldsymbol{Q}\bigl\} \hat{\boldsymbol{\xi}}_{A} \Bigl \| _{\infty} 
 & \leq
\lambda_{\bOmega}\nonumber \\ [1.2 ex]
 &    \hat{\boldsymbol{\xi}}_{A}^{\top} \bigl\{ \boldsymbol{Q}\phantom{(}\hat \bSigma_{\bW}\phantom{(}   \boldsymbol{Q} \bigl\} \hat{\boldsymbol{\xi}}_{A}\
 & \leq
\eta_{\bOmega},\nonumber
\end{eqnarray}

for $\hat \bSigma_{\bW}=b_{n}^{-1}\sum_{i\in H_{4}}\mathbf{W}_{i}\mathbf{W}_{i}^{\top}$ as well as
$$\lambda_{\bOmega}=24\sqrt{b_{n}^{-1}\log p}M^{3}M_{2} , \qquad \eta_{\bOmega}=32M^{5}M_{2}^{2}.$$ 
In the above  $\hat \xib$, a thresholded, marginal,  estimate, and $\hat \Sigmab_{\Wb}$ are computed on different folds of the data.  Correlation estimate, $\hat \xib_A$, is defined as a sparse vector containing  the  top largest elements of   the empirical inner product $ \langle\Wb, y\rangle$. 
Set $A$ denotes the largest elements,
\begin{equation}
A=\left\{ j \in [p]:\ |\tilde{\boldsymbol{\xi}}_{j}|> \tau_n\right\}, \qquad  \tilde{\boldsymbol{\xi}}=b_{n}^{-1}\sum_{i\in H_{1}}\mathbf{W}_{i}y_{i}. \qquad \label{eq: def set A}
\end{equation}
Here, $\tau_n=4Mb_{n}^{-1}\sqrt{n(\log p)(M_{1}^{2}+M_{2}^{2})}.$
Then, $\{\hat \xib_A\}_j =0$ for $j \not \in A$ and $b_{n}^{-1}\sum_{i\in H_{3}}\mathbf{W}_{ij}y_{i}$ otherwise.

\vskip 20pt

Finally, we construct the following  projected de-biased estimator
\begin{eqnarray}
{\breve \pib} & =  \qquad  \arg\min_{\boldsymbol{q}\in\mathbb{R}^{p-1}}\|\boldsymbol{q}\|_{1} & \label{eq: pi half estimator}\\[1.2 ex]
 &   \mbox{s.t.} \qquad  \Bigl|\hat{\boldsymbol{\xi}}_{A}^{\top} ( \boldsymbol{q}_{A}- {{\tilde \pib}_A} )
 \Bigl|
 & \leq\eta_{\boldsymbol{\pi}}\nonumber \\[1.2 ex]
 &  \ \quad  \qquad  \Bigl\| b_{n}^{-1}\sum_{i\in H_{4}}\mathbf{W}_{i}(Z_{i}-\mathbf{W}_{i}^{\top}\boldsymbol{q})\Bigl\|_{\infty}
 & \leq\lambda_{\boldsymbol{\pi}}/4\nonumber \\ 
 &     \ \quad \qquad b_{n}^{-1}\sum_{i\in H_{4}}(Z_{i}-\mathbf{W}_{i}^{\top}\boldsymbol{q})^{2}
 & \geq\frac{1}{2M},\nonumber
\end{eqnarray}
where the last three lines define the constraint set  and
where the tuning parameter $\eta_{\boldsymbol{\pi}}$ satisfies
$$\eta_{\boldsymbol{\pi}}=6408\sqrt{b_{n}^{-1}\log p}M^{4}M_{2}s\lambda_{\boldsymbol{\pi}}+8b_{n}^{-1/2}M^{2}M_{2}\sqrt{M\log(100/\alpha)}.$$

The estimator ${\breve \pib} $ is carefully crafted in order to achieve  the desirable bias-variance tradeoff: it has small bias for entries corresponding to ``large'' elements of  $\boldsymbol{\xi} $ and has small variance on other entries. Here, sample splitting is helpful in providing several independence structures that we need for the theoretical analysis; for example, the set $A$ that defines ``large'' and ``small'' components needs to be independent of the subsequent constructions. As a result,  ${\breve \pib} $ is quite different from the debiased estimator ${\tilde \pib} $ and these two estimators only behave similarly on large elements, i.e., $ \Bigl|\hat{\boldsymbol{\xi}}_{A}^{\top} ({\breve \pib}_{A}- {{\tilde \pib}_A} )\Bigl|$ is small. 

\vskip 10pt 

We propose the following test
$$
\psi_*=\mathbf{1} \left\{ |\hat{\beta}- \beta_0|>c_n \right\},
$$
where $\hat \beta$ is defined in \eqref{eq: optimal estimator beta} and
\begin{multline}\label{eq: CI half width}
c_{n}=2M\biggl(10b_{n}^{-1/2}\sqrt{M\left(4M_{2}^{2}M^{3}+M_{1}^{2}\right)\log(100/\alpha)} +\\
+34M(1+M_{2})\lambda_{\boldsymbol{\pi}}^{2}s+1608b_{n}^{-1}M^{2}\sqrt{n(\log p)(M_{1}^{2}+M_{2}^{2})}\lambda_{\boldsymbol{\pi}}s+2\eta_{\boldsymbol{\pi}}\biggr).
\end{multline}

We now show that even on the larger parameter $\widetilde{\Theta}(s) $ (compared to $\Theta(s) $), the test $\psi_*$ is indeed valid and has the optimal detection rate.

\vskip 20pt 

\begin{theorem}
\label{THM:8}Suppose that $p\geq\max\left\{ 2\left(1+1764M^{2}\right)s,\ 360/\alpha\right\} $
and
\begin{align*}
n\geq\max\left\{ 4+784\log p,\ (5067+220M^{2})\log(100/\alpha), \right.
\\
\left. \ 4+4054\left[1+1764M^{2}\right]s\log(16ep)\right\}.
\end{align*}
 Then $\psi_*\in \Psi_{\alpha}(\widetilde{\Theta}(s,\beta_0)) $, i.e., $$ \sup_{\theta\in\widetilde{\Theta}(s,\beta_0)}\EE_{\theta}\psi_*\leq \alpha. $$
 Moreover, $c_n \asymp sn^{-1}\log p +n^{-1/2} $ and
$$ \inf_{\theta\in\widetilde{\Theta}(s,\beta_{0}+3 c_{n})}\EE_{\theta}\psi_* \geq 1-\alpha.  $$
\end{theorem}

Theorem \ref{THM:8} demonstrates that lower bound in (\ref{eq: lower bound1})   is achievable by a test $\psi_*$ as defined above. Notice that requirement on $n$ and $p$ in Theorem \ref{THM:8} is mild; the key requirement is $s \lesssim n/ \log p$.
 The proposed uniform non-testability results indicate the new detection boundary of $ n^{-1/2}+sn^{-1}\log p$. Theorem  \ref{THM:8} establishes that  deviations of magnitude $3c_n\asymp n^{-1/2}+sn^{-1}\log p $ are uniformly testable over $\widetilde{\Theta}(s)$, whereas results in Section \ref{sec:3} imply that even on the smaller $\Theta(s)$, deviations smaller than this rate are (essentially) uniformly non-testable.

Moreover, the parametric rate can be attained whenever $s \lesssim \sqrt{n}/\log p $.  The case of $\sqrt{ n}/\log p\ll s \ll n/\log p $ is more difficult and our proposed test still achieves the optimal rate. Note our test $\psi_*$ depends on the knowledge of $s$. It turns out that the uniform non-testability results in Section \ref{sec:3} imply that such knowledge is required to achieve the minimax rate,  indicating lack of adaptivity to the precision matrix sparsity. We make this argument precise in Section \ref{sec:ci} and in more generality in Section \ref{sec:adaptive}.

\section{Connections to minimax rates and confidence intervals}\label{sec:5}
 In this section we highlight the implication of the obtained results on the minimax theory and adaptivity.

\subsection{Minimax rates}

The (essential) uniform non-testability leads to the following minimax lower bound.

\begin{corollary}\label{THM:1}

If $sn^{-1}\log p\leq 1/4$ and  $2\leq s\leq p^{c}$ for some constant $c<1/2$, then there exists a constant $h_{0}>0$ such that for any $\beta_{0}$
\[
\limsup_{n\rightarrow\infty}\sup_{\psi\in\Psi_{\alpha}(\widetilde{\Theta}(s,\beta_{0}))}\inf_{\theta\in\widetilde{\Theta}(s,\beta_{0}+h_{n})}\EE_{\theta}\psi \leq 2\alpha,
\]
with $h_{n}=h_{0}(n^{-1/2}+sn^{-1}\log p)$.
\end{corollary}


Observe that Corollary \ref{THM:1} establishes a minimax claim that spans the space of $\widetilde \Theta(s,\beta_0 +h)$; it does not impose $\Sigmab_{(-1,),(-1)} =\II_{p-1}$ and does not restrict $k$ (the sparsity of $\betab$).
  Therefore, Corollary \ref{THM:1} directly refines the existing results on minimax testing, which routinely assume $k \lesssim n/\log p $; see  \cite{cai2017confidence,robins2006adaptive,cai2004adaptation,cai2006adaptive,hoffmann2011adaptive,genovese2008adaptive,nickl2013confidence,cai2016accuracy}.
  Corollary \ref{THM:1} establishes a lower bound for the minimax detection rate of the null $H_0: \beta =\beta_0$ against the alternative
$$H_1:\beta = \beta_0 + h_0 (n^{-1/2}+sn^{-1}\log p)$$ regardless of the sparsity of the nuisance parameter $\gammab$ in the regression model \eqref{eq:model2}. As such this result is the first that derives the lower bound for the detection rate under fairly general model setting and in particular not requiring a model to be sparse. Theorem \ref{THM:8} entails that sparsity of the first row of the precision matrix (alone) is sufficient for minimax  inference  (per Corollary \ref{THM:1}), and the sparsity on regression coefficients is not necessary.


When $s \geq c_0 n/\log p$, a direct consequence of Corollary \ref{THM:1} is that it is impossible to distinguish $H_0:\ \beta=\beta_0$ and $H_1:\ \beta=\beta_0+c_0 h_0 $ in a minimax sense; in other words, there is no power even against fixed alternatives.
Whenever  $\Omegab_{1,}$ is sparse in that $\|\Omegab_{1,}\|_{0}=o(n/\log p)$,
the lower bound for  minimax detection rate  is of the order
$$n^{-1/2}+\|\Omegab_{1,}\|_{0}n^{-1}\log p.$$
However,  when $\Omegab_{1,}$ is ultra sparse in that $\|\Omegab_{1,}\|_{0}=o(\sqrt{n}/\log p)$,
then this lower bound is the parametric rate, i.e.
$$1/\sqrt{n}.$$

\subsection{Confidence intervals} \label{sec:ci}
The theoretical results in Sections \ref{sec:3} and \ref{sec:4} also imply that the expected length of confidence intervals cannot be adapted to $s$ if $s \gg \sqrt{n}/\log p $.

 We denote by $\mathcal{C}_{\alpha} (\Theta_1)$   the set of all  $(1-\alpha)$ level confidence intervals  for $\beta$ over the parameter space  $\Theta_1 $ constructed from the observed data $\Db$:
\begin{equation}
\mathcal{C}_{\alpha} (\Theta_1) = \left\{ [l(\Db),u(\Db)] :  \inf_{\theta \in \Theta_1} \PP_\theta (l(\Db) \leq \beta \leq u(\Db)) \geq 1-\alpha \right\}.
\end{equation}


The construction in Section \ref{sec:4} yields  the following confidence interval
\[
\mathcal{CI}_{*}=\left[\hat{\beta}-c_{n},\ \hat{\beta}+c_{n} \right],
\]
where $c_n$ is defined in (\ref{eq: CI half width}). Theorem \ref{THM:8} implies that

 $$\inf_{\theta\in\widetilde{\Theta}(s)} \mathbb{P}_{\theta} \left (\beta\in\mathcal{CI}_{*} \right)\geq1-\alpha.$$

 Since the diameter of the confidence set, $\mbox{diam}(\mathcal{CI}_{*})=2c_n$, Theorem \ref{THM:8} states the rate for $c_n$ and thus implies the following minimax upper bound for the expected length of the confidence intervals:
\begin{equation}\label{eq: CI upper bound minimax}
\inf_{\CIcal_\alpha\in \Ccal_{\alpha}(\widetilde{\Theta}(s))}
\sup _{\theta \in \widetilde{\Theta} (s) } \EE_\theta  \mbox{diam}(\mathcal{CI}_{\alpha}) \lesssim  n^{-1/2}+sn^{-1}\log p,
\end{equation}
where  $\mbox{diam}(CI)$ denotes the length of $CI$.

Since confidence intervals can be used to construct tests, minimax results on tests have implications for the minimax length of confidence intervals.

\begin{corollary}\label{COR:2}
Suppose that $sn^{-1}\log p\leq1/4$ and $2\leq s\leq p^{c}$ for
some constant $c<1/2$. Then for any $\alpha\in(0,1/3)$, we have
\[
\inf_{\CIcal_\alpha\in \Ccal_{\alpha}(\Theta(s))}
\sup _{\theta \in \Theta (s) } \EE_\theta  \mbox{\rm diam}(\mathcal{CI}_{\alpha}) \gtrsim n^{-1/2}+sn^{-1}\log p.
\]
\end{corollary}

See Theorem 1 and Equation (3.14) of \cite{cai2017confidence} for quantification of optimal confidence interval width for sparse or moderately sparse models i.e., $k \lesssim n/\log p$. Complementary, we  allow  for non-sparse vectors $\betab$, i.e., $k\lesssim p$.


Moreover, since $\Theta(s) \subset \wtTheta(s)$ and $\mathcal{C}_{\alpha}(\wtTheta(s))\subset \mathcal{C}_{\alpha}(\Theta(s)) $, we have
\begin{align*}
\inf_{\CIcal_{\alpha}\in \Ccal_\alpha(\wtTheta(s))} \sup _{\theta \in \wtTheta(s)} \EE_\theta  \mbox{diam}(\mathcal{CI}_{\alpha}) & \geq \inf_{\CIcal_{\alpha}\in \Ccal_\alpha(\wtTheta(s))} \sup _{\theta \in  \Theta(s)} \EE_\theta  \mbox{diam}(\mathcal{CI}_{\alpha})
 \\
 & \geq \inf_{\CIcal_{\alpha}\in \Ccal_\alpha(\Theta(s))} \sup _{\theta \in \Theta(s)} \EE_\theta  \mbox{diam}(\mathcal{CI}_{\alpha}).
\end{align*}
Therefore, Corollary \ref{COR:2} still holds if we replace $\Theta(s) $ by $\wtTheta(s) $. Combining this with (\ref{eq: CI upper bound minimax}), we obtain the minimax optimal rate for the expected length of confidence intervals over $\Theta(s)$ and $\widetilde{\Theta}(s) $:
$$\inf_{\CIcal_\alpha\in \Ccal_{\alpha}(\Theta(s))}
\sup _{\theta \in \Theta (s) } \EE_\theta  \mbox{diam}(\mathcal{CI}_{\alpha}) \asymp  n^{-1/2}+sn^{-1}\log p $$
and
$$\inf_{\CIcal_\alpha\in \Ccal_{\alpha}(\widetilde{\Theta}(s))}
\sup _{\theta \in \widetilde{\Theta} (s) } \EE_\theta  \mbox{diam}(\mathcal{CI}_{\alpha}) \asymp  n^{-1/2}+sn^{-1}\log p. $$

In Theorem \ref{THM:8}, we have constructed a minimax rate-optimal confidence interval for $\beta$ in the case that the sparsity $s$ is assumed to be known. A significant drawback of the construction is that it requires prior knowledge of $s$, which
is typically unavailable in practice.  Is it
possible to construct adaptive confidence intervals that have the guaranteed
coverage and automatically adjust the length to $s$?
In other words, does there exist  a confidence
interval in $ \Ccal_{\alpha}(\wtTheta(s))$ that  has expected length of
the order $n^{-1/2}+s_{1}n^{-1}\log p$ over all $\wtTheta(s_{1})$  and any $s_{1}\ll s$?  One consequence of the uniform non-testability result is that such adaptivity is not possible.


\begin{theorem}\label{THM:9}
Suppose that $sn^{-1}\log p\leq1/4$ and $2\leq s\leq p^{c}$ for
some constant $c<1/2$. Then for any $\alpha\in(0,1/4)$ and $s_1 \leq s/2$, we have
\[
\inf _{\CIcal_\alpha \in \Ccal_\alpha( \Theta(s))} \sup _{\theta \in \Theta(s_1)} \EE_\theta  \mbox{\rm diam}(\mathcal{CI}_{\alpha})
\asymp n^{-1/2}+s n^{-1}\log p.
\]
\end{theorem}

Even for $s_1 \ll s $, the optimal rate over $\Theta(s_1) $   for all confidence intervals that do not take into account knowledge of $s_1 $, is larger than that with the knowledge of $s_1$.
Therefore, Theorem \ref{THM:9} implies that for dense models ($k\lesssim p$), adaptivity with respect to $s$ is in general not possible if $\Omegab_{1,}$  is in at the least  moderately sparse regime ($\sqrt{ n}/\log p\ll s \lesssim n/\log p $).



\subsection{Characterization of uniform non-testability}\label{sec:adaptive}
Here, we showcase that  uniform non-testability   is equivalent to the lack of adaptivity in all subsets of the parameter space.


Let $\Theta$ be a parameter space for a general model. We are interested
in confidence intervals of $g(\theta)$, where $g$ is an arbitrary functional of the whole parameter space, characterized by
$\theta$. For any $\Theta_{1}\subseteq\Theta$, define the set of
valid confidence intervals on $\Theta_{1}$:
\[
\mathcal{C}_{\alpha}(\Theta_{1})=\left\{ CI:\ \inf_{\theta\in\Theta_{1}}\mathbb{P}_{\theta}(g(\theta)\in CI)\geq1-\alpha\right\} .
\]

For $\Theta_{1}\subseteq\Theta$, the minimax rate over $\Theta_{1}$
 confidence intervals valid over $\Theta$ can be defined as
\begin{equation}\label{eq:Ldef}
L(\Theta_{1},\Theta)=\inf_{CI\in\mathcal{C}_{\alpha}(\Theta)}\sup_{\theta\in\Theta_{1}}\mathbb{E}_{\theta}\mbox{diam}(CI).
\end{equation}

For $\Theta_{1}\subseteq\Theta$, we say that there is no adaptivity
between $\Theta$ and $\Theta_{1}$ if
\[
L(\Theta_{1},\Theta)\asymp L(\Theta,\Theta).
\]
In other words, if we use a confidence interval that is valid over
the larger set $\Theta$, then even on the smaller set $\Theta_{1}$,
the length of the confidence interval has no improvement.
For confidence intervals, we say that points in $\Theta$ are uniformly
non-testable if
\begin{equation}
\inf_{CI\in\mathcal{C}_{\alpha}(\Theta)}\sup_{\theta\in\Theta}\mathbb{E}_{\theta}\mbox{diam}(CI)\asymp\inf_{CI\in\mathcal{C}_{\alpha}(\Theta)}\inf_{\theta\in\Theta}\mathbb{E}_{\theta}\mbox{diam}(CI).\label{eq: unif non-test CI}
\end{equation}

In other words, the minimax confidence intervals have the same order
of magnitude in terms of length for all the points in $\Theta$. The
following result establishes the link between uniform non-testability
and adaptivity.
\begin{theorem}
	\label{THM:10}
The uniform non-testability
$$\inf_{CI\in\mathcal{C}_{\alpha}(\Theta)}\sup_{\theta\in\Theta}\mathbb{E}_{\theta}\mbox{\rm diam}(CI)\asymp\inf_{CI\in\mathcal{C}_{\alpha}(\Theta)}\inf_{\theta\in\Theta}\mathbb{E}_{\theta}\mbox{\rm diam}(CI)$$
	if and only if there exists a constant $c>0$ such that
	$$cL(\Theta,\Theta)\leq L(\Theta_{1},\Theta)\leq L(\Theta,\Theta)$$
	for any subset $\Theta_{1}\subseteq\Theta$.
\end{theorem}

Theorem \ref{THM:10} establishes
  that uniform non-testability simply means
that there is no adaptivity between $\Theta$ and any subset of $\Theta$.
Hence, uniform non-testability provides a way of looking at adaptivity.
Intuitively, adaptivity means that a procedure can automatically adapt
its efficiency to the parameter. Since uniform non-testability means
that the minimax optimal procedure has the same efficiency at each
point in the parameter space, this rules out the possibility that
the efficiency of the optimal procedure can change from parameter
to parameter.


In the above setup, consider the testing problem
\begin{equation}\label{eq:nulll}
H_{0}:\ g(\theta)=\tau, \qquad  \mbox{vs}  \qquad H_{1}:\ g(\theta)=\tau+c_{1}h_{n}.
\end{equation}
For any $\tau\in\mathbb{R}$, $\Theta(\tau)=\{\theta\in\Theta:\ g(\theta)=\tau\}$,
i.e., $\Theta(\tau)$ is the set of parameters $\theta$ satisfying
the null hypothesis $H_{0}:\ g(\theta)=\tau$. For any $\Theta_{1}\subseteq\Theta$,
let the set of valid tests of size $\alpha$ over $\Theta_{1}$ be
denoted by $\Psi_\alpha(\Theta_{1})$, i.e.,
\[
\Psi_\alpha(\Theta_{1})=\{\psi:\ \sup_{\theta\in\Theta_{1}}\mathbb{E}_{\theta}\psi\leq\alpha\}.
\]
Next, we showcase that the result of Theorem \ref{THM:10} applies to the hypothesis testing problems studied in  Section \ref{sec:3}.
\begin{corollary}
	\label{COR:9}Suppose that there exist constants $c_{1},c_{2}>0$
	and a confidence interval $CI_{*}\in\mathcal{C}_{\alpha}(\Theta)$
	such that
	\begin{itemize}
	\item[(1)] for any $\tau\in\mathbb{R}$
	
	$ \hskip 100pt \sup_{\psi\in\Psi_\alpha(\Theta(\tau))}\sup_{\theta\in\Theta(\tau+c_{1}h_{n})}\mathbb{E}_{\theta}\psi\leq2\alpha,$
	\item[(2)] \
	
	$ \hskip 100pt\sup_{\theta\in\Theta}\mathbb{E}_{\theta} \mbox{\rm diam}(CI_{*})\le c_{2}h_{n}$.
	\end{itemize}
	
	 	Then,
	\[
	\inf_{CI\in\mathcal{C}_{\alpha}(\Theta)}\sup_{\theta\in\Theta}\mathbb{E}_{\theta}\mbox{\rm diam}(CI)\asymp\inf_{CI\in\mathcal{C}_{\alpha}(\Theta)}\inf_{\theta\in\Theta}\mathbb{E}_{\theta}\mbox{\rm diam}(CI)\asymp h_{n}.
	\]
\end{corollary}

Condition (1)  states that any
$\alpha$  level test about   \eqref{eq:nulll}
has power at most $2\alpha$ whereas Condition (2)
assumes  a valid confidence interval  with expected
length  of the order $h_{n}$, therefore  $h_{n}$ is a
detection boundary.   Corollary \ref{COR:9} then states
that we have uniform non-testability in the sense of (\ref{eq: unif non-test CI})
and the optimal rate is $h_{n}$ which, by Sections \ref{sec:3} and
\ref{sec:4} is  $h_{n}=n^{-1/2}+sn^{-1}\log p$.

 Theorem \ref{THM:10} suggests  a much broader implication. Since the uniform non-testability implies lack of adaptivity with respect to any subset of the parameter space, our result indicates that it is impossible for a confidence interval to automatically exploit other structures of the model. In particular, if a confidence interval is valid on $\Theta$, then it will have the same rate even at points with special structures, e.g.,  sparsity,  homogeneity, etc. Hence, our result not only states that there is no adaptivity with respect to $s$, we show that there cannot be any adaptivity with respect to any structure.

 \section{Impact of an increasing \texorpdfstring{$\|{\boldsymbol{\beta}}\|_2$}{beta}}\label{sec:7}

We now discuss the case in which the $\ell_2$-norm of $\betab$  for the model (\ref{eq:model2})  is allowed to grow. To explicitly write out the dependence on $\|\betab\|_2$ i.e., $M_2$, we introduce the notation
\begin{multline*}
\widetilde{\Theta}_{M_{1},M_{2}}(s)=\biggl\{ \theta=(\beta,\boldsymbol{\gamma},\boldsymbol{\Sigma},\sigma):M^{-1}\leq\lambda_{\min}(\boldsymbol{\Sigma})\leq\lambda_{\max}(\boldsymbol{\Sigma})\leq M,\\  \|\Omegab_{1,}\|_{0}\leq s,  0\leq\sigma\leq M_{1},\ \|\boldsymbol{\beta}\|_{2}\leq M_{2}\biggl\}  , \qquad
\end{multline*}
where $M>1$ is a constant. Now define the minimax length
\[
{\mathbb A}(s,M_{1},M_{2})=L \left(\widetilde{\Theta}_{M_{1},M_{2}}(s),\widetilde{\Theta}_{M_{1},M_{2}}(s)\right),
\]
where $L(\cdot,\cdot)$ is defined in \eqref{eq:Ldef}.
The following result states a scaling property that allows us to derive the minimax result for dense models with growing $\ell_2$-norm of the parameter.
\begin{theorem}
	\label{THM:12}For any constants $Q,M_{1},M_{2}>0$, \[{\mathbb A}(s,QM_{1},QM_{2})=Q{\mathbb A}(s,M_{1},M_{2}).\]
\end{theorem}


By Theorem \ref{THM:12}, it suffices to derive $\mathbb{A}(s,M_{0},1)$
for all $M_{0}>0$. This is because ${\mathbb A}(s,M_{1},M_{2})=M_{2}\mathbb{A}(s,M_{0},1)$
with $M_{0}=M_{1}M_{2}^{-1}$.

For that end we consider a specific asymptotic regime where $M_2$ is considered fixed while $M_0$ is allowed to grow to infinity or to shrink to zero. Hence, for
  an arbitrary constant $C>0$, in the duration of this section,  we assume that $M_2>C$.

  Upper bound is obtain as a corollary of    Theorem \ref{THM:8}, from which we can easily conclude
$$
{\mathbb A}(s,M_{0},1)\leq C_1 (M_{0}+1)(n^{-1/2}+sn^{-1}\log p),
$$
where $C_1>0 $ is a constant independent of $M_1,M_2, n,s$ or $p$.

 To establish a lower bound, we establish the following result.

\begin{theorem}
	\label{THM:11}Assume that $p\geq 2n+1$. If $\alpha\in(0,1/3)$, then there exists a
	constant $C>0$, depending only on $\alpha$, such that
	$$\mathbb{A}(1,0,1)\geq C n^{-1/2}.$$
\end{theorem}

Theorem \ref{THM:11} considers a particularly simple setting where the model has no noise and  the sparsity of the precision matrix, $s=1$.
 Even in this simple case, Theorem \ref{THM:11} suggests the following:
  Even if the noise level $\sigma$ is zero, perfect  recovery of $\beta$ is not possible
 as long as the vector $\gammab$ is allowed to be non-sparse with bounded $\ell_2$-norm.

First implication of this result is that imposing bounded $\ell_2$-norm is weaker than imposing sparsity. When the model parameter $\gammab$ is assumed to be sparse $\|\gammab\|_0\lesssim n/\log p$ and the noise level is zero, one can invoke classical results (e.g., \citet{bickel2009simultaneous,raskutti2011minimax}) and obtain that exact recovery of the model parameter is achievable.
 However, Theorem \ref{THM:11} says that no estimator can   exactly recover dense signals that are only known to have bounded $\ell_2$-norm.
 This is true even if the covariance matrix of the design is known to be diagonal; in fact, by inspecting the proof, the same result holds even if the covariance matrix is known to be $\II_{p}$. Therefore, the difficulty of identifying dense signals is quite fundamental (even if their $\ell_2$-norm is bounded) and  is not due to noise in the response or to unknown distribution of the design.

Second implication of Theorem \ref{THM:11} is   a lower bound  of ${\mathbb A}(s,M_{0},1)$. Notice that for any non-singular covariance matrix, its inverse always has non-zero diagonal entries, which means $s\geq 1$. Hence,  Theorem \ref{THM:11} implies that
$$
{\mathbb A}(s,M_{0},1)\geq {\mathbb A}(1,0,1) \geq Cn^{-1/2}.
$$
To sum up the above bounds, we have
\begin{corollary}\label{THM:13}
Suppose that $p\geq 2n+1$, $sn^{-1}\log p\leq1/4$ and $1\leq s\leq p^{c}$ for
some constant $c<1/2$.
Then for any $\alpha\in(0,1/4)$ and $M_2 \geq C$, we have
\[
C_1 M_2 n^{-1/2} +C_1 sn^{-1}\log p \leq  {\mathbb A}(s,M_{1},M_{2}) \leq C_2 M_2 (n^{-1/2}+sn^{-1}\log p),
\]
where $C_1,C_2>0$ are constants depending only on $M, M_1,\alpha,C,c$.
\end{corollary}

Corollary \ref{THM:13} outlines a unique phenomenon  for dense high-dimensional models. Since efficiency for testing dense models depends on the $\ell_2$-norm of the model parameter, consistency is impossible if this magnitude in $\ell_2$-norm is  of the order larger than $\sqrt{n}$.
In contrast, for models with sparse parameters, results in \citep{cai2017confidence,javanmard2015biasing} show that   $\ell_2$-norm requirements  are not required.

An interesting, and yet challenging question arises from the above study: What is the exact minimax lower bound   as a function of the $\ell_2$-norm for  high-dimensional and dense models?
 We leave this investigation to future research.


\section{Discussion}

This paper establishes theoretical results for hypothesis testing problems in high-dimensional linear models. Our work pushes the frontier of high-dimensional inference by allowing the model sparsity to be arbitrary. 
 We derive the optimal detection rates and show that the accuracy of statistical inference without imposing model sparsity depends on the ability to decorrelate the features. The sparsity of the first row of the precision matrix controls the optimal detection rate; for sparse enough precision matrices, the parametric rate can be achieved. These results also provide additional insights into the adaptivity of optimal inference.

 The theoretical development in this paper has potential implications beyond minimax detection rates. In particular, we show that the detection rate for {\it every} point in the alternative is the same, and thus the derived detection rate is uniform over the alternative, which indicates that a simple coin toss is a uniformly most powerful test asymptotically.
 For this reason, the detection rates established in this paper are driven by the fundamental difficulty that cannot be adequately described under the general minimax framework.

Some important extensions and refinements are left open.
Our current results only
provide confidence intervals and testability results regarding univariate parameters; extending our theory to the setting of global
testing and especially multivariate testing, seems like a promising avenue for further work.
 Another challenge is that many new hypothesis testing problems have complex structures and some even non-convex boundaries. A systematic approach to studying such problems would improve and extend the current scope of inferential theoretical results.  In general, work can be done to identify a subset of points for which attainable and optimal tests can be developed, in turn, paving the way for new inferential methods.

\vskip 20pt
{\bf {\sc  Supplement} } contains the detailed proofs of all auxiliary lemmas as well as details of the proofs of Theorems \ref{THM: THM KNOWN SIGMA}, \ref{THM:8}, \ref{THM:9}, \ref{THM:10} and \ref{THM:12} as well as Corollaries \ref{THM:1}, \ref{COR:2} and \ref{COR:9}. Below we present proofs of Theorems \ref{THM:4}, \ref{THM:5} and \ref{THM:11}.


\vskip 20pt
 \appendix{}

\section{Proof of Theorem \ref{THM:4}} \label{sec: proof of thm 4}

Proof of Theorem \ref{THM:4} has been split into a sequence of smaller results.
First we present some notation, then   auxiliary  Lemmas \ref{lem:1}-\ref{lem:6}  that are useful in the proof of Theorem \ref{THM:4} and lastly the proof of the result itself.

\subsection{Notations}

In the rest of Appendix \ref{sec: proof of thm 4}, we introduce the following notation.
We utilize Lemma \ref{lem:6}  below to pinpoint the structure of the covariance matrices $\Sigmab$ that are of interest to us, i.e., the lower-right corner is equal to $\mathbb{I}_{p-1}$.

 Namely, we show that
  for any point $\theta=(\beta,\gammab,\Sigmab,\sigma) \in \Theta $, we can write $\Sigma$ as
\[
\Sigmab=\begin{pmatrix}\pib^{\top}\pib+\sigma_{\Vb}^{2} & \pib^{\top}\\
\pib & \mathbb{I}_{p-1}
\end{pmatrix},
\]
where $\pib$ is a suitably chosen vector  and $\sigma_{\Vb}^2$ is a suitably chosen constant that is positive.
This is equivalent to working with the vector $\pib \in \RR^{p-1}$  from the following regression,
\[
\Zb=   \Wb\pib + \bV
\]
for a vector of residuals $\bV \in \RR^n$. Coincidently, $\sigma_{\bV}$ will be the standard deviation of the residuals $\bV$.
Recall that $\EE_{\theta} [\Wb^\top \Wb] /n=\mathbb{I}_{p-1} $ for $\theta \in \Theta $.
We also define a matrix $L_{\theta}$ as follows
\begin{equation}
L_{\theta}=L(\theta)=\begin{pmatrix} \II_{p-1} & 0 & 0\\
\pib^{\top} & \sigma_{\Vb} & 0\\
(\pib\beta+\gammab)^{\top} & \beta\sigma_{\Vb} & \sigma
\end{pmatrix}.\label{eq: def L}
\end{equation}

From Lemma \ref{lem:6} below we know that the space of correlation matrices is spanned by the collection of $\Sigmab$'s as described above.
Notice that under $\PP_{\theta}$,   vector $(\Wb_i^\top,Z_i,y_{i})^{\top}\in\mathbb{R}^{p+1}$
has gaussian  distribution   $\Ncal(0,L_{\theta}L_{\theta}^{\top})$.


The plan of the proof  proceeds as follows.
 We pick an arbitrary test $\psi_{*}\in\Psi_{\alpha}(\Theta(s,\beta_{0}))$ and an arbitrary point in the alternative
 \begin{equation}\label{eq: appendix 1}
 \theta_{*}=(\beta_{*},\gammab_{*},\Sigmab_{*},\sigma_{\varepsilon,*})\in\Theta_{\zeta,\kappa}(s/2,\beta_{0}+h_{n}),
 \end{equation}
 where $h_n=\rho sn^{-1}\log p $ and $$\Sigmab_{*}=\begin{pmatrix}\pib_{*}^{\top}\pib_{*}+\sigma_{\Vb,*}^{2} & \pib_{*}^{\top}\\
\pib_{*} & \II_{p-1}
\end{pmatrix}.$$

We then construct points based on $\theta_*$ according to Definition \ref{def: prior dist step 1}  below.  Observe that these points are chosen to be dependent on the  alternative.
 Then we   proceed to show that (1) these points are in the null space $\Theta(s,\beta_0)$ and (2) the average $\chi^2$-distance between these points and $\theta_*$ is small. Therefore, the power of $\psi_*$ against $\theta_*$ is close to the average power against these points. Since these points are in the null space, the power against them is at most equal to the nominal level. As a result, the power against $\theta_*$ is also close to the nominal level.


In the rest of Appendix \ref{sec: proof of thm 4}, we denote $m=\left\lfloor s/2\right\rfloor$   and  define the set of all $m$-sparse vectors with entries taking values in $\{0,1\}$ as $\mathcal{M}$, i.e.,
\[
\mathcal{M}=\left\{ \delta\in\{0,1\}^{p-1}:\ \|\delta\|_{0}=m\right\}
\]
and let $N$ denote the cardinality of $\mathcal{M}$. Clearly, $N=\begin{pmatrix}p-1\\
m
\end{pmatrix}$. We list $\mathcal{M}$ as $\mathcal{M} =\{\deltab_{(1)},...,\deltab_{(N)}\}$, i.e., $\deltab_{(j)}$ denotes the element $j$ of the set $\mathcal{M}$.

\begin{definition}
\label{def: prior dist step 1} Given $\theta_*\in \Theta_{\zeta,\kappa}(s/2,\beta_{0}+h_{n})$ as in (\ref{eq: appendix 1}), let $0\leq d\leq \rho$ be such that $\beta_*=\beta_0+h $ with $h=dsn^{-1}\log p $. Let $r=\sigma_{\Vb,*} /\sigma_{\varepsilon,*}>0$.  For $j\in\{1,...,N\}$, define
$$\theta_{j}=(\beta_{0},\gammab_{(j)},\Sigmab_{(j)},\sigma_{\varepsilon,0})$$
with
\begin{align*}
\beta_{0}&=\beta_{*}-h\\
\gammab_{(j)}&=\gammab_{*}+h\pib_{(j)}+r(1-h)\sigma_{\varepsilon,*} \sqrt{h/m} \deltab_{(j)}  \\
\sigma_{\varepsilon,0}&=\sigma_{\varepsilon,*} \sqrt{1-hr^2+h^2r^2} \\
\Sigmab_{(j)}&=\begin{pmatrix}\pib_{(j)}^{\top}\pib_{(j)}+\sigma_{\Vb,0}^{2} & \pib_{(j)}^{\top}\\
\pib_{(j)} & \II_{p-1}
\end{pmatrix}
\end{align*}
where
\[
 \pib_{(j)}=\pib_{*}+\sigma_{\Vb,*}\sqrt{h/m}\deltab_{(j)}
\]
and $\sigma_{\Vb,0}=\sigma_{\Vb,*}\sqrt{1-h}$.
\end{definition}

\subsection{Auxiliary results}

Below we present  useful auxiliary results.

\begin{lemma}
\label{lem:1}
For a constant $c\in(0,1/2)$, let the sequence $(m,n,p)$ be such that
$1\leq m\leq p^{c}$ as well as $mn^{-1}\log p\leq1/4$ as $p\rightarrow\infty$.
Then for any $a\in(0,(1-2c)/4)$,
\[
\sum_{k=0}^{m}\left[1-kan^{-1}\log p\right]^{-n}\frac{\begin{pmatrix}m\\
k
\end{pmatrix}\begin{pmatrix}p-m-1\\
m-k
\end{pmatrix}}{\begin{pmatrix}p-1\\
m
\end{pmatrix}}\leq1+o(1).
\]
\end{lemma}

The next two results are useful for computing $\chi^2$-distance.

\begin{lemma}
\label{lem: chi2 distance gaussian}Let $g_{j}$ denote the probability density function
of $\Ncal(0,\Sigmab_{j})$ with nonsingular $\Sigmab_{j}\in\mathbb{R}^{k\times k}$
for $j=0,1,2$. Suppose that $\Sigmab_{j}$ can be decomposed as $\Sigmab_{j}=L_{j}L_{j}^{\top}$.
Then
\[
\EE_{g_{0}}\left(\frac{d\PP_{g_{1}}}{d\PP_{g_{0}}}\times\frac{d\PP_{g_{2}}}{d\PP_{g_{0}}}\right)=\frac{1}{\sqrt{\det\left(\II_{k}-\left[Q_{1}Q_{1}^{\top}-\II_{k}\right]\left[Q_{2}Q_{2}^{\top}-\II_{k}\right]\right)}},
\]
where $Q_{j}=L_{0}^{-1}L_{j}$ for $j=1,2$.
\end{lemma}

\begin{lemma}
\label{lem: compute Qj step 1}Consider the notations in Definition
\ref{def: prior dist step 1}. Let $Q_{j}=L_{\theta_{*}}^{-1}L_{\theta_{j}}$.
Then for any $j_{1},j_{2}\in\{1,...,N\}$,
\begin{align*}
& \det\left[\II_{p+1}-\left(Q_{j_{1}}Q_{j_{1}}^{\top}-\II_{p+1}\right)\left(Q_{j_{2}}Q_{j_{2}}^{\top}-\II_{p+1}\right)\right]
\\
& \qquad =\left[1-m^{-1}h[r^2(1-h)^2+1]\deltab_{(j_{1})}^{\top}\deltab_{(j_{2})}\right]^{2}.
\end{align*}
\end{lemma}

With the help of Lemmas 1, 2 and 3, we can provide the next result concerning the distance between the null and alternative hypothesis.

\begin{lemma}
\label{lem: chi2 distance bound step 1}Consider $\theta_{*}\in\Theta_{\zeta,\kappa }(m,\beta_{0}+\rho s n^{-1}\log p)$
as defined in the proof of Theorem \ref{THM:4}. Consider $\{\theta_{j}\}_{j=1}^{N}$
defined in Definition \ref{def: prior dist step 1}. Let $\rho$ be defined as in (\ref{eq: def rho}). Then
\[
\limsup_{n\rightarrow\infty}\EE_{\theta_{*}}\left(N^{-1}\sum_{j=1}^{N}\frac{d\PP_{\theta_{j}}}{d\PP_{\theta_{*}}}-1\right)^{2}=0.
\]
\end{lemma}

Note that $\theta_j$ is a function of $\rho$ and $\gammab$.
 Next, we show that the designed points, $\theta_{j}$  belong to the   the null parameter space.

\begin{lemma}
\label{lem: eligibility candidate null dist step 1}Consider $\theta_{*}\in\Theta_{\zeta,\kappa}(m,\beta_0+h_n)$ with $h_{n}= \rho sn^{-1}\log p$
in the proof of Theorem \ref{THM:4}. Consider $\{\theta_{j}\}_{j=1}^{N}$
defined in Definition \ref{def: prior dist step 1}. Suppose that the conditions in the statement of Theorem  \ref{THM:4} hold. Then $$ \{\theta_{j}: 1\leq j\leq N\} \subset \Theta_{\zeta,\kappa}(2m,\beta_{0}).$$
\end{lemma}

Lastly, the following lemma describes the structure of the covariance matrices.

\begin{lemma}
\label{lem:6}Consider any $a>0$ and $b\in\mathbb{R}^{p-1}$. Let $\Sigmab$ be a positive definite matrix.
If all the eigenvalue of $\begin{pmatrix}a & b^{\top}\Sigmab\\
\Sigmab b & \Sigmab
\end{pmatrix}$ are positive, then $a>b^{\top}\Sigmab b $. \\
In particular, if all the eigenvalues of  $\begin{pmatrix}a & b^{\top}\\
b & \II_{p-1}
\end{pmatrix}$ are positive, then $a>b^\top b$.
\end{lemma}

Now, that all of the auxiliary results are established, we are ready to present the main proof.

\subsection{Proof of Theorem \ref{THM:4}}

The proof methodology is novel  in that for each possible candidate  point in the alternative, we need to design a sequence of points in the null space  and demonstrate that their $\chi^2$-distances to the candidate point in the alternative  will be small therefore limiting the power of the test.

\begin{proof}[\textbf{Proof of Theorem \ref{THM:4}}]

Recall that  $m$ denotes the largest integer not exceeding $s/2$, i.e.,
$m=\left\lfloor s/2\right\rfloor$.
 Fix any $\eta>0$. Recall $\rho$ defined in (\ref{eq: def rho}).

Observe that by the properties of the supremum,
we can choose $\psi_{*}\in\Psi_{\alpha}(\Theta(s,\beta_{0}))$ and
$$\theta_{*}=(\beta_{*},\gammab_{*},\Sigmab_{*},\sigma_{\varepsilon,*})\in\Theta_{\zeta,  \kappa }(s/2,\beta_{0}+h_{n})$$
with $h_{n}=\rho sn^{-1}\log p$ such that
\begin{equation}
\sup_{\psi\in\Psi_{\alpha}(\Theta(s,\beta_{0}))}\sup_{\theta\in\Theta_{\zeta, \kappa}(s/2,\beta_{0}+h_{n})}\EE_{\theta}\psi\leq \EE_{\theta_{*}}\psi_{*}+\eta.\label{eq: main proof eq 1}
\end{equation}

Since $\|\cdot\|_{0}$ can only take values in $\mathbb{Z}$, $\Theta_{\zeta,  \kappa}(s/2,\beta_{0}+h_{n})=\Theta_{\zeta, \kappa}(m,\beta_{0}+h_{n})$.

By Lemma \ref{lem:6}, there exist $\sigma_{\Vb,*}>0$
and $\pib_{*}\in\mathbb{R}^{p-1}$ such that
$$\Sigmab_{*}=\begin{pmatrix}\pib_{*}^{\top}\pib_{*}+\sigma_{\Vb,*}^{2} & \pib_{*}^{\top}\\
\pib_{*} & \II_{p-1}
\end{pmatrix}.$$  We
construct $\{\theta_{j}\}_{j=1}^{N}$ as in the   Definition
\ref{def: prior dist step 1}.

Since $\EE_{\theta_{j}}\psi_{*}=\EE_{\theta_{*}}\left(\psi_{*}\frac{d\PP_{\theta_{j}}}{d\PP_{\theta_{*}}}\right)$,
it follows that
\begin{align*}
&\left|N^{-1}\sum_{j=1}^{N}\EE_{\theta_{j}}\psi_{*}-\EE_{\theta_{*}}\psi_{*}\right|
\\
&\qquad  =\left|N^{-1}\sum_{j=1}^{N}\left(\EE_{\theta_{*}}\psi_{*}\frac{d\PP_{\theta_{j}}}{d\PP_{\theta_{*}}}-\EE_{\theta_{*}}\psi_{*}\right)\right|  =\left|\EE_{\theta_{*}}\psi_{*}\left(N^{-1}\sum_{j=1}^{N}\frac{d\PP_{\theta_{j}}}{d\PP_{\theta_{*}}}-1\right)\right|\\
 & \qquad  \overset{(i)}{\leq}\EE_{\theta_{*}}\left|N^{-1}\sum_{j=1}^{N}\frac{d\PP_{\theta_{j}}}{d\PP_{\theta_{*}}}-1\right|  \overset{(ii)}{\leq}\left[\EE_{\theta_{*}}\left(N^{-1}\sum_{j=1}^{N}\frac{d\PP_{\theta_{j}}}{d\PP_{\theta_{*}}}-1\right)^{2}\right]^{1/2},
\end{align*}
where $(i)$ holds by $|\psi_{*}|\leq1$ almost surely and $(ii)$
holds by Lyapunov's inequality.

By Lemma \ref{lem: chi2 distance bound step 1} and the above display,
we have
\[
\limsup_{n\rightarrow\infty}\left|N^{-1}\sum_{j=1}^{N}\EE_{\theta_{j}}\psi_{*}-\EE_{\theta_{*}}\psi_{*}\right|=0.
\]

By Lemma \ref{lem: eligibility candidate null dist step 1}, $\theta_{j}\in\Theta_{\zeta,\kappa}(2m,\beta_{0})\subseteq\Theta(s,\beta_{0})$
for all $j\in\{1,...,N\}$. This and the fact that $\psi_{*}\in\Psi_{\alpha}(\Theta(s,\beta_{0}))$
imply
\[
N^{-1}\sum_{j=1}^{N}\EE_{\theta_{j}}\psi_{*}\leq\alpha.
\]

Hence, $\limsup_{n\rightarrow\infty}\EE_{\theta_{*}}\psi_{*}\leq\alpha$.
By (\ref{eq: main proof eq 1}), we have
\[
\limsup_{n\rightarrow\infty}\sup_{\psi\in\Psi_{\alpha}(\Theta(s,\beta_{0}))}\sup_{\theta\in\Theta_{\zeta,\kappa}(s/2,\beta_{0}+h_{n})}\EE_{\theta}\psi\leq\limsup_{n\rightarrow\infty}\EE_{\theta_{*}}\psi_{*}+\eta\leq\alpha+\eta.
\]

Moreover, since $\eta>0$ is arbitrary, we have
\[
\limsup_{n\rightarrow\infty}\sup_{\psi\in\Psi_{\alpha}(\Theta(s,\beta_{0}))}\sup_{\theta\in\Theta_{\zeta,\kappa}(s/2,\beta_{0}+h_{n})}\EE_{\theta}\psi\leq\alpha.
\]
Notice that a random test that rejects the hypothesis at random with probability $\alpha$ has power equal to $\alpha$. Since  $\Psi_{\alpha}(\Theta(s,\beta_0)) $ includes such random tests, the above inequality holds with equality. The proof is complete.
\end{proof}

\section{Proof of Theorem \ref{THM:5}}

\begin{proof}[Proof of Theorem \ref{THM:5}]
Here we will show  that no test can be better than the Likelihood Ratio test.

Recall $\tau=\kappa \sqrt{M^{-1}\log(1+\alpha^{2})}$.
We choose an arbitrary test
\[
\psi_{**}\in\Psi_{\alpha}(\Theta(s,\beta_{0}))
\]
and an arbitrary point
\[
\theta_{**}=(\beta_{**},\gammab_{**},\Sigmab_{**},\sigma_{\varepsilon,**})\in  \Theta_{\kappa}(s,\beta_{0}+h_n)
\]
with $h_{n}=\tau n^{-1/2}$. Throughout this proof we denote with $\theta_{**}$ the point in the alternative space. Notice that  $\beta_{**}=\beta_{0}+h$ with  $0\leq h \leq h_n$.

We define
\[
\theta_{0}=(\beta_{0},\gammab_{**},\Sigmab_{**},\sigma_{\varepsilon,**}).
\]

Clearly,
\[
\theta_{0}\in \Theta(s,\beta_{0})\quad {\rm and\quad thus}\quad E_{\theta_0}\psi_{**}\leq \alpha.
\]

Recall the notation $\Xb_{i}=(Z_i,\Wb_i^\top)^{\top}\in\mathbb{R}^{p}$.
Let $\sigma_{z}^{2}=\EE_{\theta_{0}}Z_i^{2}$. By the definition of
$\Theta_{\kappa}(s,\beta_{0})$,
\begin{equation}
\sigma_{z}^{2}\leq\lambda_{\max}(\Sigmab_{**})\leq M\ {\rm and}\ \sigma_{\varepsilon,**}\geq\kappa ^{2}.\label{eq: thm 5 eq 0.8}
\end{equation}

Then the likelihood of the data under $\PP_{\theta_{**}}$ can be written
as a product of the likelihood of $\yb$ given $\Xb$ and the likelihood
of $\Xb$:
\begin{multline*}
\left[\frac{1}{(\sqrt{2\pi}\sigma_{\varepsilon,**})^n}\exp\left(-\frac{1}{2\sigma_{\varepsilon,**}^{2}}\sum_{i=1}^{n}(y_{i}-Z_i\beta_{**}-\Wb_i^\top\gammab_{**})^{2}\right)\right]\\
\times\left[\frac{1}{(\sqrt{\det(2\pi\Sigmab_{**})})^n}\exp\left(-\frac{1}{2}\sum_{i=1}^{n}\Xb_{i}^{\top}\Sigmab_{**}^{-1}\Xb_{i}\right)\right].
\end{multline*}

Similarly, the likelihood of the data under $\PP_{\theta_{0}}$ can
be written as
\begin{multline*}
\left[\frac{1}{(\sqrt{2\pi}\sigma_{\varepsilon,**})^n}\exp\left(-\frac{1}{2\sigma_{\varepsilon,**}^{2}}\sum_{i=1}^{n}(y_{i}-Z_i\beta_{0}-\Wb_i^\top\gammab_{**})^{2}\right)\right]\\
\times\left[\frac{1}{(\sqrt{\det(2\pi\Sigmab_{**})})^n}\exp\left(-\frac{1}{2}\sum_{i=1}^{n}\Xb_{i}^{\top}\Sigmab_{**}^{-1}\Xb_{i}\right)\right].
\end{multline*}

Hence, the likelihood ratio can be written as
\begin{align}
& \frac{d\PP_{\theta_{**}}}{d\PP_{\theta_{0}}}  \nonumber
\\
&  =\exp\left(\frac{1}{2\sigma_{\varepsilon,**}^{2}}\sum_{i=1}^{n}\left[(y_{i}-Z_i\beta_{0}-\Wb_i^\top\gammab_{**})^{2}-(y_{i}-Z_i\beta_{**}-\Wb_i^\top\gammab_{**})^{2}\right]\right)\nonumber \\
 &  \overset{(i)}{=}\exp\left(\frac{h}{\sigma_{\varepsilon,**}^{2}}\sum_{i=1}^{n}Z_i\left[y_{i}-Z_i(\beta_{0}+h/2)-\Wb_i^\top\gammab_{**}\right]\right),\label{eq: thm 5 eq 1}
\end{align}
where $(i)$ follows by $\beta_{**}=\beta_{0}+h$. Thus,
\begin{align}
\left|\EE_{\theta_{0}}\psi_{**}-\EE_{\theta_{**}}\psi_{**}\right| & =\left|\EE_{\theta_{0}}\psi_{**}-\EE_{\theta_{0}}\psi_{**}\frac{d\PP_{\theta_{**}}}{d\PP_{\theta_{0}}}\right|\nonumber \\
 & =\left|\EE_{\theta_{0}}\psi_{**}\left(\frac{d\PP_{\theta_{**}}}{d\PP_{\theta_{0}}}-1\right)\right|\nonumber \\
 & \overset{(i)}{\leq}\EE_{\theta_{0}}\left|\frac{d\PP_{\theta_{**}}}{d\PP_{\theta_{0}}}-1\right|\nonumber \\
 & \leq\sqrt{\EE_{\theta_{0}}\left(\frac{d\PP_{\theta_{**}}}{d\PP_{\theta_{0}}}-1\right)^{2}}=\sqrt{\EE_{\theta_{0}}\left(\frac{d\PP_{\theta_{**}}}{d\PP_{\theta_{0}}}\right)^{2}-1},\label{eq: thm 5 eq 1.5}
\end{align}
where $(i)$ follows by $|\psi_{**}|\leq1$. By (\ref{eq: thm 5 eq 1}),
we have
\begin{align}
\EE_{\theta_{0}}\left(\frac{d\PP_{\theta_{**}}}{d\PP_{\theta_{0}}}\right)^{2} & =\EE_{\theta_{0}}\left[\exp\left(\frac{2h}{\sigma_{\varepsilon,**}^{2}}\sum_{i=1}^{n}Z_i\left[y_{i}-Z_i(\beta_{0}+h/2)-\Wb_i^\top\gammab_{**}\right]\right)\right]\nonumber \\
 & \overset{(i)}{=}\EE_{\theta_{0}}\left[\exp\left(\frac{2h}{\sigma_{\varepsilon,**}^{2}}\sum_{i=1}^{n}Z_i\left[\varepsilon_{i}-Z_ih/2\right]\right)\right]\nonumber \\
 & =\EE_{\theta_{0}}\left\{ \EE_{\theta_{0}}\left[\exp\left(\frac{2h}{\sigma_{\varepsilon,**}^{2}}\sum_{i=1}^{n}Z_i\left[\varepsilon_{i}-Z_ih/2\right]\right)\mid \Zb\right]\right\} \nonumber \\
 & =\EE_{\theta_{0}}\left\{ \EE_{\theta_{0}}\left[\exp\left(\frac{2h}{\sigma_{\varepsilon,**}^{2}}\sum_{i=1}^{n}Z_i\varepsilon_{i}\right)\mid \Zb\right]\exp\left(-\frac{h^{2}}{\sigma_{\varepsilon,**}^{2}}\sum_{i=1}^{n}Z_i^{2}\right)\right\} \label{eq: thm 5 eq 2}
\end{align}
where $(i)$ follows by the fact that under $\PP_{\theta_{0}}$, $y_{i}=Z_i\beta_{0}+\Wb_i^\top\gammab_{**}+\varepsilon_{i}$.

Notice that under $\PP_{\theta_{0}}$, $\sum_{i=1}^{n}Z_i\varepsilon_{i}$
conditional on $\Zb$ has a Gaussian distribution with mean 0 and variance
equal to $\sum_{i=1}^{n}Z_i^{2}\sigma_{\varepsilon,**}^{2}$.
Hence, by the moment generating function of Gaussian distributions,
it follows that
\[
\EE_{\theta_{0}}\left[\exp\left(\frac{2h}{\sigma_{\varepsilon,**}^{2}}\sum_{i=1}^{n}Z_i\varepsilon_{i}\right)\mid \Zb\right]=\exp\left(2\sigma_{\varepsilon,**}^{-2}h^{2}\sum_{i=1}^{n}Z_i^{2}\right).
\]

Therefore, we can use the above display to continue (\ref{eq: thm 5 eq 2})
and obtain
\begin{align*}
\EE_{\theta_{0}}\left(\frac{d\PP_{\theta_{**}}}{d\PP_{\theta_{0}}}\right)^{2} & =\EE_{\theta_{0}}\left\{ \EE_{\theta_{0}}\left[\exp\left(\frac{2h}{\sigma_{\varepsilon,**}^{2}}\sum_{i=1}^{n}Z_i\varepsilon_{i}\right)\mid \Zb\right]\exp\left(-\frac{h^{2}}{\sigma_{\varepsilon,**}^{2}}\sum_{i=1}^{n}Z_i^{2}\right)\right\} \\
 & =\EE_{\theta_{0}}\left\{ \exp\left(2\sigma_{\varepsilon,**}^{-2}h^{2}\sum_{i=1}^{n}Z_i^{2}\right)\exp\left(-\frac{h^{2}}{\sigma_{\varepsilon,**}^{2}}\sum_{i=1}^{n}Z_i^{2}\right)\right\} \\
 & =\EE_{\theta_{0}}\left[\exp\left(\sigma_{\varepsilon,**}^{-2}h^{2}\sum_{i=1}^{n}Z_i^{2}\right)\right]\\
 & =\EE_{\theta_{0}}\left[\exp\left(\sigma_{z}^{2}\sigma_{\varepsilon,**}^{-2}h^{2}\sum_{i=1}^{n}(Z_i^{2}\sigma_{z}^{-2})\right)\right]\\
 & \overset{(i)}{\leq}\EE_{\theta_{0}}\left[\exp\left(\left[\log(1+\alpha^{2})\right]n^{-1}\sum_{i=1}^{n}(Z_i^{2}\sigma_{z}^{-2})\right)\right]\\
 & \overset{(ii)}{=}\left(1-2n^{-1}\log(1+\alpha^{2})\right)^{-n/2}\\
 & \overset{(iii)}{\leq}\exp\left[\log(1+\alpha^2)\right]=1+\alpha^2,
\end{align*}
where $(i)$ follows by $0\leq h \leq h_{n}=\tau n^{-1/2}=n^{-1/2}\kappa \sqrt{M^{-1}\log(1+\alpha^{2})}$
and (\ref{eq: thm 5 eq 0.8}), $(ii)$ follows by the moment generating
function of $\chi^{2}(n)$ (chi-squared distribution with $n$ degrees
of freedom) and the fact that
$$\sum_{i=1}^{n} Z_i^{2}\sigma_{z}^{-2} $$
has a $\chi^{2}(n)$ distribution together with $n^{-1} \log(1+\alpha^{2})<1/2$ (due
to $\alpha^{2}<1/4$ and $\log(1.25)<1/2$) and $(iii)$ follows by the fact
that
$$(1-a/n)^{-n/2}\leq\exp(a/2)$$ for any $n\geq1$ and $a\geq0$.

Therefore, the above display and (\ref{eq: thm 5 eq 1.5}) imply that
\[
\left|\EE_{\theta_{0}}\psi_{**}-\EE_{\theta_{**}}\psi_{**}\right|\leq\sqrt{\EE_{\theta_{0}}\left(\frac{d\PP_{\theta_{**}}}{d\PP_{\theta_{0}}}\right)^{2}-1}=\sqrt{\alpha^{2}}=\alpha.
\]

Since $\EE_{\theta_{0}}\psi_{**}\leq\alpha$, it follows that $\EE_{\theta_{**}}\psi_{**}\leq2\alpha$.
Since $\psi_{**}$ and $\theta_{**}$ are chosen arbitrarily, the
desired result follows.
\end{proof}

\section{Proof of Theorem \ref{THM:11}}

Proof of Theorem \ref{THM:11} has been split into a sequence of smaller results. First we present some notation, then auxiliary Lemmas \ref{lem: KL computation} - \ref{lem: noiseless}  that are useful in the proof of Theorem \ref{THM:11} and lastly the proof of the result itself.

We first recall the notions of total variation and KL
divergence. Given two probability measures $\PP_{0}$ and $\PP_{1}$
that are absolutely continuous with each other, we define the total
variation
\[
\TV(\PP_{0},\PP_{1})=\frac{1}{2}\int\left|\frac{d\PP_{1}}{d\PP_{0}}-1\right|d\PP_{0}
\]
and KL divergence:
\[
\KL(\PP_{0},\PP_{1})=\int\left(\log\frac{d\PP_{0}}{d\PP_{1}}\right)d\PP_{0}.
\]

\subsection{Auxiliary results}

\begin{lemma}
	\label{lem: KL computation}Let $\PP_{0}$ and $\PP_{1}$ denote the
	probability measures for $\Ncal(\mu_{0},\Sigmab_{0})$ and $\Ncal(\mu_{1},\Sigmab_{1})$,
	respectively. Then
	\[
	\KL(\PP_{0},\PP_{1})=\frac{1}{2}\left(\trace(\Sigmab_{1}^{-1}(\Sigmab_{0}-\Sigmab_{1}))+\log\left(\frac{\det(\Sigmab_{1})}{\det(\Sigmab_{0})}\right)+(\mu_{1}-\mu_{0})^{\top}\Sigmab_{1}^{-1}(\mu_{1}-\mu_{0})\right).
	\]
\end{lemma}
The proof of Lemma \ref{lem: KL computation} follows by straight-forward
computation and is thus omitted. The next two results are useful  bounding tools.
\begin{lemma}
	\label{lem: bnd 1}Let $\Wb\in\RR^{n\times2n}$ and $\Zb\in\RR^{n}$
	have entries being i.i.d standard normal random variables. Then for any $a>0$
	\[
	\PP\left(\Zb^{\top}(\Wb\Wb^{\top})^{-1}\Zb>a\right)\leq2\exp(-0.005n)+12/a.
	\]
\end{lemma}

\begin{lemma}
	\label{lem: bnd 2}Let $\xib$ be a random vector  with
	distribution $\Ncal(0,\Sigmab)$. Then for any $x>0$,
	\[
	\PP(\|\xib\|_{2}>x)\leq x^{-2}\trace(\Sigmab).
	\]
\end{lemma}

The main lemma utilized in the proof is the following one.

\begin{lemma}
	\label{lem: noiseless}Assume that $p\geq 2n+1$. For any $r\in\RR$, we define
	\[
	\Theta_{*}(r)=\left\{ \theta=(\beta,\gammab,\Sigmab,\sigma):\ \beta=r,\ \|\gammab\|_{2}\leq1,\ \Sigmab=\II_{p},\ \sigma=0\right\} .
	\]
	Then there exists a constant $K>0$ depending only on $\alpha$ such
	that
	$$\inf_{\theta\in\Theta_{*}(n^{-1/2}K)}\EE_{\theta}\psi\leq1-2\alpha$$
	for any measurable function of the data $(\yb,\Wb,\Zb)$ satisfying
	$|\psi(\yb,\Wb,\Zb)|\leq1$ and $\sup_{\theta\in\Theta_{*}(0)}\EE_{\theta}\psi\leq\alpha$.
\end{lemma}

\subsection{Proof of Theorem \ref{THM:11}}
\begin{proof}[Proof of Theorem \ref{THM:11}]
	Let
	$$CI(\yb,\Zb,\Wb)=[l(\yb,\Zb,\Wb),u(\yb,\Zb,\Wb)]$$ be a confidence
	set for $\beta$ with nominal coverage probability $1-\alpha$ over
	$\widetilde{\Theta}_{0,1}(1)$. Define $\psi(\yb,\Zb,\Wb)=\mathbf{1}\left\{ 0\notin CI(\yb,\Zb,\Wb)\right\} $.
	
	From now on, we will write $CI$, $\psi$, $u$ and $l$ without $(\yb,\Zb,\Wb)$
	to simplify the notation.
	
	Recall the notation $\Theta_{*}(r)$ from Lemma \ref{lem: noiseless}.
	Since $\Theta_{*}(0)\subset\widetilde{\Theta}_{0,1}(1)$, we have
	$$\sup_{\theta\in\Theta_{*}(0)}\EE_{\theta}\psi\leq\alpha.$$ Moreover, by the same Lemma
	\ref{lem: noiseless},
	$$\inf_{\theta\in\Theta_{*}(n^{-1/2}K)}\EE_{\theta}\psi\leq2\alpha$$
	for some constant $K>0$ depending only on $\alpha$. This means that
	\[
	\inf_{\theta\in\Theta_{*}(n^{-1/2}K)}\PP_{\theta}\left(l\leq0\leq u\right)\geq1-2\alpha.
	\]
	
	Since $\Theta_{*}(n^{-1/2}K)\subset\widetilde{\Theta}_{0,1}(1)$,
	we have that
	\[
	\inf_{\theta\in\Theta_{*}(n^{-1/2}K)}\PP_{\theta}\left(l\leq n^{-1/2}K\leq u\right)\geq1-\alpha.
	\]
	
	Therefore,
	\[
	\inf_{\theta\in\Theta_{*}(n^{-1/2}K)}\PP_{\theta}\left(l\leq0<n^{-1/2}K\leq u\right)\geq1-3\alpha.
	\]
	
	It follows that
	\[
	\inf_{\theta\in\Theta_{*}(n^{-1/2}K)}\PP_{\theta}\left(u-l\geq n^{-1/2}K\right)\geq1-3\alpha
	\]
	and thus
	\begin{align*}
	\sup_{\theta\in\widetilde{\Theta}_{0,1}(1)}\EE_{\theta}(u-l) & \geq\sup_{\theta\in\Theta_{*}(n^{-1/2}K)}\EE_{\theta}(u-l)
	\\
	&\geq
	\sup_{\theta\in\Theta_{*}(n^{-1/2}K)}\EE_{\theta}\left[\mathbf{1}\left\{ u-l\geq n^{-1/2}K\right\} \times n^{-1/2}K\right]
	\\
	&\geq n^{-1/2}K\times(1-3\alpha).
	\end{align*}
	
	Since the above bound holds for any confidence interval $CI$, the
	proof is complete.
\end{proof}

\bibliographystyle{apalike}
\bibliography{biblio_state_space}

\setcounter{equation}{0}
\setcounter{lemma}{0}

\newpage

\appendix
\setcounter{section}{3}
    \pagenumbering{arabic}
\setcounter{page}{1}
\title{
	Supplement to the ``Testability of high-dimensional  linear models with non-sparse structures''}

\bigskip 
\bigskip 

This document collects detailed proofs of Theorems \ref{THM: THM KNOWN SIGMA}, \ref{THM:8}, \ref{THM:9}, \ref{THM:10} and \ref{THM:12} as well as Corollaries \ref{THM:1}, \ref{COR:2} and \ref{COR:9} of the main text, as well as detailed proofs of  the twelve  supplementary lemmas (alphabetically enumerated in this document): Lemma \ref{lem: Bernstein} - \ref{lem: scale minimax}. In particular, Lemmas \ref{lem: Bernstein} - \ref{lem: dantzig pi} are used for Theorem \ref{THM:8}.  Lemmas \ref{lem: sufficient for non adaptivity}-\ref{lem: necessary for non adaptivity} are used for Theorem \ref{THM:10}. 
Lemmas \ref{lem: scaling para space} - \ref{lem: scale minimax} are used for Theorem \ref{THM:12}.

\section{Proof of  Theorem \ref{THM: THM KNOWN SIGMA}}\label{sec: Proof known Omega}
\begin{proof}[\textbf{Proof of  Theorem \ref{THM: THM KNOWN SIGMA}}]
	To simply notation, we write $\EE $ instead of $\EE_{\theta}$. All the statements hold uniformly for any $\theta\in\widetilde{\Theta} $.
	Let $\hat{\beta}=\Omegab_{1,}\Xb^{\top}\yb/n$. Then
	\begin{align*}
	\hat{\beta}-\beta & =\Omegab_{1,}\Xb^{\top}(\Zb\beta+\Wb\gammab+\varepsilonb)/n-\beta\\
	& =\left[\Omegab_{1,}\Xb^{\top}\Zb/n-1\right]\beta+\Omegab_{1,}\Xb^{\top}(\Wb\gammab+\varepsilonb)/n.
	\end{align*}
	
	Notice that
	$$\Sigmab=\begin{pmatrix}\gammab^{\top}\Sigmab_{\Wb}\gammab+\sigma_{\Vb}^{2} & \Sigmab_{\Wb}\gammab\\
	\gammab^{\top}\Sigmab_{\Wb} & \Sigmab_{\Wb}
	\end{pmatrix}$$
	and $\Omegab_{1,}=\sigma_{\Vb}^{-2}(1,-\pib^{\top})$, where
	$\Sigmab_{\Wb}=\EE(\Wb^{\top}\Wb)/n$, $\sigma_{\Vb}^{2}=\EE(\Vb^{\top}\Vb)/n$ and
	$\Vb=\Zb-\Wb\pib$.
	Then
	\begin{align*}
	\left[\Omegab_{1,}\Xb^{\top}\Zb/n-1\right]\beta+\Omegab_{1,}\Xb^{\top}(\Wb\gammab+\varepsilonb)/n.
	\\
	=n^{-1}\sum_{i=1}^{n}\left[(v_{i}Z_i\sigma_{\Vb}^{-2}-1)+\sigma_{\Vb}^{-2}v_{i}(\Wb_i^\top\gammab+\varepsilon_{i})\right].
	\end{align*}
	where $v_{i}$, $Z_i$ and $\Wb_i^\top\gammab$ denote the $i$-th
	entry of $\Vb$, $\Zb$ and $\Wb\gammab$, respectively.
	
	Notice that
	$$\left\{(v_{i}Z_i\sigma_{\Vb}^{-2}-1)+\sigma_{\Vb}^{-2}v_{i}(\Wb_i^\top\gammab+\varepsilon_{i})\right\}_{i=1}^{n}$$
	is an i.i.d sequence of random variables with bounded sub-exponential
	norms. Therefore, 
	\begin{align*}
	\EE (\hat{\beta}-\beta)^2 & =n^{-2}\sum_{i=1}^{n}\left[(v_{i}Z_i\sigma_{\Vb}^{-2}-1)+\sigma_{\Vb}^{-2}v_{i}(\Wb_i^\top\gammab+\varepsilon_{i})\right]^2\lesssim n^{-1}.
	\end{align*}
	The desired result follows by noticing $\EE |\hat{\beta}-\beta| \leq \sqrt{\EE (\hat{\beta}-\beta)^2} $.  \end{proof}

\section{Proof of Theorem \ref{THM:8}}

Before the main proof we establish a sequence of useful auxiliary results. Then we shall prove Theorem \ref{THM:8}. To simplify notations, we write $\PP $ instead of $\PP_{\theta} $. Note that all the results here hold uniformly over $\theta\in\widetilde{\Theta}(s) $ in finite samples. Therefore, we also omit $\sup_{\theta\in\widetilde{\Theta}(s) } $ and $\inf_{\theta\in\widetilde{\Theta}(s) } $ whenever possible.

\subsection{Auxiliary results}

The following result establishes a concentration result regarding the product of two Gaussian random variables that are allowed to be dependent. In particular, the result generalizes the concentration of measure of chi-squared random variables.

\begin{lemma}
	\label{lem: Bernstein}Let $\{r_{i,1}\}_{i=1}^{n}$ and $\{r_{i,2}\}_{i=1}^{n}$
	be sequences of i.i.d random variables with $\mathcal{N}(0,\sigma_{1}^{2})$
	and $\mathcal{N}(0,\sigma_{2}^{2})$ distributions, respectively that are not necessarily independent from each other.
	Then for any $t>0$,
	\[
	\mathbb{P}\left(\left|\sum_{i=1}^{n}\left(r_{i,1}r_{i,2}-\mathbb{E}r_{i,1}r_{i,2}\right)\right|\geq t\sigma_{1}\sigma_{2}\right)\leq2\exp\left(-\frac{t^{2}}{2(2n+7t)}\right).
	\]
\end{lemma}

\begin{lemma}
	\label{lem: bounded A} Let the assumption of Theorem \ref{THM:8}
	hold. Then,
	\begin{itemize}
		\item[(1)] The population parameter $\xib$ satisfies $\|\boldsymbol{\xi}\|_{2}\leq2M^{2}M_{2}$.
		\item[(2)] The estimator $\hat \xib$ satisfies
		$$\mathbb{P}\left(\|\hat{\boldsymbol{\xi}}-\boldsymbol{\xi}\|_{\infty}>2b_{n}^{-1}M\sqrt{n(\log p)(M_{1}^{2}+M_{2}^{2})}\right)\leq2/p.$$
		\item[(3)] Similarly,
		$$\mathbb{P}\biggl(\|\hat{\boldsymbol{\xi}}-b_{n}^{-1}\sum_{i\in H_{4}}\mathbf{W}_{i}y_{i}\|_{\infty}>4b_{n}^{-1}M\sqrt{n(\log p)(M_{1}^{2}+M_{2}^{2})}\biggl)\leq4/p.$$
		\item[(4)] Moreover,
		$$\mathbb{P}\left(\|\hat{\boldsymbol{\xi}}_{A}\|_{2}\leq4M^{2}M_{2}\right)\geq1-4/p. $$
		\item[(5)] The $\ell_\infty$-norm of estimation error of the thresholded estimator is
		$$\mathbb{P}\left(\|\hat{\boldsymbol{\xi}}_{A^{c}}\|_{\infty}\leq8b_{n}^{-1}M\sqrt{n(\log p)(M_{1}^{2}+M_{2}^{2})}\right)\geq1-4/p.$$
		\item[(6)] Lastly,
		\begin{align*}
		& \mathbb{P}\biggl(\bigl|b_{n}^{-1}\sum_{i\in H_{4}}v_{i}(\mathbf{W}_{i}^{\top}(\boldsymbol{\pi}\beta+\boldsymbol{\gamma})+\varepsilon_{i} ) )\bigl| \biggl.
		\\
		&
		\qquad \qquad
		\biggl.
		\leq10b_{n}^{-1/2}\sqrt{M\left(4M_{2}^{2}M^{3}+M_{1}^{2}\right)\log(100/\alpha)}\biggl)\geq1-0.02\alpha,
		\end{align*}
		\item[(7)] and
		$$\mathbb{P}\biggl(b_{n}^{-1}\sum_{i\in H_{4}}v_{i}^{2}\geq(2M)^{-1}\biggl)\geq1-2\exp\left(-M^{-2}b_{n}/44\right).$$
	\end{itemize}
\end{lemma}

We now discuss the estimation properties of the  proposed regularized estimator $\hat{\boldsymbol{\Omega}}_{\mathbf{W}}$.

\begin{lemma}
	\label{lem: feasibility Omega}Let the assumption of Theorem \ref{THM:8}
	hold. Then $\boldsymbol{\Omega}_{\mathbf{W}}$ satisfies the constraint
	in (\ref{eq: constraint OmegaW}) for $\hat{\boldsymbol{\Omega}}_{\mathbf{W}}$
	with probability at least $1-10/p-2\exp(-b_{n}/18)$.
\end{lemma}

The next result establishes  a lower bound on the restricted eigenvalue constant
\[
\kappa(s)=\underset{|J|\subset\{1,...,p-1\},|J|\leq s}{\min}\ \underset{\|\boldsymbol{q}_{J^{c}}\|_{1}\leq3\|\boldsymbol{q}_{J}\|_{1}}{\min}\frac{b_{n}^{-1}\sum_{i\in H_{2}}(\mathbf{W}_{i}^{\top}\boldsymbol{q})^{2}}{\|\boldsymbol{q}_{J}\|_{2}^{2}}.
\]

\begin{lemma}
	\label{lem: RE condition}Let $\tau\in(0,1)$ be an arbitrary constant.
	Whenever
	$$\left(1+36M^{2}(1+\tau)^{2}(1-\tau)^{-2}\right)s\leq p-1,$$
	and
	$b_{n}\geq570\left[1+36M^{2}(1+\tau)^{2}(1-\tau)^{-2}\right]\tau^{-2}s\log(12ep/\tau)$,
	then
	\[
	\mathbb{P}\left(\kappa(s)>0.24(1-\tau)^{2}M^{-1}\right)\geq1-4\exp(-\tau^{2}b_{n}/570).
	\]
\end{lemma}

The following result establishes  finite-sample properties of the Lasso estimator and follows by standard arguments. We include it here for completeness and clarity.

\begin{lemma}
	\label{lem: lasso bound}Let the assumption of Theorem \ref{THM:8}
	hold. Then,
	\[
	\mathbb{P}\left(\|\hat \pib-\boldsymbol{\pi}\|_{1}\leq267s\lambda_{\boldsymbol{\pi}}M\right)\geq1-4\exp\left(-3b_{n}/3040\right)-2/p^{2}
	\]
	and
	\[
	\mathbb{P}\biggl(\bigl\|\sum_{i\in H_{4}}\mathbf{W}_{i}v_{i}\bigl\| _{\infty}/b_{n}\leq\lambda_{\boldsymbol{\pi}}/4\biggl)\geq1-2/p^{2}.
	\]
\end{lemma}

The next two results establish the properties of the proposed   regularized estimator ${\breve \pib}$.

\begin{lemma}
	\label{lem: feasibility pi}Let the assumption of Theorem \ref{THM:8}
	hold. Then $\boldsymbol{\pi}$ satisfies the constraints in (\ref{eq: pi half estimator})
	for ${\breve \pib}$ with probability at
	least $1-14/p-0.02\alpha-6\exp\left(-3b_{n}/3040\right)-2\exp(-M^{-2}b_{n}/44)$.
\end{lemma}

\begin{lemma}
	\label{lem: dantzig pi}Let the assumption of Theorem \ref{THM:8}
	hold. Then with probability at least $1-14/p-0.02\alpha-10\exp\left(-3b_{n}/3040\right)-2\exp(-M^{-2}b_{n}/44)$,
	\[
	\|{\breve \pib}-\boldsymbol{\pi}\|_{1}\leq134M\lambda_{\boldsymbol{\pi}}s.
	\]
\end{lemma}

\subsection{Proof of Theorem \ref{THM:8}}

Now we are ready to prove Theorem \ref{THM:8}.

\begin{proof}[\textbf{Proof of Theorem \ref{THM:8}}]

	Let $\boldsymbol{\delta}={\breve \pib}-\boldsymbol{\pi}$.
	Notice that $\hat{v}_{i}=v_{i}-\mathbf{W}_{i}^{\top}\boldsymbol{\delta}$.
	Then
	\begin{align}
	\hat{\beta}-\beta & =\frac{b_{n}^{-1}\sum_{i\in H_{4}}\hat{v}_{i}(y_{i}-\beta\hat{v}_{i})}{b_{n}^{-1}\sum_{i\in H_{4}}\hat{v}_{i}^{2}}\nonumber \\
	&=\underset{T_{1}}{\underbrace{\frac{b_{n}^{-1}\sum_{i\in H_{4}}v_{i}(y_{i}-\beta\hat{v}_{i})}{b_{n}^{-1}\sum_{i\in H_{4}}\hat{v}_{i}^{2}}}}-\underset{T_{2}}{\underbrace{\frac{b_{n}^{-1}\sum_{i\in H_{4}}\boldsymbol{\delta}^{\top}\mathbf{W}_{i}(y_{i}-\beta\hat{v}_{i})}{b_{n}^{-1}\sum_{i\in H_{4}}\hat{v}_{i}^{2}}}}.\label{eq: thm upper bnd 2}
	\end{align}

	We now bound $T_{1}$ and $T_{2}$ in two steps. We first make the following observations.
	Notice that Lemma \ref{lem: dantzig pi} implies
	\[
	\mathbb{P}(\mathcal{M}_{1})\geq1-14/p-0.02\alpha-10\exp(-3b_{n}/3040)-2\exp(-M^{-2}b_{n}/44),
	\]
	where
	$$\mathcal{M}_{1}=\left\{ \|{\breve \pib}-\boldsymbol{\pi}\|_{1}\leq134M\lambda_{\boldsymbol{\pi}}s\right\} .$$ Moreover, Lemma \ref{lem: bounded A}  implies that $\mathbb{P}(\mathcal{M}_{2})\geq1-8/p-0.02\alpha$, where
	\begin{align*}
	\mathcal{M}_{2}= & \biggl\{ \biggl \| \hat{\boldsymbol{\xi}}-b_{n}^{-1}\sum_{i\in H_{4}}\mathbf{W}_{i}y_{i}\biggl \|_\infty \leq 4b_{n}^{-1}M\sqrt{n(\log p)(M_{1}^{2}+M_{2}^{2})}\biggl\} \\
	& \bigcap\left\{ \|\hat{\boldsymbol{\xi}}_{A^{c}}\|_{\infty}\leq8b_{n}^{-1}M\sqrt{n(\log p)(M_{1}^{2}+M_{2}^{2})}\right\} \\
	& \bigcap\Biggl\{\biggl|b_{n}^{-1}\sum_{i\in H_{4}}v_{i}\left(\mathbf{W}_{i}^{\top}(\boldsymbol{\pi}\beta+\boldsymbol{\gamma})+\varepsilon_{i}\right)\biggl|\\
	& \qquad\qquad\qquad\qquad\qquad\leq10b_{n}^{-1/2}\sqrt{M\left(4M_{2}^{2}M^{3}+M_{1}^{2}\right)\log(100/\alpha)}\Biggr\}.
	\end{align*}
	Finally, Lemma \ref{lem: feasibility pi} implies that
	\[
	\mathbb{P}(\mathcal{M}_{3})\geq1-14/p-0.02\alpha-6\exp(-3b_{n}/3040)-2\exp(-M^{-2}b_{n}/44),
	\]
	where
	\begin{align*}
	\mathcal{M}_{3}=
	& \left\{ \left|\hat{\boldsymbol{\xi}}_{A}^{\top}\boldsymbol{\pi}_{A}-\hat{\boldsymbol{\xi}}_{A}^{\top}{{\tilde \pib}_A}\right|\leq\eta_{\boldsymbol{\pi}}\right\}
	\bigcap
	\biggl\{ b_{n}^{-1}\sum_{i\in H_{4}}(Z_{i}-\mathbf{W}_{i}^{\top}{\breve \pib})^{2}\geq\frac{1}{2M}\biggl\} \\
	&  \bigcap
	\biggl\{ \biggl \| b_{n}^{-1}\sum_{i\in H_{4}}\mathbf{W}_{i}(Z_{i}-\mathbf{W}_{i}^{\top}\boldsymbol{\pi})\biggl\|_{\infty}\leq\lambda_{\boldsymbol{\pi}}/4\biggl\} .
	\end{align*}
	
	Define
	\[
	\mathcal{M}=\mathcal{M}_{1}\bigcap\mathcal{M}_{2}\bigcap\mathcal{M}_{3}.
	\]
	
	Since $b_{n}>n/4-1>n/5$ (due to $n>784$) and $p\geq360/\alpha$,
	we have
	\begin{align}
	\mathbb{P}\left(\mathcal{M}\right) & \geq1-36/p-0.06\alpha-16\exp\left(-3b_{n}/3040\right)-4\exp(-M^{-2}b_{n}/44)\nonumber \\
	& >1-0.1\alpha-0.06\alpha-16\exp\left(-3n/15200\right)-4\exp(-M^{-2}n/220)\nonumber \\
	& \overset{(i)}{\geq}1-0.16\alpha-16\times0.01\alpha-4\times0.01\alpha>1-\alpha\label{eq: thm upper bnd 2.5}
	\end{align}
	where $(i)$ follows by the assumption of $n\geq5067\log(100/\alpha)$
	and $n\geq220M^{2}\log(100/\alpha)$.
	
	Since $\hat{v}_{i}=Z_{i}-\mathbf{W}_{i}^{\top}{\breve \pib}$,
	we have that by definition, on the event $\mathcal{M}$,
	\begin{equation}
	b_{n}^{-1}\sum_{i\in H_{4}}\hat{v}_{i}^{2}\geq\frac{1}{2M}.\label{eq: thm upper bnd 3}
	\end{equation}
	
	\textbf{Step 1:} bound $T_{1}$.
	
	First observe that
	$$y_{i}=Z_{i}\beta+\mathbf{W}_{i}^{\top}\boldsymbol{\gamma}+\varepsilon_{i}=\mathbf{W}_{i}^{\top}(\boldsymbol{\pi}\beta+\boldsymbol{\gamma})+\beta v_{i}+\varepsilon_{i}.$$
	Hence, $y_{i}-\beta\hat{v}_{i}=\mathbf{W}_{i}^{\top}(\boldsymbol{\pi}\beta+\boldsymbol{\gamma})+\mathbf{W}_{i}^{\top}\boldsymbol{\delta}+\varepsilon_{i}$.
	Therefore,
	\[
	b_{n}^{-1}\sum_{i\in H_{4}}v_{i}(y_{i}-\beta\hat{v}_{i})=\underset{T_{1,1}}{\underbrace{b_{n}^{-1}\sum_{i\in H_{4}}v_{i}\left(\mathbf{W}_{i}^{\top}(\boldsymbol{\pi}\beta+\boldsymbol{\gamma})+\varepsilon_{i}\right)}}+\underset{T_{1,2}}{\underbrace{b_{n}^{-1}\sum_{i\in H_{4}}v_{i}\mathbf{W}_{i}^{\top}\boldsymbol{\delta}}}.
	\]
	
	By definition, on the event $\mathcal{M}$, we have
	\[
	|T_{1,1}|\leq10b_{n}^{-1/2}\sqrt{M\left(4M_{2}^{2}M^{3}+M_{1}^{2}\right)\log(100/\alpha)}.
	\]
	
	Notice that $\mathbf{W}_{i}v_{i}=\mathbf{W}_{i}(Z_{i}-\mathbf{W}_{i}^{\top}\boldsymbol{\pi})$.
	Therefore, on the event $\mathcal{M}$,
	\[
	|T_{1,2}|\leq\|\boldsymbol{\delta}\|_{1} \biggl\| b_{n}^{-1}\sum_{i\in H_{4}}\mathbf{W}_{i}v_{i}
	\biggl\|_{\infty}
	\leq
	\left(134M\lambda_{\boldsymbol{\pi}}s\right)\times\left(\lambda_{\boldsymbol{\pi}}/4\right)<34M\lambda_{\boldsymbol{\pi}}^{2}s.
	\]
	
	The above displays and (\ref{eq: thm upper bnd 3}) imply that on
	the event $\mathcal{M}$,
	\begin{equation}
	|T_{1}|\leq2M\left(10b_{n}^{-1/2}\sqrt{M\left(4M_{2}^{2}M^{3}+M_{1}^{2}\right)\log(100/\alpha)}+34M\lambda_{\boldsymbol{\pi}}^{2}s\right).\label{eq: thm upper bnd 5}
	\end{equation}
	
	\textbf{Step 2:} bound $T_{2}$.
	
	First notice that
	\begin{align}\label{eq: thm upper bnd 6}
	&b_{n}^{-1}\sum_{i\in H_{4}}\boldsymbol{\delta}^{\top}\mathbf{W}_{i}(y_{i}-\beta\hat{v}_{i})
	\\
	\nonumber
	&=
	\underset{T_{2,1}}{\underbrace{b_{n}^{-1}\sum_{i\in H_{4}}\boldsymbol{\delta}^{\top}(\mathbf{W}_{i}y_{i}-\hat{\boldsymbol{\xi}})}}
	+\underset{T_{2,2}}{\underbrace{\boldsymbol{\delta}^{\top}\hat{\boldsymbol{\xi}}}}-\underset{T_{2,3}}{\underbrace{b_{n}^{-1}\sum_{i\in H_{4}}\boldsymbol{\delta}^{\top}\mathbf{W}_{i}\hat{v}_{i}\beta}}.
	\end{align}
	
	On the event $\mathcal{M}$, by H\"older's inequality, we have
	\begin{align}
	|T_{2,1}| & \leq\|\boldsymbol{\delta}\|_{1}
	\biggl \| b_{n}^{-1}\sum_{i\in H_{4}}(\mathbf{W}_{i}y_{i}-\hat{\boldsymbol{\xi}}) \biggl \| _{\infty}\nonumber \\
	& \leq\left(134M\lambda_{\boldsymbol{\pi}}s\right)\times\left(4b_{n}^{-1}M\sqrt{n(\log p)(M_{1}^{2}+M_{2}^{2})}\right)\nonumber \\
	& =536b_{n}^{-1}M^{2}\sqrt{n(\log p)(M_{1}^{2}+M_{2}^{2})}\lambda_{\boldsymbol{\pi}}s.\label{eq: thm uppber bnd 7}
	\end{align}
	
	To bound $T_{2,2}$, notice that on the event $\mathcal{M}$,
	\begin{align}
	\left|T_{2,2}\right|
	& =\left|\boldsymbol{\delta}_{A}^{\top}\hat{\boldsymbol{\xi}}_{A}+\boldsymbol{\delta}_{A^{c}}^{\top}\hat{\boldsymbol{\xi}}_{A^{c}}\right|\nonumber \\
	& \leq\left|
	\left({\breve \pib}_A-{{\tilde \pib}_A}\right)^{\top}\hat{\boldsymbol{\xi}}_{A}\right|
	+\left|\left({{\tilde \pib}_A}-\boldsymbol{\pi}_{A}\right)^{\top}\hat{\boldsymbol{\xi}}_{A}\right|+\left|\boldsymbol{\delta}_{A^{c}}^{\top}\hat{\boldsymbol{\xi}}_{A^{c}}\right|\nonumber \\
	& \overset{(i)}{\leq}\eta_{\boldsymbol{\pi}}+\left|\left({{\tilde \pib}_A}-\boldsymbol{\pi}_{A}\right)^{\top}\hat{\boldsymbol{\xi}}_{A}\right|+\left|\boldsymbol{\delta}_{A^{c}}^{\top}\hat{\boldsymbol{\xi}}_{A^{c}}\right|\nonumber \\
	& \overset{(ii)}{\leq}\eta_{\boldsymbol{\pi}}+\eta_{\boldsymbol{\pi}}+\left|\boldsymbol{\delta}_{A^{c}}^{\top}\hat{\boldsymbol{\xi}}_{A^{c}}\right|\nonumber \\
	& \leq2\eta_{\boldsymbol{\pi}}+\|\boldsymbol{\delta}_{A^{c}}\|_{1}\|\hat{\boldsymbol{\xi}}_{A^{c}}\|_{\infty}\nonumber \\
	& \leq2\eta_{\boldsymbol{\pi}}+\|\boldsymbol{\delta}\|_{1}\|\hat{\boldsymbol{\xi}}_{A^{c}}\|_{\infty}\nonumber \\
	& \overset{(iii)}{\leq}2\eta_{\boldsymbol{\pi}}+\left(134M\lambda_{\boldsymbol{\pi}}s\right)\times\left(8b_{n}^{-1}M\sqrt{n(\log p)(M_{1}^{2}+M_{2}^{2})}\right)\nonumber \\
	& =2\eta_{\boldsymbol{\pi}}+1072b_{n}^{-1}M^{2}\sqrt{n(\log p)(M_{1}^{2}+M_{2}^{2})}\lambda_{\boldsymbol{\pi}}s,\label{eq: thm upper bnd 8}
	\end{align}
	where $(i)$ follows by the constraint (\ref{eq: pi half estimator})
	and $(ii)$ and $(iii)$ follow by the definition of $\mathcal{M}$.
	
	To bound $T_{2,3}$, notice on the event $\mathcal{M}$, the constraint
	in (\ref{eq: pi half estimator}) is satisfied by ${\breve \pib}$
	and thus $\|b_{n}^{-1}\sum_{i\in H_{4}}\mathbf{W}_{i}(Z_{i}-\mathbf{W}_{i}^{\top}{\breve \pib})\|_{\infty}\leq\lambda_{\boldsymbol{\pi}}/4$,
	which is
	\[
	\left\Vert b_{n}^{-1}\sum_{i\in H_{4}}\mathbf{W}_{i}\hat{v}_{i}\right\Vert _{\infty}\leq\lambda_{\boldsymbol{\pi}}/4.
	\]
	
	Therefore, on the event $\mathcal{M}$,
	\begin{align}
	|T_{2,3}|&\leq\|\boldsymbol{\delta}\|_{1}\biggl \| b_{n}^{-1}\sum_{i\in H_{4}}\mathbf{W}_{i}\hat{v}_{i}\biggl\| _{\infty}|\beta|
	\nonumber
	\\
	&\overset{(i)}{\leq}\left(134M\lambda_{\boldsymbol{\pi}}s\right)\times\left(\lambda_{\boldsymbol{\pi}}/4\right)\times M_{2}<34MM_{2}\lambda_{\boldsymbol{\pi}}^{2}s.\label{eq: thm upper bnd 9}
	\end{align}
	where $(i)$ follows by the definition of $\mathcal{B}$ and the fact
	that $|\beta|^{2}\leq\beta^{2}+\|\boldsymbol{\gamma}\|_{2}^{2} =\|\bbeta\|_2^2\leq M_{2}^{2}$.
	
	In light of (\ref{eq: thm upper bnd 6}) and (\ref{eq: thm upper bnd 3}),
	we combine (\ref{eq: thm uppber bnd 7}), (\ref{eq: thm upper bnd 8})
	and (\ref{eq: thm upper bnd 9}), obtaining that on the event $\mathcal{M}$,
	\begin{equation}
	|T_{2}|\leq2M\left(1608b_{n}^{-1}M^{2}\sqrt{n(\log p)(M_{1}^{2}+M_{2}^{2})}\lambda_{\boldsymbol{\pi}}s+2\eta_{\boldsymbol{\pi}}+34MM_{2}\lambda_{\boldsymbol{\pi}}^{2}s\right).\label{eq: thm upper bnd 10}
	\end{equation}
	
	By (\ref{eq: thm upper bnd 2}), (\ref{eq: thm upper bnd 5}) and
	(\ref{eq: thm upper bnd 10}), it follows that on the event $\mathcal{M}$,
	\begin{equation}\label{eq: thm upper bnd 11}
	|\hat{\beta}-\beta|\leq c_{n}.
	\end{equation}
	
	Therefore,  by (\ref{eq: thm upper bnd 2.5}), for any $\theta\in\widetilde{\Theta}(s,\beta_0)$, we have $\EE_{\theta} \psi_* =  \PP_{\theta}(|\hat{\beta}-\beta_0|>c_n)=  \PP_{\theta}(|\hat{\beta}-\beta|>c_n)\leq \alpha $. This proves the first part of Theorem \ref{THM:8}.

	We now show  the second part of Theorem \ref{THM:8}.  It is straight-forward to see that $b_n\asymp n $, $\lambda_{\boldsymbol{\pi}} \asymp \sqrt{n^{-1}\log p} $ and $\eta_{\boldsymbol{\pi}} \asymp sn^{-1}\log p + n^{-1/2} $. Therefore, $c_n \asymp n^{-1/2}+sn^{-1}\log p $. 
	
	Moreover, for any $\theta\in \widetilde{\Theta}(s,\beta_0+3c_n) $, we have that on the event $\mathcal{M}$,
	$$
	|\hat{\beta}-\beta_0| \geq |\beta-\beta_0|-|\hat{\beta}-\beta| = 3c_n-|\hat{\beta}-\beta| \overset{(i)}{\geq} 2c_n>c_n,
	$$
	where $(i)$ follows by (\ref{eq: thm upper bnd 11}). Thus, for any $ \theta\in\widetilde{\Theta}(s,\beta_0+3c_n) $, we have $$ \EE_{\theta}\psi_*=\PP_{\theta}(|\hat{\beta}-\beta_0|>c_n) \geq \PP_{\theta}(\mathcal{M})  \overset{(i)}{\geq} 1-\alpha, $$ where $(i)$ holds by (\ref{eq: thm upper bnd 2.5}).  This proves the second part of Theorem \ref{THM:8}. 
\end{proof}

\section{Proof of Corollary  \ref{THM:1}}
\begin{proof}[\textbf{Proof of Corollary \ref{THM:1}}]

	Let $h_0=\min\{\rho,\tau\}$, where $\rho$ and $\tau$ are defined in Theorems \ref{THM:4} and \ref{THM:5}, respectively. Notice that
	\begin{multline*}
	\Theta_{\zeta,\kappa }(s/2,\beta_{0}+h_0 (n^{-1/2} +sn^{-1}\log p)) \\
	\subset \Theta_{\zeta,\kappa }(s/2,\beta_{0}+\rho sn^{_1}\log p ) \bigcap  \Theta_{\kappa }(s,\beta_{0}+\tau n^{-1/2} ).
	\end{multline*}
	Thus, Theorems \ref{THM:4} and \ref{THM:5} imply
	$$
	\limsup_{n\rightarrow\infty}\sup_{\psi\in\Psi_{\alpha}(\Theta(s,\beta_{0}))}\sup_{\theta\in\Theta_{\zeta,\kappa }(s/2,\beta_{0}+h_0 (n^{-1/2}+sn^{-1}\log p))}\EE_{\theta}\psi \leq  2 \alpha.
	$$
	Hence,
	\begin{equation}\label{eq: Coro 5 eq 1}
	\limsup_{n\rightarrow\infty}\sup_{\psi\in\Psi_{\alpha}(\Theta(s,\beta_{0}))}\inf_{\theta\in\Theta_{\zeta,\kappa }(s/2,\beta_{0}+h_0 (n^{-1/2}+sn^{-1}\log p))}\EE_{\theta}\psi \leq  2 \alpha.
	\end{equation}
	The desired result follows by noticing that $\Psi_{\alpha}(\wtTheta(s,\beta_{0})) \subset \Psi_{\alpha}(\Theta(s,\beta_{0})) $  and $ \Theta_{\zeta,\kappa }(s/2,\beta_{0}+h_n)\subset \widetilde{\Theta}(s,\beta_{0}+h_{n})$ with $h_n=h_0(n^{-1/2}+sn^{-1}\log p)$.

	\section{Proof of  Corollary \ref{COR:2}}
	
	Consider a sequence of ${\CIcal}=[l,u]\in \mathcal{C}_{\alpha}(\Theta(s))$ such
	that
	\[
	\limsup_{n\rightarrow\infty}\inf_{\theta\in\Theta(s)}\EE_{\theta}\mbox{diam}(\mathcal{CI})=\limsup_{n\rightarrow\infty}\inf_{{\CIcal'}\in \mathcal{C}_{\alpha}(\Theta(s))}\sup_{\theta\in\Theta(s)}\EE_{\theta}\mbox{diam}(\mathcal{CI}').
	\]
	Consider the test
	$$\psi=\mathbf{1}\{\beta_{0}\notin {\CIcal}\}$$
	for testing
	$\theta\in\Theta(s,\beta_{0})$. Clearly, $\psi\in\Psi_{\alpha}(\Theta(s,\beta_{0}))$. Consider $\Theta(s,\beta_{0}+h_n)$ with $h_n=h_0(n^{-1/2}+sn^{-1}\log p)$  defined in Corollary \ref{THM:1}.

	Fix any $\theta\in\Theta(s,\beta_{0}+h_n)$.  We have that $\beta=\beta_0+h' $ with $0\leq h'\leq h_n$. Notice that
	\begin{align*}
	&1-\EE_{\theta}\psi   =\PP_{\theta}(\beta_{0}\in {\CIcal})\\
	& =\PP_{\theta}(l\leq\beta_{0}\leq u)\\
	& =\PP_{\theta}(l\leq\beta_{0}\leq u\ {\rm and}\ \beta\in {\CIcal})+\PP_{\theta}(l\leq\beta_{0}\leq u\ {\rm and}\ \beta\notin {\CIcal})\\
	& \overset{(i)}{=}\PP_{\theta}(l\leq\beta_{0}\leq u\ {\rm and}\ \beta_{0}+h'\in {\CIcal})+\PP_{\theta}(l\leq\beta_{0}\leq u\ {\rm and}\ \beta\notin {\CIcal})\\
	& =\PP_{\theta}(\max\{l,l-h'\}\leq\beta_{0}\leq\min\{u,u-h'\})+\PP_{\theta}(l\leq\beta_{0}\leq u\ {\rm and}\ \beta\notin {\CIcal})\\
	& \leq \PP_{\theta}(\max\{l,l-h'\}\leq\min\{u,u-h'\})+\PP_{\theta}(\beta\notin {\CIcal})\\
	& \leq \PP_{\theta}(l\leq u-h')+\alpha\\
	& =\PP_{\theta}(\mbox{diam}(\mathcal{CI})\geq h')+\alpha \leq \PP_{\theta}(\mbox{diam}(\mathcal{CI})\geq h_n)+\alpha
	\end{align*}
	where $(i)$ follows by  $\beta=\beta_0+h' $.
	Hence,
	\[
	\inf_{\theta\in\Theta(s,\beta_{0}+h_n)}\EE_{\theta}\psi\geq1-\alpha-\sup_{\theta\in\Theta(s,\beta_{0}+h_n)}\PP_{\theta}(\mbox{diam}(\mathcal{CI})\geq h_n).
	\]
	
	By (\ref{eq: Coro 5 eq 1}) in the proof of Corollary \ref{THM:1}, we have that
	\begin{align*}
	& \limsup_{n\rightarrow\infty}\sup_{\psi\in\Psi_{\alpha}(\Theta(s,\beta_{0}))}\inf_{\theta\in\Theta(s,\beta_{0}+h_n)}\EE_{\theta}\psi\\
	& \leq \limsup_{n\rightarrow\infty}\sup_{\psi\in\Psi_{\alpha}(\Theta(s,\beta_{0}))}\inf_{\theta\in\Theta(s/2,\beta_{0}+h_n)}\EE_{\theta}\psi\\
	& \leq \limsup_{n\rightarrow\infty}\sup_{\psi\in\Psi_{\alpha}(\Theta(s,\beta_{0}))}\inf_{\theta\in\Theta_{\zeta,\kappa }(s/2,\beta_{0}+h_n)}\EE_{\theta}\psi \leq 2 \alpha.
	\end{align*}

	The above two displays imply
	\[
	\liminf_{n\rightarrow\infty}\sup_{\theta\in\Theta(s,\beta_{0}+h)}\PP_{\theta}(\mbox{diam}(\mathcal{CI})\geq h_n)\geq1-3\alpha.
	\]
	
	The desired result follows by noticing that
	\[
	\mbox{diam}(\mathcal{CI})\geq\mbox{diam}(\mathcal{CI})\mathbf{1}\{\mbox{diam}(\mathcal{CI})\geq h_n\} \geq h_n \mathbf{1}\{\mbox{diam}(\mathcal{CI})\geq h_n\}
	\]
	and thus
	\[
	\sup_{\theta\in\Theta(s,\beta_{0}+h_n)}\EE_{\theta}\mbox{diam}(\mathcal{CI})\geq h_n\sup_{\theta\in\Theta(s,\beta_{0}+h_n)}\PP_{\theta}(\mbox{diam}(\mathcal{CI})\geq h_n).
	\]
\end{proof}

\section{Proof of Theorem  \ref{THM:9}}

\begin{proof}[Proof of Theorem \ref{THM:9}]
	By Theorem \ref{THM:8}, we have
	\[
	\inf_{\CIcal\in\mathcal{C}_{\alpha}(\Theta(s))}\sup_{\theta\in\mathcal{C}_{\alpha}(\Theta(s_{1}))}\EE_{\theta}\mbox{diam}(\mathcal{CI})=O\left(n^{-1/2}+sn^{-1}\log p\right).
	\]
	
	Hence, it suffices to show that
	\begin{equation}
	\liminf_{n\rightarrow \infty}\frac{\inf_{\CIcal\in\mathcal{C}_{\alpha}(\Theta(s))}\sup_{\theta\in\mathcal{C}_{\alpha}(\Theta(s_{1}))}\EE_{\theta}\mbox{diam}(\mathcal{CI})}{\left(n^{-1/2}+sn^{-1}\log p\right)}>0.\label{eq: thm adaptivity eq 1}
	\end{equation}
	
	We proceed by contradiction. Let $h_{n}$ be defined as in Theorem \ref{THM:4}. Fix an arbitrary $\beta_{0}\in\mathbb{R}$. Suppose that there
	exists $\CIcal_{0}=[l_{0},u_{0}]\in\mathcal{C}_{\alpha}(\Theta(s))$ such
	that
	$$\sup_{\theta\in\mathcal{C}_{\alpha}(\Theta(s_{1}))}\EE_{\theta}\mbox{diam}(\mathcal{CI}_0)=\tilde{h}_{n}$$
	with $\liminf_{n\rightarrow\infty}(\tilde{h}_{n}/h_{n})=0$. Define
	$\Delta=\alpha^{-1}\tilde{h}_{n}$. Consider
	$$\psi_{0}=\mathbf{1}\{\beta_{0}\notin \CIcal_{0}\}$$
	as the test for $H_{0}:\ \beta=\beta_{0}$ vs $H_{a}:\ \beta=\beta_{0}+\Delta$.
	
	Notice that
	\[
	\sup_{\theta\in\Theta(\beta_{0},s)}\EE_{\theta}\psi_{0}=\sup_{\theta\in\Theta(\beta_{0},s)}P_{\theta}\left(\beta_{0}\notin \CIcal_{0}\right)\overset{(i)}{\leq}\alpha,
	\]
	where $(i)$ follows by $\CIcal_{0}\in\mathcal{C}_{\alpha}(\Theta(s))$.
	Thus, $\psi_{0}\in\Psi(\Theta(\beta_{0},s))$.
	
	Fix an arbitrary $\theta_{1}\in\Theta_{\zeta,\kappa}(s_{1},\beta_{0}+\Delta)$. Notice
	that on the event
	$$\{\beta_{0}+\Delta\in \CIcal_{0}\}\bigcap\{u_{0}-\Delta<l_{0}\},$$
	we have $\beta_{0}+\Delta\leq u_{0}$, which means $\beta_{0}\leq u_{0}-\Delta<l_{0}$
	and thus $\beta_{0}\notin \CIcal_{0}$. Hence,
	\begin{align*}
	\EE_{\theta_{1}}\psi_{0}=\PP_{\theta_{1}}\left(\beta_{0}\notin \CIcal_{0}\right) & \geq \PP_{\theta_{1}}\left(\{\beta_{0}+\Delta\in \CIcal_{0}\}\bigcap\{u_{0}-\Delta<l_{0}\}\right)\\
	& \geq \PP_{\theta_{1}}\left(\beta_{0}+\Delta\in \CIcal_{0}\right)-\PP_{\theta_{1}}\left(u_{0}-\Delta\geq l_{0}\right)\\
	& \overset{(i)}{\geq}1-\alpha-\PP_{\theta_{1}}\left(u_{0}-l_{0}\geq \Delta\right)\\
	& \overset{(ii)}{\geq}1-\alpha-\frac{\EE_{\theta_{1}}|u_{0}-l_{0}|}{\Delta}\\
	& \overset{(iii)}{\geq}1-\alpha-\frac{\tilde{h}_{n}}{\Delta}\\
	& \overset{(iv)}{=}1-2\alpha
	\end{align*}
	where $(i)$ follows by $\CIcal_{0}\in\mathcal{C}_{\alpha}(\Theta(s))$,
	$(ii)$ follows by Markov's inequality, $(iii)$ follows by the fact
	that $\theta_{1}\in\Theta(\beta_{0}+\Delta,s_{1})$ and $\sup_{\theta\in\mathcal{C}_{\alpha}(\Theta(s_{1}))}\EE_{\theta}\mbox{diam}(\mathcal{CI}_0)\leq\tilde{h}_{n}$
	and $(iv)$ follows by $\Delta=\alpha^{-1}\tilde{h}_{n}$. Consequently, we obtain
	\[
	\limsup_{n\rightarrow\infty}\sup_{\psi\in\Psi_{\alpha}(\Theta(s,\beta_{0}))}\sup_{\theta\in\Theta_{\zeta,\kappa}(s_{1},\beta_{0}+\Delta)}\EE_{\theta}\psi\geq1-2\alpha.
	\]
	
	Since $\Delta\asymp\tilde{h}_{n}=o(h_{n})$ and $s_1 \leq s/2$, we have that $\Theta_{\zeta,\kappa}(s/2,\beta_{0}+h_{n}) $ contains $\Theta_{\zeta,\kappa}(s_{1},\beta_{0}+\Delta)$ for large $n$ and thus
	\begin{align*}
	&\limsup_{n\rightarrow\infty}\sup_{\psi\in\Psi_{\alpha}(\Theta(s,\beta_{0}))}\sup_{\theta\in\Theta_{\zeta,\kappa}(s_{1},\beta_{0}+\Delta)}\EE_{\theta}\psi
	\\
	& \qquad \qquad \qquad \qquad \leq\limsup_{n\rightarrow\infty}\sup_{\psi\in\Psi_{\alpha}(\Theta(s,\beta_{0}))}\sup_{\theta\in\Theta_{\zeta,\kappa}(s/2,\beta_{0}+h_{n})}\EE_{\theta}\psi\overset{(i)}{\leq}\alpha,
	\end{align*}
	where $(i)$ follows by Theorem \ref{THM:4}. The above two displays imply
	that $\alpha\geq1-2\alpha$. This is not possible since $\alpha<1/3$.
	Hence, we have arrived at the contradiction.
	
	Therefore, there does not exist $\CIcal_{0}=[l_{0},u_{0}]\in\mathcal{C}_{\alpha}(\Theta(s))$
	such that $\sup_{\theta\in\mathcal{C}_{\alpha}(\Theta(s_{1}))}\EE_{\theta}\mbox{diam}(\mathcal{CI}_0)=O(\tilde{h}_{n})$
	with $\tilde{h}_{n}=o(h_{n})$. Hence,
	\[
	\liminf_{n\rightarrow\infty}\left(\inf_{\CIcal\in\mathcal{C}_{\alpha}(\Theta(s))}\sup_{\theta\in\mathcal{C}_{\alpha}(\Theta(s_{1}))}\EE_{\theta}\mbox{diam}(\mathcal{CI})\right)/h_{n}>0.
	\]
	
	Similarly using Theorem \ref{THM:5}, we can show that
	\[
	\liminf_{n\rightarrow\infty}\left(\inf_{\CIcal\in\mathcal{C}_{\alpha}(\Theta(s))}\sup_{\theta\in\mathcal{C}_{\alpha}(\Theta(s_{1}))}\EE_{\theta}\mbox{diam}(\mathcal{CI})\right)/(n^{-1/2})>0.
	\]
	
	Therefore, we have proved that the claim in (\ref{eq: thm adaptivity eq 1}).
	The proof is complete.
\end{proof}

\section{Proof of Theorem \ref{THM:10}
}

We rely on the following two lemmas.
\begin{lemma}
	\label{lem: sufficient for non adaptivity}Suppose that points in
	$\Theta$ are uniformly non-testable, i.e., 
	\[
	\inf_{CI\in\mathcal{C}_{\alpha}(\Theta)}\sup_{\theta\in\Theta}\mathbb{E}_{\theta}\mbox{\rm diam}(CI)\asymp\inf_{CI\in\mathcal{C}_{\alpha}(\Theta)}\inf_{\theta\in\Theta}\mathbb{E}_{\theta}\mbox{\rm diam}(CI).
	\]
	
	Then there exists a constant $c>0$ such that $cL(\Theta,\Theta)\leq L(\Theta_{1},\Theta)\leq L(\Theta,\Theta)$
	for any $\Theta_{1}\subseteq\Theta$. 
\end{lemma}

\begin{lemma}
	\label{lem: necessary for non adaptivity}Suppose that there exists
	a constant $c>0$ such that $cL(\Theta,\Theta)\leq L(\Theta_{1},\Theta)\leq L(\Theta,\Theta)$
	for any subset $\Theta_{1}\subseteq\Theta$. Then
	\[
	\inf_{CI\in\mathcal{C}_{\alpha}(\Theta)}\sup_{\theta\in\Theta}\mathbb{E}_{\theta}\mbox{\rm diam}(CI)\asymp\inf_{CI\in\mathcal{C}_{\alpha}(\Theta)}\inf_{\theta\in\Theta}\mathbb{E}_{\theta}\mbox{\rm diam}(CI).
	\]
\end{lemma}

Now we are ready to prove Theorem \ref{THM:10}.
\begin{proof}[Proof of Theorem \ref{THM:10}]
	The result  is simple consequence of the two  Lemmas, Lemma \ref{lem: sufficient for non adaptivity}
	and \ref{lem: necessary for non adaptivity} whose proofs can be found in Section \ref{sec:M}.
\end{proof}

\section{Proof of Corollary \ref{COR:9}}
\begin{proof}[Proof of Corollary \ref{COR:9}] 
	Clearly, 
	\[
	\inf_{CI\in\mathcal{C}_{\alpha}(\Theta)}\sup_{\theta\in\Theta}\mathbb{E}_{\theta}\mbox{\rm diam}(CI)\leq\sup_{\theta\in\Theta}\mathbb{E}_{\theta}\mbox{\rm diam}(CI)\leq c_{2}h_{n}.
	\]
	
	It remains to show that $\inf_{CI\in\mathcal{C}_{\alpha}(\Theta)}\inf_{\theta\in\Theta}\mathbb{E}_{\theta}\mbox{\rm diam}(CI)\gtrsim h_{n}$.
	For that end, we fix an arbitrary $\tau\in\mathbb{R}$, an arbitrary $CI\in\mathcal{C}_{\alpha}(\Theta)$
	as well as an arbitrary $\theta\in\Theta(\tau)$. 
	
	Define a test $\psi=\mathbf{1}\{\tau\notin CI\}$. Clearly, $\psi\in\Psi(\Theta(\tau))$.
	Let $[l,u]=CI$. Since $g(\theta)=\tau+c_{1}h_{n}$ for $\theta\in\Theta(\tau+c_{1}h_{n})$
	and $CI\in\mathcal{C}_{\alpha}(\Theta)$, we have that for any $\theta\in\Theta(\tau+c_{1}h_{n})$
	\[
	\mathbb{P}_{\theta}\left(l\leq\tau+c_{1}h_{n}\leq u\right)\geq1-\alpha.
	\]
	
	By assumption, 
	\[
	\mathbb{P}_{\theta}\left(\left\{ \tau<l\right\} \bigcup\left\{ \tau>u\right\} \right)=\mathbb{E}_{\theta}\psi\leq2\alpha.
	\]
	
	Let $\mathcal{M}=\{l\leq\tau+c_{1}h_{n}\leq u\}\bigcap\{l\leq\tau\leq u\}$.
	Clearly, $\mathbb{P}_{\theta}(\mathcal{M})\geq1-3\alpha$. 
	
	Notice that on the event $\mathcal{M}$, $l\leq\tau\leq u-c_{1}h_{n}$,
	which means $u-l\geq c_{1}h_{n}$. It follows that 
	\[
	\mathbb{E}_{\theta}\mbox{\rm diam}(CI)\geq\mathbb{E}_{\theta}\mbox{\rm diam}(CI)\times\mathbf{1}\{\mathcal{M}\}\geq c_{1}h_{n}\mathbb{P}_{\theta}(\mathcal{M})\geq(1-3\alpha)c_{1}h_{n}.
	\]
	
	Notice that the above bound holds for any $\theta\in\Theta(\tau)$
	with any $\tau\in\mathbb{R}$. Hence, 
	\[
	\inf_{\theta\in\Theta}\mathbb{E}_{\theta}\mbox{\rm diam}(CI)\geq(1-3\alpha)c_{1}h_{n}.
	\]
	
	Since the above bound holds for any $CI\in\mathcal{C}_{\alpha}(\Theta)$,
	we have 
	$$\inf_{CI\in\mathcal{C}_{\alpha}(\Theta)}\inf_{\theta\in\Theta}\mathbb{E}_{\theta}\mbox{\rm diam}(CI)\geq(1-3\alpha)c_{1}h_{n}.$$
	
\end{proof}

\section{Proof of Theorem \ref{THM:12}}

For $\theta=(\beta,\boldsymbol{\gamma},\boldsymbol{\Sigma},\sigma)$
and $Q>0$, we denote $\theta\odot Q=(\beta Q,\boldsymbol{\gamma}Q,\boldsymbol{\Sigma},\sigma Q)$.
For any $C\subseteq\mathbb{R}$ and $Q>0$, we define $Q\cdot C=\{Qx:\ x\in C\}$. 
\begin{lemma}
	\label{lem: scaling para space}For any $Q,N_{1},N_{2}>0$, 
	$$\widetilde{\Theta}_{QN_{1},QN_{2}}(s)=\{\theta\odot Q:\ \theta\in\widetilde{\Theta}_{N_{1},N_{2}}(s)\}.$$ 
\end{lemma}

\begin{lemma}
	\label{lem: obs equi}For any $D,N_{1},N_{2}>0$, let $\theta\in\widetilde{\Theta}_{N_{1},N_{2}}(s)$.
	Then $(\mathbf{y},\mathbf{Z},\mathbf{W})\sim(\theta\odot D)$ if and
	only if $(\mathbf{y}D^{-1},\mathbf{Z},\mathbf{W})\sim\theta$. 
\end{lemma}
\begin{lemma}
	\label{lem: scale minimax}For any $D,N_{1},N_{2}>0$, 
	$$D{\mathbb A}(s,N_{1},N_{2})\geq {\mathbb A}(s,DN_{1},DN_{2}).$$
\end{lemma}

\begin{proof}[Proof of Theorem \ref{THM:12}]
	By Lemma \ref{lem: scale minimax} with $(D,N_{1},N_{2})=(Q,M_{1},M_{2})$,
	we have that $Q{\mathbb A}(s,M_{1},M_{2})\geq {\mathbb A}(s,QM_{1},QM_{2})$.
	
	We now apply Lemma \ref{lem: scale minimax} with $(D,N_{1},N_{2})=(Q^{-1},QM_{1},QM_{2})$,
	obtaining $Q^{-1}{\mathbb A}(s,QM_{1},QM_{2})\geq {\mathbb A}(s,M_{1},M_{2})$. The desired
	result follows. 
\end{proof}

\section{Proof of auxiliary lemmas}

\subsection{Proof of auxiliary lemmas used in proving Theorem \ref{THM:4}}

\begin{proof}[Proof of Lemma \ref{lem:1}]
	Let
	\[
	A_{k}=\left[1-kan^{-1}\log p\right]^{-n}\frac{\begin{pmatrix}m\\
		k
		\end{pmatrix}\begin{pmatrix}p-m-1\\
		m-k
		\end{pmatrix}}{\begin{pmatrix}p-1\\
		m
		\end{pmatrix}}.
	\]
	
	Notice that for $0\leq k\leq m$,
	\begin{align}
	\log\frac{A_{k+1}}{A_{k}} & =\log\left[\left(1-\frac{an^{-1}\log p}{1-kan^{-1}\log p}\right)^{-n}\frac{(m-k)^{2}}{(k+1)(p-2m+k)}\right] \nonumber \\
	& =-n\log\left(1-\frac{an^{-1}\log p}{1-kan^{-1}\log p}\right)+\log\frac{(m-k)^{2}}{(k+1)(p-2m+k)}\nonumber \\
	& \leq -n\log\left(1-\frac{an^{-1}\log p}{1-kan^{-1}\log p}\right)+\log\frac{(m-k)^{2}}{p-2m+k}\nonumber \\
	& \overset{(i)}{\leq}-n\log\left(1-2an^{-1}\log p\right)+\log\frac{p^{2c}}{p-2p^{c}}\nonumber \\
	& \overset{(ii)}{\leq}\frac{2a\log p}{1-2an^{-1}\log p}+\log\frac{p^{2c}}{p-2p^{c}}\nonumber \\
	& \overset{(iii)}{<}4a\log p+\log\frac{p^{2c}}{p-2p^{c}}\nonumber \\
	& = \log\frac{p^{4a+2c-1}}{1-2p^{c-1}}, \label{eq: hypergeo bound eq 0.5}
	\end{align}
	where $(i)$ follows by the fact that
	$$1-kan^{-1}\log p\geq1-man^{-1}\log p\geq1-a/4\geq1/2,$$
	$(ii)$ follows by the fact that $\log(1-x)\geq x/(x-1)$ for any
	$x\in(0,1)$ and $2an^{-1}\log p\in (0,1)$ (due to $2an^{-1}\log p\leq2a/(4m)\leq a/2<1/2$) and $(iii)$ follows by $2an^{-1}\log p<1/2$.

	Notice that $4a+2c-1<0$ and $c-1<0$. Hence, for large $p$, $\log(A_{k+1}/A_{k})\leq-\log2$ for any $0\leq k\leq m$.
	It follows that for large $p$,
	\begin{equation}
	\sum_{k=0}^{m}A_{k}=A_{0}+\sum_{k=1}^{m}A_{k}\leq A_{0}+A_{1}\sum_{k=1}^{m}2^{-k}\leq A_{0}+2A_{1}\label{eq: hypergeo bound eq 1}
	\end{equation}
	
	Notice that
	\[
	A_{0}=\frac{\begin{pmatrix}p-m-1\\
		m
		\end{pmatrix}}{\begin{pmatrix}p-1\\
		m
		\end{pmatrix}}=\prod_{j=0}^{m-1}\frac{p-2m+j}{p-m+j}=\prod_{j=0}^{m-1}\left(1-\frac{m}{p-m+j}\right).
	\]
	
	Hence,
	\[
	\left(1-\frac{m}{p-m}\right)^{m}\leq A_{0}\leq\left(1-\frac{m}{p}\right)^{m}
	\]
	
	Since $m^2/p\leq p^{2c-1} \rightarrow0$, both sides tend to 1 and thus $A_{0}\rightarrow1$.
	To bound $A_{1}$, notice that (\ref{eq: hypergeo bound eq 0.5}) implies
	$$
	A_1 \leq \frac{p^{4a+2c-1}}{1-2p^{c-1}} A_0\overset{(i)}{=} o(A_0),
	$$
	where $(i)$ follows by $4a+2c-1<0$ and $c<1$. Hence, $A_1=o(1)$. In light of  (\ref{eq: hypergeo bound eq 1}), the desired result follows.
\end{proof}

\begin{proof} [Proof of Lemma \ref{lem: chi2 distance gaussian}]
	Notice that
	\[
	\EE_{g_{0}}\left(\frac{d\PP_{g_{1}}}{d\PP_{g_{0}}}\times\frac{d\PP_{g_{2}}}{d\PP_{g_{0}}}\right)=\int_{\mathbb{R}^{k}}\frac{g_{1}(x)g_{2}(x)}{g_{0}(x)}dx.
	\]
	
	By Lemma 11 in \citet{cai2017confidence}, we have
	\begin{align*}
	&\int_{\mathbb{R}^{k}}\frac{g_{1}(x)g_{2}(x)}{g_{0}(x)}dx
	\\
	& =\frac{1}{\sqrt{\det\left(\II_{k}-\Sigmab_{0}^{-1}\left[\Sigmab_{1}-\Sigmab_{0}\right]\Sigmab_{0}^{-1}\left[\Sigmab_{2}-\Sigmab_{0}\right]\right)}}\\
	& =\frac{1}{\sqrt{\det\left(\II_{k}-(L_{0}^{-1})^{\top}L_{0}^{-1}\left[L_{1}L_{1}^{\top}-L_{0}L_{0}^{\top}\right](L_{0}^{-1})^{\top}L_{0}^{-1}\left[L_{2}L_{2}^{\top}-L_{0}L_{0}^{\top}\right]\right)}}\\
	& =\frac{1}{\sqrt{\det\left(\II_{k}-(L_{0}^{-1})^{\top}\left[Q_{1}Q_{1}^{\top}-\II_{k}\right]\left[Q_{2}Q_{2}^{\top}-\II_{k}\right]L_{0}^{\top}\right)}}\\
	& =\frac{1}{\sqrt{\det\left(\II_{k}-(L_{0}^{-1})^{\top}\left[Q_{1}Q_{1}^{\top}-\II_{k}\right]\left[Q_{2}Q_{2}^{\top}-\II_{k}\right]L_{0}^{\top}\right)}}\\
	& =\frac{1}{\sqrt{\det\left\{ (L_{0}^{-1})^{\top}\left(\II_{k}-\left[Q_{1}Q_{1}^{\top}-\II_{k}\right]\left[Q_{2}Q_{2}^{\top}-\II_{k}\right]\right)L_{0}^{\top}\right\} }}\\
	& =\frac{1}{\sqrt{\det\left[(L_{0}^{-1})^{\top}\right]\det\left(\II_{k}-\left[Q_{1}Q_{1}^{\top}-\II_{k}\right]\left[Q_{2}Q_{2}^{\top}-\II_{k}\right]\right)\det\left(L_{0}^{\top}\right)}}\\
	& =\frac{1}{\sqrt{\det\left(\II_{k}-\left[Q_{1}Q_{1}^{\top}-\II_{k}\right]\left[Q_{2}Q_{2}^{\top}-\II_{k}\right]\right)}}.
	\end{align*}
\end{proof}

\begin{proof}[Proof of Lemma \ref{lem: compute Qj step 1}]
	We first derive some preliminary results and then compute
	$$\det\left(\II_{p+1}-\left[Q_{j_{1}}Q_{j_{1}}^{\top}-\II_{p+1}\right]\left[Q_{j_{2}}Q_{j_{2}}^{\top}-\II_{p+1}\right]\right).$$

	\textbf{Step 1:}  First we derive the form of the matrix  $Q_{j}Q_{j}^{\top}-\II_{p}$ for $ 1\leq j\leq N$.
	
	By straight-forward computation, we can verify that
	\[
	L_{\theta_{*}}^{-1}=\begin{pmatrix}\II_{p-1} & 0 & 0\\
	-\sigma_{\Vb,*}^{-1}\pib_{*}^{\top} & \sigma_{\Vb,*}^{-1} & 0\\
	-\sigma_{\varepsilon,*}^{-1}\gammab_{*}^{\top} & -\beta_{*}\sigma_{\varepsilon,*}^{-1} & \sigma_{\varepsilon,*}^{-1}
	\end{pmatrix}.
	\]
	
	Thus,
	\begin{align*}
	Q_{j}&=L_{\theta_{*}}^{-1}L_{\theta_{j}}
	\\
	& = \begin{pmatrix}\II_{p-1} & 0 & 0\\
	-\sigma_{\Vb,*}^{-1}\pib_{*}^{\top} & \sigma_{\Vb,*}^{-1} & 0\\
	-\sigma_{\varepsilon,*}^{-1}\gammab_{*}^{\top} & -\beta_{*}\sigma_{\varepsilon,*}^{-1} & \sigma_{\varepsilon,*}^{-1}
	\end{pmatrix}  \begin{pmatrix}\II_{p-1} & 0 & 0\\
	\pib_{(j)}^{\top} & \sigma_{\Vb,0} & 0\\
	(\pib_{(j)}\beta_{0}+\gammab_{(j)})^{\top} & \beta_{0}\sigma_{\Vb,0} & \sigma_{\varepsilon,0}
	\end{pmatrix}\\
	& =\begin{pmatrix}\II_{p-1} & 0 & 0\\
	\sigma_{\Vb,*}^{-1}(\pib_{(j)}-\pib_{*})^{\top} & \sigma_{\Vb,*}^{-1}\sigma_{\Vb,0} & 0\\
	\sigma_{\varepsilon,*}^{-1}\left[\gammab_{(j)}-\gammab_{*}+(\beta_{0}-\beta_{*})\pib_{(j)}\right]^{\top} & -h\sigma_{\varepsilon,*}^{-1}\sigma_{\Vb,0} & \sigma_{\varepsilon,*}^{-1}\sigma_{\varepsilon,0}
	\end{pmatrix}\\
	& \overset{(i)}{=} \begin{pmatrix}\II_{p-1} & 0 & 0\\
	a_2 \deltab_{(j)}^{\top} & \sigma_{\Vb,*}^{-1}\sigma_{\Vb,0} & 0\\
	a_1 \deltab_{(j)}^{\top} & -h\sigma_{\varepsilon,*}^{-1}\sigma_{\Vb,0} & \sigma_{\varepsilon,*}^{-1}\sigma_{\varepsilon,0}
	\end{pmatrix}
	\end{align*}
	for $a_1=r(1-h)\sqrt{h/m}$ and $ a_2=\sqrt{h/m}$, where $(i)$ follows by Definition \ref{def: prior dist step 1}.
	Since $\deltab_{(j)}^{\top}\deltab_{(j)}=m$, we have
	\begin{alignat}{1}
	& Q_{j}Q_{j}^{\top}-\II_{p+1}
	\nonumber \\
	& \qquad=\begin{pmatrix}\II_{p-1} & 0 & 0\\
	a_2\deltab_{(j)}^{\top} & \sigma_{\Vb,*}^{-1}\sigma_{\Vb,0} & 0\\
	a_1\deltab_{(j)}^{\top} & -h\sigma_{\varepsilon,*}^{-1}\sigma_{\Vb,0} & \sigma_{\varepsilon,*}^{-1}\sigma_{\varepsilon,0}
	\end{pmatrix}\begin{pmatrix}\II_{p-1} & a_2\deltab_{(j)} & a_1\deltab_{(j)}\\
	0 & \sigma_{\Vb,*}^{-1}\sigma_{\Vb,0} & -h\sigma_{\varepsilon,*}^{-1}\sigma_{\Vb,0}\\
	0 & 0 & \sigma_{\varepsilon,*}^{-1}\sigma_{\varepsilon,0}
	\end{pmatrix}-\II_{p+1}\nonumber \\
	&\qquad  \overset{(i)}{=}\begin{pmatrix}0 & a_{2}\deltab_{(j)} & a_{1}\deltab_{(j)}\\
	a_{2}\deltab_{(j)}^{\top} & 0 & 0\\
	a_{1}\deltab_{(j)}^{\top} & 0 & 0
	\end{pmatrix}.\label{eq: compute Qj eq 1}
	\end{alignat}
	where $(i)$ follows by Definition \ref{def: prior dist step 1} and
	the definitions of $a_{1}$ and $a_{2}$.
	
	\textbf{Step 2:} Compute $\det\left(\II_{p+1}-\left[Q_{j_{1}}Q_{j_{1}}^{\top}-\II_{p+1}\right]\left[Q_{j_{2}}Q_{j_{2}}^{\top}-\II_{p+1}\right]\right)$
	for any $j_{1},j_{2}\in\{1,...,N\}$.
	
	From Step 1, we have that for any $j_{1},j_{2}\in\{1,...,M\}$,
	\begin{align*}
	& \II_{p+1}-\left(Q_{j_{1}}Q_{j_{1}}^{\top}-\II_{p+1}\right)\left(Q_{j_{2}}Q_{j_{2}}^{\top}-\II_{p+1}\right)\\
	& =\II_{p+1}-\begin{pmatrix}0 & a_{2}\deltab_{(j_{1})} & a_{1}\deltab_{(j_{1})}\\
	a_{2}\deltab_{(j_{1})}^{\top} & 0 & 0\\
	a_{1}\deltab_{(j_{1})}^{\top} & 0 & 0
	\end{pmatrix}\begin{pmatrix}0 & a_{2}\deltab_{(j_{2})} & a_{1}\deltab_{(j_{2})}\\
	a_{2}\deltab_{(j_{2})}^{\top} & 0 & 0\\
	a_{1}\deltab_{(j_{2})}^{\top} & 0 & 0
	\end{pmatrix}.\\
	& =\begin{pmatrix}\II_{p-1}-(a_{1}^{2}+a_{2}^{2})\deltab_{(j_{1})}\deltab_{(j_{2})}^{\top} & 0 & 0\\
	0 & 1-a_{2}^{2}\deltab_{(j_{1})}^{\top}\deltab_{(j_{2})} & -a_{1}a_{2}\deltab_{(j_{1})}^{\top}\deltab_{(j_{2})}\\
	0 & -a_{1}a_{2}\deltab_{(j_{1})}^{\top}\deltab_{(j_{2})} & 1-a_{1}^{2}\deltab_{(j_{1})}^{\top}\deltab_{(j_{2})}
	\end{pmatrix}.
	\end{align*}
	
	Since this is a block-diagonal matrix, the desired result follows
	by simple computation
	\begin{align*}
	& \det\left[\II_{p+1}-\left(Q_{j_{1}}Q_{j_{1}}^{\top}-\II_{p+1}\right)\left(Q_{j_{2}}Q_{j_{2}}^{\top}-\II_{p+1}\right)\right]\\
	& =\det\left(\II_{p-1}-(a_{1}^{2}+a_{2}^{2})\deltab_{(j_{1})}\deltab_{(j_{2})}^{\top}\right)\det\begin{pmatrix}1-a_{2}^{2}\deltab_{(j_{1})}^{\top}\deltab_{(j_{2})} & -a_{1}a_{2}\deltab_{(j_{1})}^{\top}\deltab_{(j_{2})}\\
	-a_{1}a_{2}\deltab_{(j_{1})}^{\top}\deltab_{(j_{2})} & 1-a_{1}^{2}\deltab_{(j_{1})}^{\top}\deltab_{(j_{2})}
	\end{pmatrix}\\
	& \overset{(i)}{=} \left[1-(a_{1}^{2}+a_{2}^{2})\deltab_{(j_{1})}^{\top}\deltab_{(j_{2})}\right] \det\begin{pmatrix}1-a_{2}^{2}\deltab_{(j_{1})}^{\top}\deltab_{(j_{2})} & -a_{1}a_{2}\deltab_{(j_{1})}^{\top}\deltab_{(j_{2})}\\
	-a_{1}a_{2}\deltab_{(j_{1})}^{\top}\deltab_{(j_{2})} & 1-a_{1}^{2}\deltab_{(j_{1})}^{\top}\deltab_{(j_{2})}
	\end{pmatrix}\\
	& =\left[1-(a_{1}^{2}+a_{2}^{2})\deltab_{(j_{1})}^{\top}\deltab_{(j_{2})}\right]^{2},
	\end{align*}
	where $(i)$ follows by the Sylvester's determinant identity. The desired result follows by the definitions of $a_1 $ and $a_2 $.  \end{proof}

\begin{proof}[Proof of Lemma \ref{lem: chi2 distance bound step 1}]
	Recall all the notations in Lemma \ref{lem: compute Qj step 1} and $\rho$ defined in (\ref{eq: def rho}). Notice that
	\begin{align*}
	&\EE_{\theta_{*}}\left(N^{-1}\sum_{j=1}^{N}\frac{d\PP_{\theta_{j}}}{d\PP_{\theta_{*}}}-1\right)^{2}
	\\
	& =N^{-2}\sum_{j_{2}=1}^{N}\sum_{j_{1}=1}^{N}\EE_{\theta_{*}}\left(\frac{dP_{\theta_{j_{1}}}}{dP_{\theta_{*}}}\times\frac{dP_{\theta_{j_{2}}}}{dP_{\theta_{*}}}\right)-1\\
	& \overset{(i)}{=}N^{-2}\sum_{j_{2}=1}^{N}\sum_{j_{1}=1}^{N}\left[1-m^{-1}h[r^2 (1-h)^2+1] \deltab_{(j_{1})}^{\top}\deltab_{(j_{2})}\right]^{-n}-1\\
	& \overset{(ii)}{=}N^{-1}\sum_{j=1}^{N}\left[1-m^{-1}h[r^2 (1-h)^2+1]\deltab_{(1)}^{\top}\deltab_{(j)}\right]^{-n}-1,
	\end{align*}
	where $(i)$ follows by Lemmas \ref{lem: chi2 distance gaussian}
	and \ref{lem: compute Qj step 1} (since there are $n$ i.i.d observations, likelihood is a simple product) and $(ii)$ follows by observing
	that
	$$\sum_{j_{1}=1}^{N}\left[1-m^{-1}h[r^2 (1-h)^2+1]\deltab_{(j_{1})}^{\top}\deltab_{(j_{2})}\right]^{-n}$$
	does not depend on $j_{2}$. To see this, simply notice that $\{\deltab_{(j)}^{\top}\deltab_{(j_2)}\}_{1\leq j\leq N} $ is a permutation of $\{\deltab_{(j)}^{\top}\deltab_{(1)}\}_{1\leq j\leq N} $ for any $1\leq j_2\leq N$.
	
	For $k\in\{0,1,...,m\}$, let
	$$S_{k}=\{j\in\{1,...,N\}:\ \deltab_{(1)}^{\top}\deltab_{(j)}=k\}.$$
	Notice that the cardinality of $S_{k}$ is $\begin{pmatrix}m\\
	k
	\end{pmatrix}\begin{pmatrix}p-m-1\\
	m-k
	\end{pmatrix}$. Recall that $N=\begin{pmatrix}p-1\\
	m
	\end{pmatrix}$. It follows that
	\begin{align*}
	&\EE_{\theta_{*}}\left(N^{-1}\sum_{j=1}^{N}\frac{d\PP_{\theta_{j}}}{d\PP_{\theta_{*}}}-1\right)^{2}
	\\
	&=\sum_{k=0}^{m}\left[1-m^{-1}h[r^2 (1-h)^2+1]k\right]^{-n}\frac{\begin{pmatrix}m\\
		k
		\end{pmatrix}\begin{pmatrix}p-m-1\\
		m-k
		\end{pmatrix}}{\begin{pmatrix}p-1\\
		m
		\end{pmatrix}}-1.
	\end{align*}
	
	By Lemma \ref{lem:1}, it suffices to verify that we can
	choose $a\in (0,(1-2c)/4) $ such
	that  $ m^{-1}h[r^2 (1-h)^2+1] \leq a n^{-1}\log p $. We now verify the stronger condition of
	$$ \frac{m^{-1}h[r^2 (1-h)^2+1]}{n^{-1}\log p}<(1-2c)/5.$$ To this end, we
	recall $h=d s n^{-1}\log p $, $0\leq d \leq \rho $ and $s/2-1 <  m\leq s/2 $ from Definition \ref{def: prior dist step 1}. Since $m\geq 1 $, we have $s/m\leq (2m+1)/m\leq 3$. Now we observe that
	\begin{align*}
	\frac{m^{-1}h[r^{2}(1-h)^{2}+1]}{n^{-1}\log p} & =\frac{\left(m^{-1}d sn^{-1}\log p\right)[r^{2}(1-h)^{2}+1]}{n^{-1}\log p}\\
	& \leq m^{-1}\rho s[r^{2}(1-h)^{2}+1]\\
	& \overset{(i)}{\leq}3\rho[r^{2}(1-h)^{2}+1]\\
	& =3\rho[r^{2}(1-d sn^{-1}\log p)^{2}+1]\\
	& \overset{(ii)}{\leq}3\rho[r^{2}+1]\\
	& \overset{(iii)}{\leq}3\rho\left[\kappa ^{-2}M+1\right]\\
	& \overset{(iv)}{\leq}(1/2-c)/5,
	\end{align*}
	where $(i)$ follows by $s/m\leq 3$, $(ii)$ follows by $d sn^{-1}\log p\leq 1$ (due to $sn^{-1}\log p\leq1/4$
	and $0\leq d \leq\rho\leq 4$), $(iii)$ follows by $r\leq \sqrt{M}/\kappa $
	(since $r=\sigma_{\Vb,*}/\sigma_{\varepsilon,*}$, $\sigma_{\Vb,*}^{2}\leq M$
	and $\sigma_{\varepsilon,*}\geq\kappa $) and $(v)$ follows
	by the definition of $\rho$. The proof is complete.
\end{proof}

\begin{proof}[Proof of Lemma \ref{lem: eligibility candidate null dist step 1}]
	Recall that from Lemma \ref{lem:6},   we can write
	$\theta_{*}=(\beta_{*},\gammab_{*},\Sigmab_{*},\sigma_{\varepsilon,*})\in \Theta$ using
	$$\Sigmab_{*}=\begin{pmatrix}\pib_{*}^{\top}\pib_{*}+\sigma_{\Vb,*}^{2} & \pib_{*}^{\top}\\
	\pib_{*} & \II_{p-1}
	\end{pmatrix}.
	$$
	
	Since $\theta_{*}\in\Theta_{\zeta,\kappa}(m,\beta_{0}+h_n)$,
	we have (1) $\beta_*=\beta_0+h$ with  $h=d sn^{-1}\log p$ and $0\leq d \leq \rho$ and (2) $\lambda_{\max}(\Sigmab_{*})\leq\zeta M<M$. Notice
	that $\pib_{*}^{\top}\pib_{*}+\sigma_{\Vb,*}^{2}\leq\lambda_{\max}(\Sigmab_{*})$.
	Hence,
	\begin{equation}
	\max\left\{ \|\pib_{*}\|_{2},\ \sigma_{\Vb,*}\right\} \leq \sqrt{M}.\label{eq: eligibility cand null eq 1}
	\end{equation}
	Recall $r=\sigmaVstar/\sigmaepsstar $. By the definition of $\Theta_{\zeta,\kappa}(s,\beta_0+h_n)$, we have
	\begin{equation} \label{eq: eligibility cand null eq 1.1}
	r\leq \sqrt{M}/\kappa .
	\end{equation}
	
	The rest of the proof proceeds in four steps, where we verify that
	
	(1) $ \sigma_{\varepsilon,0}\leq M_{1}$,
	
	(2) $\|(\Sigmab_{(j)}^{-1})_{,1}\|_{0}\leq2m$ and
	
	(3) $M^{-1}\leq\lambda_{\min}(\Sigmab_{(j)})\leq\lambda_{\max}(\Sigmab_{(j)})\leq M$.
	
	(4) $\beta_{0}^{2}+\|\boldsymbol{\gamma}_{(j)}\|_{2}^{2}\leq\zeta^{2}M_{2}^{2}$.
	\vskip 10pt
	
	
	
	\vskip 5pt
	\textbf{Step 1:} Show $\sigma_{\varepsilon,0}\leq M_{1}$.
	
	Notice that
	$$
	\sigma_{\varepsilon,0} =\sigma_{\varepsilon,*}\sqrt{1-h r^2+h^2 r^2}  \leq \zeta M_{1}\sqrt{1-h r^2+h^2 r^2} \overset{(i)}{\leq} \zeta M_{1}<M_1,
	$$
	where $(i)$ $-hr^2+h^2r^2 \leq 0$ (since $0\leq h \leq \rho sn^{-1}\log p\leq \rho/4\leq 1 $).

	\vskip 10pt
	\textbf{Step 2:} Show $\|(\Sigmab_{(j)}^{-1})_{,1}\|_{0}\leq2m$.
	
	Observe that $(\Sigmab_{(j)}^{-1})_{,1}=\begin{pmatrix}1\\
	-\pib_{(j)}
	\end{pmatrix}\sigma_{\Vb,0}^{-2}$ and $(\Sigmab_{*}^{-1})_{,1}=\begin{pmatrix}1\\
	-\pib_{*}
	\end{pmatrix}\sigma_{\Vb,*}^{-2}$. Hence,
	$$\|(\Sigmab_{(j)}^{-1})_{,1}\|_{0}=\|\pib_{(j)}\|_{0}+1$$ and
	$\|(\Sigmab_{*}^{-1})_{,1}\|_{0}=\|\pib_{*}\|_{0}+1$. Since
	$$\|\pib_{(j)}\|_{0}\leq\|\pib_{*}\|_{0}+\|\deltab_{(j)}\|_{0}=\|\pib_{*}\|_{0}+m$$
	and $\theta_{*}\in\Theta_{\zeta,\kappa}(m,\beta_{0}+h_n)$,
	we have
	$$\|(\Sigmab_{(j)}^{-1})_{,1}\|_{0}\leq\|(\Sigmab_{*}^{-1})_{,1}\|_{0}+m\leq2m.$$
	
	\vskip 10pt
	\textbf{Step 3:} Show $M^{-1}\leq\lambda_{\min}(\Sigmab_{(j)})\leq\lambda_{\max}(\Sigmab_{(j)})\leq M$.
	
	Since $2m\leq s \leq 2m+1$ and $s\geq 2$, we have $m\geq 1$ and
	$$2 \leq s/m\leq 2+1/m \leq 3.$$ Notice that $\|\deltab_{(j)}\|_{2}=\sqrt{m}$ and
	\begin{equation}
	\|\pib_{(j)}-\pib_{*}\|_{2}=\sigmaVstar \sqrt{h/m} \|\deltab_{(j)}\|_{2}   =\sigmaVstar \sqrt{h}.\label{eq: norm bound Delta pi}
	\end{equation}
	
	Let $\|\cdot\|_{F} $ denote the Frobenius norm and observe that
	\begin{align*}
	\|\Sigmab_{(j)}-\Sigmab_{*}\|_{F}^{2} & =\left(\pib_{(j)}^{\top}\pib_{(j)}-\pib_{*}^{\top}\pib_{*}+\sigma_{\Vb,0}^{2}-\sigma_{\Vb,*}^{2}\right)^{2}+2\|\pib_{(j)}-\pib_{*}\|_{2}^{2}\\
	& \overset{(i)}{=}\left(\|\pib_{(j)}-\pib_{*}\|_{2}^{2}+2(\pib_{(j)}-\pib_{*})^{\top}\pib_{*}-h\sigma_{\Vb,*}^{2}\right)^{2}
	\\
	& \qquad \qquad +2\|\pib_{(j)}-\pib_{*}\|_{2}^{2}\\
	& \overset{(ii)}{=}\left(2(\pib_{(j)}-\pib_{*})^{\top}\pib_{*}\right)^{2}+2\sigmaVstar^2 h \\
	& \leq\left(2\|\pib_{(j)}-\pib_{*}\|_{2}\times\|\pib_{*}\|_{2}\right)^{2} + 2\sigmaVstar^2 h\\
	& \overset{(iii)}{\leq}\left(2\sigmaVstar\sqrt{hM}\right)^{2}+2\sigmaVstar^2 h\\
	& \overset{(iv)}{\leq}\left(2\sqrt{h}M\right)^{2}+2Mh\\
	& \overset{(v)}{\leq} M(2M+1) \rho/2\\
	& \overset{(vi)}{\leq}\min\left\{ \frac{1}{M^{2}}\left(\frac{1}{\zeta}-1\right)^{2},\ M^{2}(1-\zeta)^{2}\right\} ,
	\end{align*}
	where $(i)$ follows by $\sigma_{\Vb,0}^{2}-\sigma_{\Vb,*}^{2}=-\sigmaVstar^2 h $ (due to Definition \ref{def: prior dist step 1}), $(ii)$ follows by (\ref{eq: norm bound Delta pi}), $(iii)$ follows by (\ref{eq: norm bound Delta pi}), $(iv)$ follows by (\ref{eq: eligibility cand null eq 1}), $(v)$ follows by $h\leq \rho/4$ (due to $h=d s n^{-1}\log p $ with $0\leq d\leq \rho$ and $s n^{-1}\log p \leq 1/4 $) and $(vi)$ follows by $0\leq d\leq \rho$ and the definition of $\rho $ in (\ref{eq: def rho}).
	
	Let $\|\cdot\|$ denote the spectral norm of a matrix (i.e., $\|A\|=\sqrt{\lambda_{\max}(A^{\top}A)}$).
	Notice that
	$$\lambda_{\min}(\Sigmab_{(j)})\geq\lambda_{\min}(\Sigmab_{*})-\|\Sigmab_{(j)}-\Sigmab_{*}\|$$
	and $\lambda_{\max}(\Sigmab_{(j)})\leq\lambda_{\max}(\Sigmab_{*})+\|\Sigmab_{(j)}-\Sigmab_{*}\|$.
	Since $\|\Sigmab_{(j)}-\Sigmab_{*}\|\leq\|\Sigmab_{(j)}-\Sigmab_{*}\|_{F}$,
	the above display implies that
	\begin{align*}
	\lambda_{\min}(\Sigmab_{(j)})&\geq\lambda_{\min}(\Sigmab_{*})-\min\left\{ \frac{1}{M}\left(\frac{1}{\zeta}-1\right),\ M(1-\zeta)\right\}
	\\
	& \geq\lambda_{\min}(\Sigmab_{*})-\frac{1}{M}\left(\frac{1}{\zeta}-1\right)
	\end{align*}
	and  similarly
	\begin{align*}
	\lambda_{\max}(\Sigmab_{(j)})
	& \leq\lambda_{\max}(\Sigmab_{*})+\min\left\{ \frac{1}{M}\left(\frac{1}{\zeta}-1\right),\ M(1-\zeta)\right\}
	\\
	&\leq\lambda_{\max}(\Sigmab_{*})+M(1-\zeta).
	\end{align*}
	
	Since $(\zeta M)^{-1}\leq\lambda_{\min}(\Sigmab_{*})\leq\lambda_{\max}(\Sigmab_{*})\leq\zeta M$,
	we obtain
	$$M^{-1}\leq\lambda_{\min}(\Sigmab_{(j)})\leq\lambda_{\max}(\Sigmab_{(j)})\leq M$$.

	\vskip 10pt
	\textbf{Step 4:} Show $\beta_{0}^{2}+\|\boldsymbol{\gamma}_{(j)}\|_{2}^{2}\leq\zeta^{2}M_{2}^{2}$.

	Since $\theta_{*}\in\Theta_{\zeta,\kappa}(m,\beta_{0}+h_{n})$, we
	have that
	\begin{equation}
	(\beta_{0}+h)^{2}+\|\boldsymbol{\gamma}_{*}\|_{2}^{2}\leq\zeta^{2}M_{2}^{2}.\label{eq: null restriction 23}
	\end{equation}
	
	Therefore, we need to show that
	\begin{equation}
	\left[\beta_{0}^{2}+\|\boldsymbol{\gamma}_{(j)}\|_{2}^{2}\right]-\left[(\beta_{0}+h)^{2}+\|\boldsymbol{\gamma}_{*}\|_{2}^{2}\right]\leq(1-\zeta^{2})M_{2}^{2}.\label{eq: null restriction 24}
	\end{equation}
	
	Let $\boldsymbol{\delta}_{\boldsymbol{\gamma},j}=\boldsymbol{\gamma}_{(j)}-\boldsymbol{\gamma}_{*}$.
	Notice that
	\begin{align}
	& \left[\beta_{0}^{2}+\|\boldsymbol{\gamma}_{(j)}\|_{2}^{2}\right]-\left[(\beta_{0}+h)^{2}+\|\boldsymbol{\gamma}_{*}\|_{2}^{2}\right]\nonumber \\
	& \qquad=-2(\beta_{0}+h)h+h^{2}+\|\boldsymbol{\delta}_{\boldsymbol{\gamma},j}\|_{2}^{2}+2\boldsymbol{\gamma}_{*}^{\top}\boldsymbol{\delta}_{\boldsymbol{\gamma},j}\nonumber \\
	& \qquad\leq2|\beta_{0}+h|h+h^{2}+\|\boldsymbol{\delta}_{\boldsymbol{\gamma},j}\|_{2}^{2}+2\|\boldsymbol{\gamma}_{*}\|_{2}\cdot\|\boldsymbol{\delta}_{\boldsymbol{\gamma},j}\|_{2}\nonumber \\
	& \qquad\overset{(i)}{\leq}2\zeta M_{2}h+h^{2}+\|\boldsymbol{\delta}_{\boldsymbol{\gamma},j}\|_{2}^{2}+2\zeta M_{2}\|\boldsymbol{\delta}_{\boldsymbol{\gamma},j}\|_{2}\nonumber \\
	& \qquad\leq2\zeta M_{2}h_{n}+h_{n}^{2}+\|\boldsymbol{\delta}_{\boldsymbol{\gamma},j}\|_{2}^{2}+2\zeta M_{2}\|\boldsymbol{\delta}_{\boldsymbol{\gamma},j}\|_{2},\label{eq: null restriction 25}
	\end{align}
	where $(i)$ follows by $\|\boldsymbol{\gamma}_{*}\|_{2}\leq\zeta M_{2}$
	and $|\beta_{0}+h|\leq\zeta M_{2}$ (due to (\ref{eq: null restriction 23})).
	
	By the assumption of $sn^{-1}\log p\leq1/4$, $M>1$ and  the definition of
	$\rho$ in (\ref{eq: def rho}) we have that
	\begin{equation}
	h_{n}^{2}=\rho^{2}(sn^{-1}\log p)^{2}\leq\rho^{2}/16 \leq \frac{(1-\zeta^2)M_2^2}{64M} <(1-\zeta^{2})M_{2}^{2}/4\label{eq: null restriction 26}
	\end{equation}
	and
	\begin{multline}
	2\zeta M_{2}h_{n}=2\zeta M_{2}\rho sn^{-1}\log p\leq\zeta M_{2}\rho/2\\
	\leq  \frac{(1-\zeta^{2})M_{2}^{2}}{16\sqrt{M}}  <(1-\zeta^{2})M_{2}^{2}/4.\label{eq: null restriction 27}
	\end{multline}
	
	By Definition \ref{def: prior dist step 1}, we have
	\[
	\|\boldsymbol{\delta}_{\boldsymbol{\gamma},j}\|_{2}\leq h\|\pi_{(j)}\|_{2}+r\sigma_{\varepsilon,*}\sqrt{h/m}\|\boldsymbol{\delta}_{(j)}\|_{2}=h\|\pi_{(j)}\|_{2}+r\sigma_{\varepsilon,*}\sqrt{h}.
	\]
	
	By (\ref{eq: eligibility cand null eq 1}) and (\ref{eq: norm bound Delta pi}), $\|\pi_{(j)}\|_{2}\leq\|\pi_{*}\|_{2}+\|\pi_{(j)}-\pi_{*}\|_{2}\leq\sqrt{M}+\sigma_{\Vb,*}\sqrt{h}$.
	Since $h\leq h_{n}=\rho sn^{-1}\log p\leq\rho/4$, we have that
	\begin{align*}
	\|\boldsymbol{\delta}_{\boldsymbol{\gamma},j}\|_{2} & \leq\frac{1}{4}\rho\left(\sqrt{M}+\sigma_{\Vb,*}\sqrt{\rho/4}\right)+r\sigma_{\varepsilon,*}\sqrt{\rho/4}\\
	& \overset{(i)}{\leq}\frac{1}{4}\rho\left(1+\sqrt{\rho/4}\right)\sqrt{M}+\kappa^{-1}\sqrt{M}\zeta M_{1}\sqrt{\rho/4}\\
	& \overset{(ii)}{\leq}\frac{1}{2}\rho\sqrt{M}+\kappa^{-1}\sqrt{M}\zeta M_{1}\sqrt{\rho/4},
	\end{align*}
	where $(i)$ follows by $\sigma_{\Vb,*}\leq\sqrt{M}$ (due to (\ref{eq: eligibility cand null eq 1})),
	$\sigma_{\varepsilon,*}\leq\zeta M_{1}$ (due to the definition of
	$\Theta_{\zeta,\kappa}(s)$) and $r\leq\sqrt{M}/\kappa$ (due to (\ref{eq: eligibility cand null eq 1.1}))
	and $(ii)$ follows by $\rho\leq4$. By the definition of $\rho$
	in (\ref{eq: def rho}), we have
	\begin{align}
	2\zeta M_{2}\|\boldsymbol{\delta}_{\boldsymbol{\gamma},j}\|_{2} & \leq\zeta M_{2}\sqrt{M}\rho+\frac{\sqrt{M}}{\kappa}\zeta^{2}M_{1}M_{2}\sqrt{\rho}\nonumber \\
	& \leq\frac{(1-\zeta^{2})M_{2}^{2}}{8}+\frac{(1-\zeta^{2})M_{2}^{2}}{8}\leq\frac{(1-\zeta^{2})M_{2}^{2}}{4}.\label{eq: null restriction 28}
	\end{align}
	
	By the elementary inequality of $(a+b)^{2}\leq2a^{2}+2b^{2}$, we
	also have
	\begin{align}
	\|\boldsymbol{\delta}_{\boldsymbol{\gamma},j}\|_{2}^{2} & \leq\left(\frac{1}{2}\rho\sqrt{M}+\kappa^{-1}\sqrt{M}\zeta M_{1}\sqrt{\rho/4}\right)^{2}\nonumber \\
	& \leq\frac{1}{2}\rho^{2}M+\frac{M}{2\kappa^{2}}\zeta^{2}M_{1}^{2}\rho\nonumber \\
	& \overset{(i)}{\leq}\frac{(1-\zeta^{2})M_{2}^{2}}{8}+\frac{(1-\zeta^{2})M_{2}^{2}}{8}\leq\frac{(1-\zeta^{2})M_{2}^{2}}{4},\label{eq: null restriction 29}
	\end{align}
	where $(i)$ follows by the definition of $\rho$ in (\ref{eq: def rho}).
	
	In light of (\ref{eq: null restriction 25}), we obtain (\ref{eq: null restriction 24})
	by combining (\ref{eq: null restriction 26}), (\ref{eq: null restriction 27}),
	(\ref{eq: null restriction 28}) and (\ref{eq: null restriction 29}).
	The proof is complete.
\end{proof}

\begin{proof}[Proof of Lemma \ref{lem:6}]
	Notice that
	\begin{multline*}
	\begin{pmatrix}a & b^{\top}\Sigmab\\
	\Sigmab b & \Sigmab
	\end{pmatrix}^{-1}\\
	=\begin{pmatrix}a^{-1}+a^{-2}b^{\top}\Sigmab(\Sigmab-a^{-1}\Sigmab bb^{\top}\Sigmab)^{-1}\Sigmab b & -b^{\top}\Sigmab(\Sigmab-a^{-1}\Sigmab bb^{\top}\Sigmab)^{-1}\\
	-(\Sigmab-a^{-1}\Sigmab bb^{\top}\Sigmab)^{-1}\Sigmab b & (\Sigmab-a^{-1}\Sigmab bb^{\top}\Sigmab)^{-1}
	\end{pmatrix}.
	\end{multline*}

	Since all the eigenvalues of the above matrix are positive, the eigenvalues of the blocks on the diagonal are also positive. This means that the eigenvalues of $\Sigmab-a^{-1}\Sigmab bb^{\top}\Sigmab$ are positive. Notice that
	$$\Sigmab-a^{-1}\Sigmab bb^{\top}\Sigmab=\Sigmab^{1/2} (\II-a^{-1}\Sigmab^{1/2} bb^{\top}\Sigmab^{1/2})\Sigmab^{1/2}. $$
	Since $\Sigmab^{1/2}$ is positive definite, we have that all the eigenvalues of $\II-a^{-1}\Sigmab^{1/2} bb^{\top}\Sigmab^{1/2}$ is positive. It follows that $$\det (\II-a^{-1}\Sigmab^{1/2} bb^{\top}\Sigmab^{1/2})>0.$$
	By Sylvester's determinant identity, we have  $\det (\II-a^{-1}\Sigmab^{1/2} bb^{\top}\Sigmab^{1/2})= 1-a^{-1}b^{\top}\Sigmab b $. The desired result follows.


\end{proof}

\subsection{Proof of auxiliary lemmas used in proving Theorem \ref{THM:8}}

\begin{proof}[Proof of Lemma \ref{lem: Bernstein}]
	We first prove the result assuming $\sigma_{1}=\sigma_{2}=1$. Let
	$r_{i}=r_{i,1}r_{i,2}$. Then for any $m\geq3$,
	\[
	|r_{i}|^{m}=|r_{i,1}r_{i,2}|^{m}\overset{(i)}{\leq}2^{-m}(r_{i,1}^{2}+r_{i,2}^{2})^{m}\overset{(ii)}{\leq}\frac{1}{2}(|r_{i,1}|^{2m}+|r_{i,2}|^{2m}),
	\]
	where $(i)$ follows by $|r_{i,1}r_{i,2}|\leq(r_{i,1}^{2}+r_{i,2}^{2})/2$,
	$(ii)$ follows by the elementary inequality $(a+b)^{m}\leq2^{m-1}(a^{m}+b^{m})$
	for $a,b\geq0$ and $m\geq2$. Hence,
	\[
	\sum_{i=1}^{n}\mathbb{E}|r_{i}|^{m}\leq\frac{n}{2}\left(\mathbb{E}|r_{1,1}|^{2m}+\mathbb{E}|r_{1,2}|^{2m}\right).
	\]
	
	Since $r_{1,1}\sim\mathcal{N}(0,1)$, we have that $r_{1,1}^{2}\sim\chi^{2}(1)$.
	The moment generating function of $\chi^{2}$ distributions implies
	\[
	\mathbb{E}\exp(r_{1,1}^{2}/3)=(1-2/3)^{-1}=3.
	\]
	
	Notice that by Taylor's series,
	\[
	\mathbb{E}\exp(r_{1,1}^{2}/3)=1+\sum_{j=1}^{\infty}\frac{3^{-j}\mathbb{E}\exp(|r_{1,1}|^{2j})}{j!}.
	\]
	Therefore, for any $j\geq1$,
	\[
	\frac{3^{-j}\mathbb{E}\exp(|r_{1,1}|^{2j})}{j!}<3.
	\]
	
	Similarly, we can show that for any $j\geq1$,
	\[
	\frac{3^{-j}\mathbb{E}\exp(|r_{1,2}|^{2j})}{j!}<3.
	\]
	
	Let $\nu=2n$. Hence, for $m\geq6$,
	\begin{align*}
	&\sum_{i=1}^{n}\mathbb{E}|r_{i}|^{m}\leq\frac{n}{2}\left(\mathbb{E}|r_{1,1}|^{2m}+\mathbb{E}|r_{1,2}|^{2m}\right)
	\\
	& \qquad \qquad \leq\frac{n}{2}\left(3^{m+1}m!+3^{m+1}m!\right)=n3^{m+1}m!<\frac{m!}{2}\nu7^{m-2}.
	\end{align*}
	
	Since both $r_{1,1}$ and $r_{1,2}$ are standard normal, we can
	easily compute for $m=3,4,5$
	\[
	\sum_{i=1}^{n}\mathbb{E}|r_{i}|^{m}\leq\frac{n}{2}\left(\mathbb{E}|r_{1,1}|^{2m}+\mathbb{E}|r_{1,2}|^{2m}\right)=\begin{cases}
	15n & m=3\\
	105n & m=4\\
	945n & m=5.
	\end{cases}
	\]
	
	Thus, $\sum_{i=1}^{n}\mathbb{E}|r_{i}|^{m}\leq\frac{m!}{2}\nu7^{m-2}$
	for $m\geq3$. Clearly, $\sum_{i=1}^{n}\mathbb{E}(r_{i}^{2})=n<\nu$.
	Therefore, by Corollary 2.11 of \citet{boucheron2013concentration},
	we have that for any $t>0$,
	\[
	\mathbb{P}\left(\sum_{i=1}^{n}\left(r_{i}-\mathbb{E}r_{i}\right)\geq t\right)\leq\exp\left(-\frac{t^{2}}{2(2n+7t)}\right).
	\]
	
	Similarly, we can show the same result for $-r_{i}$: for any $t>0$,
	\[
	\mathbb{P}\left(-\sum_{i=1}^{n}\left(r_{i}-\mathbb{E}r_{i}\right)\geq t\right)\leq\exp\left(-\frac{t^{2}}{2(2n+7t)}\right).
	\]
	
	Hence,
	\[
	\mathbb{P}\left(\left|\sum_{i=1}^{n}\left(r_{i}-\mathbb{E}r_{i}\right)\right|\geq t\right)\leq2\exp\left(-\frac{t^{2}}{2(2n+7t)}\right).
	\]
	
	We have proved the result for $\sigma_{1}=\sigma_{2}=1$. In the general
	case, notice that $r_{i,1}\sigma_{1}^{-1}\sim\mathcal{N}(0,1)$
	and $r_{i,2}\sigma_{2}^{-1}\sim\mathcal{N}(0,1)$. Hence, the above
	display implies
	\[
	\mathbb{P}\left(\left|\sum_{i=1}^{n}\left(r_{i,1}r_{i,2}\sigma_{1}^{-1}\sigma_{2}^{-1}-\mathbb{E}r_{i,1}r_{i,2}\sigma_{1}^{-1}\sigma_{2}^{-1}\right)\right|\geq t\right)\leq2\exp\left(-\frac{t^{2}}{2(2n+7t)}\right).
	\]
	
	The desired result follows.
\end{proof}

\begin{proof}[Proof of Lemma \ref{lem: bounded A}]
	By the definition of $\tilde{\Theta}(s)$, we have that $\beta^{2}+\|\boldsymbol{\gamma}\|_{2}^{2}\leq M_{2}^{2}$.
	Notice that the first row of $\boldsymbol{\Sigma}^{-1}$ is $(1,-\boldsymbol{\pi}^{\top})\sigma_{\mathbf{V}}^{-2}$.
	Therefore, 
	$$M^{-1}\leq\lambda_{\min}(\boldsymbol{\Sigma}^{-1})\leq\|\boldsymbol{\pi}\|_{2}^{2}\sigma_{\mathbf{V}}^{-2}+\sigma_{\mathbf{V}}^{-2}\leq\lambda_{\max}(\boldsymbol{\Sigma}^{-1})\leq M.$$
	This means that $M^{-1/2}\leq\sigma_{\mathbf{V}}\leq M^{1/2}$ and
	$\|\boldsymbol{\pi}\|_{2}\leq M$. Since $M>1$, it follows that
	\begin{align*}
	\|\boldsymbol{\xi}\|_{2}
	&\leq\lambda_{\max}(\boldsymbol{\Sigma}_{\mathbf{W}})\left(|\beta|\cdot\|\boldsymbol{\pi}\|_{2}+\|\boldsymbol{\gamma}\|_{2}\right)
	\\
	&\leq M\left(M_{2}M+M_{2}\right)=M_{2}M(M+1)<2M^{2}M_{2}.
	\end{align*}
	
	This proves part (1).
	
	Since $\mathbb{E}\mathbf{W}_{i}y_{i}=\boldsymbol{\Sigma}_{\mathbf{W}}(\boldsymbol{\pi}\beta+\boldsymbol{\gamma})$,
	we have that for each $1\leq j\leq p-1$,
	$$\hat{\boldsymbol{\xi}}_{j}-\boldsymbol{\xi}_{j}=b_{n}^{-1}\sum_{i\in H_{3}}\left[\mathbf{W}_{i,j}y_{i}-\mathbb{E}\mathbf{W}_{i,j}y_{i}\right].$$
	Notice that both $\mathbf{W}_{i,j}$ and $y_{i}$ are normal random
	variables with mean zero. Moreover,
	\[
	\mathbb{E}\mathbf{W}_{i,j}^{2}\leq\lambda_{\max}(\boldsymbol{\Sigma})\leq M
	\]
	and
	\[
	\mathbb{E}y_{i}^{2}=\sigma^{2}+\boldsymbol{\beta}^{\top}\boldsymbol{\Sigma}\boldsymbol{\beta}\leq\sigma^{2}+\lambda_{\max}(\boldsymbol{\Sigma}_{\mathbf{W}})\|\boldsymbol{\beta}\|_{2}^{2}\leq M_{1}^{2}+MM_{2}^{2},
	\]
	where we recall $\boldsymbol{\beta}=(\beta,\boldsymbol{\gamma}^{\top})^{\top}\in\mathbb{R}^{p}$.
	
	It follows by Lemma \ref{lem: Bernstein} that $\forall t>0$,
	\[
	\mathbb{P}\left(b_{n}|\hat{\boldsymbol{\xi}}_{j}-\boldsymbol{\xi}_{j}|>t\sqrt{M(M_{1}^{2}+MM_{2}^{2})}\right)\leq2\exp\left(-\frac{t^{2}}{2(2b_{n}+7t)}\right).
	\]
	
	We set $t=2\sqrt{n\log p}$. Since $n/4-1<b_{n}\leq n/4$ and $n/\log p\geq784=28^{2}$,
	the union bound implies
	\begin{align}
	& \mathbb{P}\left(\|\hat{\boldsymbol{\xi}}-\boldsymbol{\xi}\|_{\infty}>2b_{n}^{-1}\sqrt{n(\log p)M(M_{1}^{2}+MM_{2}^{2})}\right)\nonumber \\
	& \leq2p\exp\left(-\frac{t^{2}}{2(2b_{n}+7t)}\right)\nonumber \\
	& \leq2p\exp\left(-\frac{4n\log p}{2(n/2+14\sqrt{n\log p})}\right)\nonumber \\
	& =2\exp\left(\left(1-\frac{4}{1+28\sqrt{n^{-1}\log p}}\right)\log p\right)\nonumber \\
	& \leq2\exp\left(\left(1-\frac{4}{1+1}\right)\log p\right)=2/p.\label{eq: bounded A 11}
	\end{align}
	
	Since $M>1$, we have proved part (2).
	
	By the same argument,
	\begin{equation}
	\mathbb{P}\left(\|\tilde{\boldsymbol{\xi}}-\boldsymbol{\xi}\|_{\infty}>2b_{n}^{-1}M\sqrt{n(\log p)(M_{1}^{2}+M_{2}^{2})}\right)\leq2/p\label{eq: bounded A 12}
	\end{equation}
	and
	\[
	\mathbb{P}\left(\left\Vert b_{n}^{-1}\sum_{i\in H_{4}}\mathbf{W}_{i}y_{i}-\boldsymbol{\xi}\right\Vert _{\infty}>2b_{n}^{-1}M\sqrt{n(\log p)(M_{1}^{2}+M_{2}^{2})}\right)\leq2/p.
	\]
	
	Part (3) follows.
	
	Now we prove part (4).
	
	Denote $\tau=2b_{n}^{-1}M\sqrt{n(\log p)(M_{1}^{2}+M_{2}^{2})}$ and
	the event $\mathcal{B}=\{\|\tilde{\boldsymbol{\xi}}-\boldsymbol{\xi}\|_{\infty}\leq\tau\}$.
	Notice that $A=\{j:\ |\tilde{\boldsymbol{\xi}}_{j}|\geq2\tau\}$ by
	the definition in (\ref{eq: def set A}). Define $A_{\tau}=\{j:\ |\boldsymbol{\xi}_{j}|\geq\tau\}$.
	
	Since $|\boldsymbol{\xi}_{j}|\geq|\tilde{\boldsymbol{\xi}}_{j}|-|\tilde{\boldsymbol{\xi}}_{j}-\boldsymbol{\xi}_{j}|$,
	we have that $|\boldsymbol{\xi}_{j}|\geq|\tilde{\boldsymbol{\xi}}_{j}|-\|\tilde{\boldsymbol{\xi}}-\boldsymbol{\xi}\|_{\infty}$.
	Therefore, on the event $\mathcal{B}$, $|\boldsymbol{\xi}_{j}|\geq\tau$
	for any $j\in A$. In other words, on the event $\mathcal{B}$, $A\subseteq A_{\tau}$
	and thus $|A|\leq|A_{\tau}|$. To bound $|A_{\tau}|$, notice that
	$\tau^{2}|A_{\tau}|\leq\|\boldsymbol{\xi}\|_{2}^{2}$.
	
	Define the event $\mathcal{B}'=\{\|\hat{\boldsymbol{\xi}}-\boldsymbol{\xi}\|_{\infty}\leq\tau\}$.
	On the event $\mathcal{B}\bigcap\mathcal{B}'$,
	\begin{align*}
	\|\hat{\boldsymbol{\xi}}_{A}\|_{2} & \leq\|\boldsymbol{\xi}_{A}\|_{2}+\|\hat{\boldsymbol{\xi}}_{A}-\boldsymbol{\xi}_{A}\|_{2}\\
	& \leq\|\boldsymbol{\xi}\|_{2}+\sqrt{|A|}\|\hat{\boldsymbol{\xi}}_{A}-\boldsymbol{\xi}_{A}\|_{\infty}\\
	& \leq\|\boldsymbol{\xi}\|_{2}+\sqrt{|A_{\tau}|}\|\hat{\boldsymbol{\xi}}_{A}-\boldsymbol{\xi}_{A}\|_{\infty}\\
	& \leq\|\boldsymbol{\xi}\|_{2}+\sqrt{\|\boldsymbol{\xi}\|_{2}^{2}\tau^{-2}}\tau\\
	& =2\|\boldsymbol{\xi}\|_{2}\leq4M^{2}M_{2}.
	\end{align*}
	
	Part (4) follows because (\ref{eq: bounded A 12}) and part (2) imply
	$\mathbb{P}(\mathcal{B}\bigcap\mathcal{B}')\geq1-4/p$.
	
	To see part (5), notice that for any $j\in A^{c}$,
	\[
	|\hat{\boldsymbol{\xi}}_{j}|\leq\|\hat{\boldsymbol{\xi}}-\boldsymbol{\xi}\|_{\infty}+\|\tilde{\boldsymbol{\xi}}-\boldsymbol{\xi}\|_{\infty}+|\tilde{\boldsymbol{\xi}}_{j}|\leq\|\hat{\boldsymbol{\xi}}-\boldsymbol{\xi}\|_{\infty}+\|\tilde{\boldsymbol{\xi}}-\boldsymbol{\xi}\|_{\infty}+2\tau.
	\]
	Therefore, on the event $\mathcal{B}\bigcap\mathcal{B}'$, $|\hat{\boldsymbol{\xi}}_{j}|\leq4\tau$
	for any $j\in A^{c}$. Part (5) follows.
	
	Now we show part (6). The argument is similar to the proof of part
	(2). Notice that $v_{i}\sim\mathcal{N}(0,\sigma_{\mathbf{V}}^{2})$
	and
	$$\mathbf{W}_{i}^{\top}(\boldsymbol{\pi}\beta+\boldsymbol{\gamma})+\varepsilon_{i}\sim\mathcal{N}(0,(\boldsymbol{\pi}\beta+\boldsymbol{\gamma})^{\top}\boldsymbol{\Sigma}_{\mathbf{W}}(\boldsymbol{\pi}\beta+\boldsymbol{\gamma})+\sigma^{2}).$$
	
	Also notice that $\sigma_{\mathbf{V}}^{2}\leq M$ and
	\begin{align*}
	(\boldsymbol{\pi}\beta+\boldsymbol{\gamma})^{\top}\boldsymbol{\Sigma}_{\mathbf{W}}(\boldsymbol{\pi}\beta+\boldsymbol{\gamma})+\sigma^{2} & \leq\lambda_{\max}(\boldsymbol{\Sigma}_{\mathbf{W}})\|\boldsymbol{\pi}\beta+\boldsymbol{\gamma}\|_{2}^{2}+M_{1}^{2}\\
	& \leq\lambda_{\max}(\boldsymbol{\Sigma}_{\mathbf{W}})\left(\|\boldsymbol{\pi}\|_{2}\cdot|\beta|+\|\boldsymbol{\gamma}\|_{2}\right)^{2}+M_{1}^{2}\\
	& \leq M\left(M_{2}M+M_{2}\right)^{2}+M_{1}^{2}\overset{(i)}{<}4M_{2}^{2}M^{3}+M_{1}^{2},
	\end{align*}
	where $(i)$ follows by $M>1$.
	
	Since $\mathbb{E}v_{i}[\mathbf{W}_{i}^{\top}(\boldsymbol{\pi}\beta+\boldsymbol{\gamma})+\varepsilon_{i}]=0$,
	it follows by Lemma \ref{lem: Bernstein} that for any $t>0$,
	\begin{align}
	& \mathbb{P}\left(\left|\sum_{i\in H_{4}}v_{i}\left(\mathbf{W}_{i}^{\top}(\boldsymbol{\pi}\beta+\boldsymbol{\gamma})+\varepsilon_{i}\right)\right|>tb_{n}^{1/2}\sqrt{M\left(4M_{2}^{2}M^{3}+M_{1}^{2}\right)}\right)\nonumber \\
	& \leq2\exp\left(-\frac{t^{2}b_{n}}{2(2b_{n}+7tb_{n}^{1/2})}\right)=2\exp\left(-\frac{t^{2}}{2(2+7tb_{n}^{-1/2})}\right).\label{eq: bnded A 14}
	\end{align}
	
	Now we take $t=10\sqrt{\log(100/\alpha)}$. The assumption of $(n-4)/\log p\geq784$
	implies that $n>784$. Hence, $b_{n}>n/4-1>n/5$, which means $b_{n}^{-1/2}<\sqrt{5/n}$.
	Thus, the assumption of Theorem \ref{THM:8} implies that $n>500\log(100/\alpha)$
	and thus $tb_{n}^{-1/2}\leq10\sqrt{5n^{-1}\log(100/\alpha)}<1$. The
	above display implies
	\begin{align*}
	& \mathbb{P}\left(\left|\sum_{i\in H_{4}}v_{i}\left(\mathbf{W}_{i}^{\top}(\boldsymbol{\pi}\beta+\boldsymbol{\gamma})+\varepsilon_{i}\right)\right|>10b_{n}^{1/2}\sqrt{M\left(4M_{2}^{2}M^{3}+M_{1}^{2}\right)\log(100/\alpha)}\right)\\
	& \qquad \qquad<2\exp\left(-\frac{100\log(100/\alpha)}{2(2+7)}\right)=2\exp\left(-\frac{50}{9}\log(100/\alpha)\right)
	\\
	&\qquad \qquad <2\exp\left(-\log(100/\alpha)\right)=\alpha/50.
	\end{align*}
	
	This proves part (6).
	
	It remains to show part (7). Notice that $v_{i}\sim\mathcal{N}(0,\sigma_{\mathbf{V}}^{2})$
	and $M^{-1}\leq\sigma_{\mathbf{V}}^{2}\leq M$. By an argument similar
	to (\ref{eq: bnded A 14}), we have that for any $t>0$,
	\[
	\mathbb{P}\left(\left|\sum_{i\in H_{4}}\left(v_{i}^{2}-\mathbb{E}v_{i}^{2}\right)\right|>tM\right)\leq2\exp\left(-\frac{t^{2}}{2(2b_{n}+7t)}\right).
	\]
	
	Now we take $t=b_{n}/(2M^{2})$. Hence,
	\begin{align*}
	& \mathbb{P}\left(b_{n}^{-1}\sum_{i\in H_{4}}v_{i}^{2}<\frac{1}{2M}\right)  \leq\mathbb{P}\left(b_{n}^{-1}\sum_{i\in H_{4}}(v_{i}^{2}-\mathbb{E}v_{i}^{2})<\frac{1}{2M}-\mathbb{E}v_{i}^{2}\right)\\
	& \leq\mathbb{P}\left(b_{n}^{-1}\sum_{i\in H_{4}}(v_{i}^{2}-\mathbb{E}v_{i}^{2})<\frac{1}{2M}-\frac{1}{M}\right)\\
	& \leq\mathbb{P}\left(b_{n}^{-1}\left|\sum_{i\in H_{4}}\left(v_{i}^{2}-\mathbb{E}v_{i}^{2}\right)\right|>\frac{1}{2M}\right)\\
	& =\mathbb{P}\left(\left|\sum_{i\in H_{4}}\left(v_{i}^{2}-\mathbb{E}v_{i}^{2}\right)\right|>tM\right)\\
	& \leq2\exp\left(-\frac{b_{n}^{2}/(4M^{4})}{2(2b_{n}+7b_{n}/(2M^{2}))}\right)\\
	& =2\exp\left(-\frac{b_{n}/M^{4}}{(16+28/M^{2})}\right)\overset{(i)}{<}2\exp\left(-M^{-2}b_{n}/44\right),
	\end{align*}
	where $(i)$ follows by $28/M^{2}<28$ (since $M>1$). The proof is
	complete.
\end{proof}

\begin{proof}[Proof of Lemma \ref{lem: feasibility Omega}]
	We need to show
	\begin{equation}
	\mathbb{P}\left(\bigl \| \bigl(\mathbb{I}_{p-1}-b_{n}^{-1}\sum_{i\in H_{4}}\mathbf{W}_{i}\mathbf{W}_{i}^{\top}\boldsymbol{\Omega}_{\mathbf{W}}\bigl)\hat{\boldsymbol{\xi}}_{A}\bigl \|_{\infty}>24\sqrt{b_{n}^{-1}\log p}M^{3}M_{2}\right)<6/p\label{eq: feasibility Omega 2}
	\end{equation}
	and
	\begin{equation}
	\mathbb{P}\biggl(\hat{\boldsymbol{\xi}}_{A}^{\top}\boldsymbol{\Omega}_{\mathbf{W}}\biggl(b_{n}^{-1}\sum_{i\in H_{4}}\mathbf{W}_{i}\mathbf{W}_{i}^{\top}\biggl)\boldsymbol{\Omega}_{\mathbf{W}}\hat{\boldsymbol{\xi}}_{A}\leq32M^{5}M_{2}^{2}\biggl)\leq2\exp(-b_{n}/18)+4/p.\label{eq: feasibility Omega 3}
	\end{equation}
	
	We prove these two claims in two steps.
	
	\textbf{Step 1:} show (\ref{eq: feasibility Omega 2}).
	
	Define $q_{i}=\mathbf{W}_{i}\mathbf{W}_{i}^{\top}\boldsymbol{\Omega}_{\mathbf{W}}\hat{\boldsymbol{\xi}}_{A}$
	and $q_{i,j}=\mathbf{W}_{i,j}\mathbf{W}_{i}^{\top}\boldsymbol{\Omega}_{\mathbf{W}}\hat{\boldsymbol{\xi}}_{A}$
	for $1\leq j\leq p-1$. Let $\mathcal{F}$ denote the $\sigma$-algebra
	generated by $\{(\mathbf{W}_{i},y_{i},Z_{i})\}_{i\in H_{1}\bigcup H_{3}}$.
	Notice that $\hat{\boldsymbol{\xi}}_{A}$ is $\mathcal{F}$-measurable
	and $\{\mathbf{W}_{i}\}_{i\in H_{4}}$ is independent of $\mathcal{F}$
	due to the sample splitting. Therefore, for $i\in H_{4}$, conditional
	on $\mathcal{F}$, $\mathbf{W}_{i,j}$ and $\mathbf{W}_{i}^{\top}\boldsymbol{\Omega}_{\mathbf{W}}\hat{\boldsymbol{\xi}}_{A}$
	are both Gaussian with mean zero.
	
	Also observe that for $i\in H_{4}$, $\mathbb{E}(\mathbf{W}_{i,j}^{2}\mid\mathcal{F})\leq\lambda_{\max}(\boldsymbol{\Sigma}_{\mathbf{W}})\leq M$
	and
	\begin{align*}
	\mathbb{E}[(\mathbf{W}_{i}^{\top}\boldsymbol{\Omega}_{\mathbf{W}}\hat{\boldsymbol{\xi}}_{A})^{2}\mid\mathcal{F}]
	&
	\leq\hat{\boldsymbol{\xi}}_{A}^{\top}\boldsymbol{\Omega}_{\mathbf{W}}\boldsymbol{\Sigma}_{\mathbf{W}}\boldsymbol{\Omega}_{\mathbf{W}}\hat{\boldsymbol{\xi}}_{A}   =\hat{\boldsymbol{\xi}}_{A}^{\top}\boldsymbol{\Omega}_{\mathbf{W}}\hat{\boldsymbol{\xi}}_{A}\\
	& \leq\lambda_{\max}(\boldsymbol{\Omega}_{\mathbf{W}})\|\hat{\boldsymbol{\xi}}_{A}\|_{2}^{2}\leq\frac{\|\hat{\boldsymbol{\xi}}_{A}\|_{2}^{2}}{\lambda_{\min}(\boldsymbol{\Sigma}_{\mathbf{W}})}\leq M\|\hat{\boldsymbol{\xi}}_{A}\|_{2}^{2}.
	\end{align*}
	
	Therefore, Lemma \ref{lem: Bernstein} implies that for any $t>0$,
	\[
	\mathbb{P}\left(\left|\sum_{i\in H_{4}}\left[q_{i,j}-\mathbb{E}(q_{i,j}\mid\mathcal{F})\right]\right|>tM\|\hat{\boldsymbol{\xi}}_{A}\|_{2}\mid\mathcal{F}\right)\leq2\exp\left(-\frac{t^{2}}{2(2b_{n}+7t)}\right).
	\]
	
	Since $\mathbb{E}(q_{i}\mid\mathcal{F})=\mathbb{E}(\mathbf{W}_{i}\mathbf{W}_{i}^{\top}\boldsymbol{\Omega}_{\mathbf{W}}\hat{\boldsymbol{\xi}}_{A}\mid\mathcal{F})=\boldsymbol{\Sigma}_{\mathbf{W}}\boldsymbol{\Omega}_{\mathbf{W}}\hat{\boldsymbol{\xi}}_{A}=\hat{\boldsymbol{\xi}}_{A}$,
	we apply the union bound and obtain that $\forall t>0$,
	\begin{align*}
	& \mathbb{P}\left(\bigl\| \bigl(b_{n}^{-1}\sum_{i\in H_{4}}\mathbf{W}_{i}\mathbf{W}_{i}^{\top}\boldsymbol{\Omega}_{\mathbf{W}}-\mathbb{I}_{p-1}\bigl)\hat{\boldsymbol{\xi}}_{A}\bigl\| _{\infty}>tM\|\hat{\boldsymbol{\xi}}_{A}\|_{2}\mid\mathcal{F}\right)\\
	& =\mathbb{P}\left(\max_{1\leq j\leq p-1}\biggl|\sum_{i\in H_{4}}\left[q_{i,j}-\mathbb{E}(q_{i,j}\mid\mathcal{F})\right]\biggl|>tb_{n}M\|\hat{\boldsymbol{\xi}}_{A}\|_{2}\mid\mathcal{F}\right)\\
	& \leq2p\exp\left(-\frac{t^{2}b_{n}^{2}}{2(2b_{n}+7tb_{n})}\right)=2p\exp\left(-\frac{t^{2}b_{n}}{2(2+7t)}\right).
	\end{align*}
	
	By choosing $t=6\sqrt{b_{n}^{-1}\log p}$, it follows that
	\begin{align*}
	& \mathbb{P}\left(\biggl\| \biggl(b_{n}^{-1}\sum_{i\in H_{4}}\mathbf{W}_{i}\mathbf{W}_{i}^{\top}\boldsymbol{\Omega}_{\mathbf{W}}-\mathbb{I}_{p-1}\biggl)\hat{\boldsymbol{\xi}}_{A} \biggl \|_\infty>6\sqrt{b_{n}^{-1}\log p}M\|\hat{\boldsymbol{\xi}}_{A}\|_{2}\right)\\
	& \leq2p\exp\biggl(-\frac{36\log p}{4+14\times8\sqrt{b_{n}^{-1}\log p}}\biggl)\\
	& \overset{(i)}{\leq}2p\exp\left(-\frac{36\log p}{4+14\times6/14}\right)=2p^{-2.6}<2p^{-2},
	\end{align*}
	where $(i)$ follows by the fact that $b_{n}>n/4-1$ and the assumption
	$(n-4)/\log p\geq784=28^{2}$. By Lemma \ref{lem: bounded A}, $\mathbb{P}\left(\|\hat{\boldsymbol{\xi}}_{A}\|_{2}\leq4M^{2}M_{2}\right)\geq1-4/p$.
	Therefore, we have
	\begin{align*}
	& \mathbb{P}\left(\biggl \| \bigl(\mathbb{I}_{p-1}-b_{n}^{-1}\sum_{i\in H_{4}}\mathbf{W}_{i}\mathbf{W}_{i}^{\top}\boldsymbol{\Omega}_{\mathbf{W}}\bigl)\hat{\boldsymbol{\xi}}_{A}\biggl \| _{\infty}>24\sqrt{b_{n}^{-1}\log p}M^{3}M_{2}\right)
	\\
	&\qquad \qquad \leq4/p+2p^{-2}<6/p.
	\end{align*}
	
	We have proved (\ref{eq: feasibility Omega 2}).
	
	\textbf{Step 2:} show (\ref{eq: feasibility Omega 3}).
	
	Let $r_{i}=\hat{\boldsymbol{\xi}}_{A}^{\top}\boldsymbol{\Omega}_{\mathbf{W}}\mathbf{W}_{i}$.
	For $i\in H_{4}$, notice that conditional on $\mathcal{F}$, $r_{i}$
	is Gaussian with mean zero and variance $\hat{\boldsymbol{\xi}}_{A}^{\top}\boldsymbol{\Omega}_{\mathbf{W}}\boldsymbol{\Sigma}_{\mathbf{W}}\boldsymbol{\Omega}_{\mathbf{W}}\hat{\boldsymbol{\xi}}_{A}=\hat{\boldsymbol{\xi}}_{A}^{\top}\boldsymbol{\Omega}_{\mathbf{W}}\hat{\boldsymbol{\xi}}_{A}$.
	It follows by Lemma \ref{lem: Bernstein} that
	\[
	\mathbb{P}\left(\left|\sum_{i\in H_{4}}\left[r_{i}^{2}-\mathbb{E}(r_{i}^{2}\mid\mathcal{F})\right]\right|>t\hat{\boldsymbol{\xi}}_{A}^{\top}\boldsymbol{\Omega}_{\mathbf{W}}\hat{\boldsymbol{\xi}}_{A}\mid\mathcal{F}\right)\leq2\exp\left(-\frac{t^{2}}{2(2b_{n}+7t)}\right).
	\]
	
	Since $\mathbb{E}(r_{i}^{2}\mid\mathcal{F})=\hat{\boldsymbol{\xi}}_{A}'\boldsymbol{\Omega}_{\mathbf{W}}\hat{\boldsymbol{\xi}}_{A}$,
	we have
	\begin{align*}
	& \mathbb{P}\left(b_{n}^{-1}\sum_{i\in H_{4}}r_{i}^{2}>\left(1+b_{n}^{-1}t\right)\hat{\boldsymbol{\xi}}_{A}^{\top}\boldsymbol{\Omega}_{\mathbf{W}}\hat{\boldsymbol{\xi}}_{A}\right)\\
	& =\mathbb{P}\left(\sum_{i\in H_{4}}\left[r_{i}^{2}-\mathbb{E}(r_{i}^{2}\mid\mathcal{F})\right]>t\hat{\boldsymbol{\xi}}_{A}^{\top}\boldsymbol{\Omega}_{\mathbf{W}}\hat{\boldsymbol{\xi}}_{A}\right)\leq2\exp\left(-\frac{t^{2}}{2(2b_{n}+7t)}\right).
	\end{align*}
	
	Notice that $\hat{\boldsymbol{\xi}}_{A}^{\top}\boldsymbol{\Omega}_{\mathbf{W}}\hat{\boldsymbol{\xi}}_{A}\leq\lambda_{\max}(\boldsymbol{\Omega}_{\mathbf{W}})\|\hat{\boldsymbol{\xi}}_{A}\|_{2}^{2}=\|\hat{\boldsymbol{\xi}}_{A}\|_{2}^{2}/\lambda_{\max}(\boldsymbol{\Sigma}_{\mathbf{W}})\leq M\|\hat{\boldsymbol{\xi}}_{A}\|_{2}^{2}$.
	By Lemma \ref{lem: bounded A}, $\|\hat{\boldsymbol{\xi}}_{A}\|_{2}\leq4M^{2}M_{2}$
	with probability at least $1-4/p$. Therefore, we have that
	\[
	\mathbb{P}\left(b_{n}^{-1}\sum_{i\in H_{4}}r_{i}^{2}>16\left(1+b_{n}^{-1}t\right)M^{5}M_{2}^{2}\right)\leq2\exp\left(-\frac{t^{2}}{2(2b_{n}+7t)}\right)+4/p.
	\]
	
	Since $b_{n}^{-1}\sum_{i\in H_{4}}r_{i}^{2}=\hat{\boldsymbol{\xi}}_{A}^{\top}\boldsymbol{\Omega}_{\mathbf{W}}\left(b_{n}^{-1}\sum_{i\in H_{4}}\mathbf{W}_{i}\mathbf{W}_{i}^{\top}\right)\boldsymbol{\Omega}_{\mathbf{W}}\hat{\boldsymbol{\xi}}_{A}$,
	we choose $t=b_{n}$ and obtain (\ref{eq: feasibility Omega 3}).
\end{proof}

\begin{proof}[Proof of Lemma \ref{lem: RE condition}]
	We invoke Corollary 18 of \citet{rudelson2013reconstruction} and
	Lemma 4.1 of \citet{bickel2009simultaneous}.
	
	For any $k$ between $1$ and $p$, we define the sparse eigenvalues
	\[
	\phi_{\min}(k)=\min_{1\leq\|\boldsymbol{q}\|_{0}\leq k}\frac{b_{n}^{-1}\sum_{i\in H_{2}}(\mathbf{W}_{i}^{\top}\boldsymbol{q})^{2}}{\|\boldsymbol{q}\|_{2}^{2}}
	\]
	and
	\[
	\phi_{\max}(k)=\max_{1\leq\|\boldsymbol{q}\|_{0}\leq k}\frac{b_{n}^{-1}\sum_{i\in H_{2}}(\mathbf{W}_{i}^{\top}\boldsymbol{q})^{2}}{\|\boldsymbol{q}\|_{2}^{2}}.
	\]
	
	The proof proceeds in two steps. We first verify a sufficient condition
	for the sparse eigenvalue condition and then derive the desired result.
	
	\textbf{Step 1:} Show that rows of $\boldsymbol{\Sigma}_{\mathbf{W}}^{-1/2}\mathbf{W}$
	are isotropic and $\psi_{2}$ with constant $\sqrt{8/3}$.
	
	Notice that $\boldsymbol{\Sigma}_{\mathbf{W}}^{-1/2}\mathbf{W}$ is
	a matrix whose entries are i.i.d $\mathcal{N}(0,1)$. Let $\boldsymbol{r}^{\top}$
	denote the first row of $\boldsymbol{\Sigma}_{\mathbf{W}}^{-1/2}\mathbf{W}$.
	For any nonzero vector $\boldsymbol{q}\in\mathbb{R}^{p-1}$, $(\boldsymbol{r}^{\top}\boldsymbol{q})^{2}/\|\boldsymbol{q}\|_{2}^{2}$
	has a chi-squared distribution with one degree of freedom. By the
	moment generating function of chi-squared distributions, we have that
	for any $t>\sqrt{2}\|\boldsymbol{q}\|_{2}$,
	\[
	\mathbb{E}\left[\exp\left((\boldsymbol{r}^{\top}\boldsymbol{q})^{2}/t^{2}\right)\right]=\mathbb{E}\left[\exp\left(\frac{(\boldsymbol{r}^{\top}\boldsymbol{q})^{2}}{\|\boldsymbol{q}\|_{2}^{2}}\times\frac{\|\boldsymbol{q}\|_{2}^{2}}{t^{2}}\right)\right]=\left(1-\frac{2\|\boldsymbol{q}\|_{2}^{2}}{t^{2}}\right)^{-1/2}.
	\]
	
	Thus,
	\[
	\inf\left\{ t:\ \mathbb{E}\left[\exp\left((\boldsymbol{r}^{\top}\boldsymbol{q})^{2}/t^{2}\right)\right]\right\} \leq\sqrt{8/3}\|\boldsymbol{q}\|_{2}.
	\]
	
	In other words, $\boldsymbol{r}$ is isotropic and $\psi_{2}$ with
	constant $\sqrt{8/3}$; see Definition 5 of \citet{rudelson2013reconstruction}.
	
	\textbf{Step 2:} Show the desired result.
	
	By Corollary 18 of \citet{rudelson2013reconstruction}, we have that
	with probability at least $1-2\exp(-\tau^{2}b_{n}/570)$,
	\[
	(1-\tau)^{2}M^{-1}\leq\phi_{\min}(k)\leq\phi_{\max}(k)\leq(1+\tau)^{2}M
	\]
	if $b_{n}\geq570\tau^{-2}k\log(12ep/\tau)$. Let $m$ be the smallest
	integer satisfying 
	$$m\geq36M^{2}(1+\tau)^{2}(1-\tau)^{-2}s.$$
	
	This means that if $b_{n}\geq570\tau^{-2}(s+m)\log(12ep/\tau)$, then
	$\mathbb{P}\left(\mathcal{B}\right)\geq1-4\exp(-\tau^{2}b_{n}/570)$,
	where the event $\mathcal{B}$ is defined as
	\[
	\mathcal{B}=\left\{ \phi_{\min}(s+m)\geq(1-\tau)^{2}M^{-1}\ \text{and}\ \phi_{\max}(m)\leq(1+\tau)^{2}M\right\} .
	\]
	
	Notice that on the event $\mathcal{B}$, $m\phi_{\min}(m+s)>c_{0}^{2}s\phi_{\max}(m)$
	with $c_{0}=3$. By Lemma 4.1(ii) of \citet{bickel2009simultaneous},
	on the event $\mathcal{B}$
	\begin{align*}
	\sqrt{\kappa(s)} & =\sqrt{\phi_{\min}(m+s)}\left(1-c_{0}\sqrt{\frac{s\phi_{\max}(m)}{m\phi_{\min}(m+s)}}\right)\\
	& =\sqrt{\phi_{\min}(m+s)}-c_{0}\sqrt{\frac{s}{m}\phi_{\max}(s)}\\
	& \geq(1-\tau)M^{-1/2}-3\times\sqrt{\frac{s}{36M^{2}(1+\tau)^{2}(1-\tau)^{-2}s}\times(1+\tau)^{2}M}\\
	& =0.5(1-\tau)M^{-1/2}.
	\end{align*}
	
	The desired result follows.
\end{proof}

\begin{proof}[Proof of Lemma \ref{lem: lasso bound}]
	We invoke Theorem 6.1 of \citet{buhlmann2011statistics}. We first
	show a concentration result for $\|\sum_{i\in H_{2}}\mathbf{W}_{i}v_{i}\|_{\infty}$.
	
	For $1\leq j\leq p-1$, $\mathbf{W}_{i,j}\sim\mathcal{N}(0,\mathbb{E}(\mathbf{W}_{i,j}^{2}))$
	with $\mathbb{E}(\mathbf{W}_{i,j}^{2})\leq\lambda_{\max}(\boldsymbol{\Sigma})\leq M$.
	Also observe that $v_{i}\sim\mathcal{N}(0,\sigma_{\mathbf{V}}^{2})$
	with $\sigma_{\mathbf{V}}^{2}\leq\lambda_{\max}(\boldsymbol{\Sigma})\leq M$.
	Since $\mathbb{E}(\mathbf{W}_{i,j}v_{i})=0$, it follows by Lemma
	\ref{lem: Bernstein} that $\forall t>0$,
	\[
	\mathbb{P}\left(\left|\sum_{i\in H_{2}}\mathbf{W}_{i,j}v_{i}\right|>tM\right)\leq2\exp\left(-\frac{t^{2}}{2(2b_{n}+7t)}\right).
	\]
	
	By the union bound, we have
	\[
	\mathbb{P}\left(\bigl\| \sum_{i\in H_{2}}\mathbf{W}_{i}v_{i}\bigl\| _{\infty}>tM\right)\leq2p\exp\left(-\frac{t^{2}}{2(2b_{n}+7t)}\right).
	\]
	
	Taking $t=6\sqrt{b_{n}\log p}$, we have that
	\begin{align*}
	& \mathbb{P}\left(\bigl\| \sum_{i\in H_{2}}\mathbf{W}_{i}v_{i}\bigl\|_{\infty}>6M\sqrt{b_{n}\log p}\right)  \leq2p\exp\left(-\frac{36\log p}{4+14\times6\sqrt{b_{n}^{-1}\log p}}\right)\\
	& \qquad \qquad \overset{(i)}{\leq}2p\exp\left(-\frac{36\log p}{4+6}\right)=2p^{-2.6}<2/p^{2},
	\end{align*}
	where $(i)$ follows by $b_{n}>n/4-1$ and the assumption $(n-4)/\log p\geq784=28^{2}$.
	In other words,
	\begin{equation}
	\mathbb{P}\left(2\bigl\| \sum_{i\in H_{2}}\mathbf{W}_{i}v_{i}\bigl\|_{\infty}/b_{n}\leq\lambda_{\boldsymbol{\pi}}/2\right)\geq1-2/p^{2}.\label{eq: lasso bnd 2}
	\end{equation}
	
	By the assumptions of Theorem \ref{THM:8} and $b_{n}>n/4-1$, we
	can easily verify the assumption of Lemma \ref{lem: RE condition}
	with $\tau=3/4$. Thus, we  apply Lemma \ref{lem: RE condition} with
	$\tau=3/4$ and obtain the restricted eigenvalue condition
	\begin{equation}
	\mathbb{P}\left(\kappa(s)>0.015M^{-1}\right)\geq1-4\exp\left(-3b_{n}/3040\right),\label{eq: lasso bnd 4}
	\end{equation}
	where $\kappa(s)$ is defined in Lemma \ref{lem: RE condition}. Notice
	that due to H\"older's inequality, $\kappa(s)$ is smaller than the
	compatibility constant in Equation (6.4) of \citet{buhlmann2011statistics}:
	\begin{align*}
	\kappa(s) & =\underset{|J|\subset\{1,...,p-1\},|J|\leq s}{\min}\ \underset{\|q_{J^{c}}\|_{1}\leq3\|q_{J}\|_{1}}{\min}\frac{b_{n}^{-1}\sum_{i\in H_{2}}(\mathbf{W}_{i}^{\top}q)^{2}}{\|q_{J}\|_{2}^{2}}\\
	& \leq\underset{|J|\subset\{1,...,p-1\},|J|\leq s}{\min}\ \underset{\|q_{J^{c}}\|_{1}\leq3\|q_{J}\|_{1}}{\min}\frac{b_{n}^{-1}\sum_{i\in H_{2}}(\mathbf{W}_{i}^{\top}q)^{2}}{\|q_{J}\|_{1}^{2}/s}.
	\end{align*}
	
	By (\ref{eq: lasso bnd 2}) and (\ref{eq: lasso bnd 4}), together
	with Theorem 6.1 of \citet{buhlmann2011statistics}, we have that
	\[
	\mathbb{P}\left(\|\hat \pib-\boldsymbol{\pi}\|_{1}\leq267s\lambda_{\boldsymbol{\pi}}M\right)\geq1-4\exp\left(-3b_{n}/3040\right)-2/p^{2}.
	\]
	
	This proves the first claim. For the second claim, we simply follow
	the same argument as in (\ref{eq: lasso bnd 2}) with $H_{2}$ replaced
	by $H_{4}$.
\end{proof}

\begin{proof}[Proof of Lemma \ref{lem: feasibility pi}]
	We need to show that with high probability,
	\begin{equation}
	\left|\hat{\boldsymbol{\xi}}_{A}^{\top}\boldsymbol{\pi}_{A}-\hat{\boldsymbol{\xi}}_{A}^{\top}{{\tilde \pib}_A}\right|\leq\eta_{\boldsymbol{\pi}},\label{eq: feasibility pi 1}
	\end{equation}
	and
	\[
	\bigl\| b_{n}^{-1}\sum_{i\in H_{4}}\mathbf{W}_{i}(Z_{i}-\mathbf{W}_{i}^{\top}\boldsymbol{\pi})\bigl\|_{\infty}\leq\lambda_{\boldsymbol{\pi}}/4
	\]
	as well as
	\[
	b_{n}^{-1}\sum_{i\in H_{4}}(Z_{i}-\mathbf{W}_{i}^{\top}\boldsymbol{\pi})^{2}\geq\frac{1}{2M}.
	\]
	
	Since $Z_{i}-\mathbf{W}_{i}^{\top}\boldsymbol{\pi}=v_{i}$, Lemmas
	\ref{lem: lasso bound} and \ref{lem: bounded A} imply that
	\begin{equation}
	\mathbb{P}\left(\bigl \| b_{n}^{-1}\sum_{i\in H_{4}}\mathbf{W}_{i}(Z_{i}-\mathbf{W}_{i}^{\top}\boldsymbol{\pi})\bigl \| _{\infty}\leq\lambda_{\boldsymbol{\pi}}/4\right)\geq1-2/p^{2}>1-2/p\label{eq: feasibility pi 3}
	\end{equation}
	and
	
	\begin{equation}
	\mathbb{P}\left(b_{n}^{-1}\sum_{i\in H_{4}}(Z_{i}-\mathbf{W}_{i}^{\top}\boldsymbol{\pi})^{2}\geq\frac{1}{2M}\right)\geq1-2\exp(-M^{-2}b_{n}/44).\label{eq: feasibility pi 4}
	\end{equation}
	
	It remains to show (\ref{eq: feasibility pi 1}). Notice that
	\[
	{\tilde \pib}-\boldsymbol{\pi}=\left(\mathbb{I}_{p-1}-\hat{\boldsymbol{\Omega}}_{\mathbf{W}}b_{n}^{-1}\sum_{i\in H_{4}}\mathbf{W}_{i}\mathbf{W}_{i}^{\top}\right)\left(\hat \pib-\boldsymbol{\pi}\right)+b_{n}^{-1}\sum_{i\in H_{4}}\hat{\boldsymbol{\Omega}}_{\mathbf{W}}\mathbf{W}_{i}v_{i}
	\]
	and thus
	\begin{align*}
	& \hat{\boldsymbol{\xi}}_{A}^{\top}{{\tilde \pib}_A}-\hat{\boldsymbol{\xi}}_{A}^{\top}\boldsymbol{\pi}_{A}\\
	& =\underset{T_{1}}{\underbrace{\hat{\boldsymbol{\xi}}_{A}^{\top}\left(\mathbb{I}_{p-1}-\hat{\boldsymbol{\Omega}}_{\mathbf{W}}b_{n}^{-1}\sum_{i\in H_{4}}\mathbf{W}_{i}\mathbf{W}_{i}^{\top}\right)\left(\hat \pib-\boldsymbol{\pi}\right)}}+\underset{T_{2}}{\underbrace{b_{n}^{-1}\sum_{i\in H_{4}}\hat{\boldsymbol{\xi}}_{A}^{\top}\hat{\boldsymbol{\Omega}}_{\mathbf{W}}\mathbf{W}_{i}v_{i}}}.
	\end{align*}
	
	We proceed in two steps. We first bound $T_{1}$ and then $T_{2}$.
	
	Let $\mathcal{B}$ denote the event that $\boldsymbol{\Omega}_{\mathbf{W}}$
	satisfies the constraint in (\ref{eq: constraint OmegaW}) for $\hat{\boldsymbol{\Omega}}_{\mathbf{W}}$.
	By Lemma \ref{lem: feasibility Omega},
	\begin{equation}
	\mathbb{P}\left(\mathcal{B}\right)\geq1-10/p-2\exp(-b_{n}/18).\label{eq: pi feasibility 5}
	\end{equation}
	
	\textbf{Step 1:} bound $T_{1}$
	
	Notice that on the event $\mathcal{B}$, $\hat{\boldsymbol{\Omega}}_{\mathbf{W}}$
	satisfies the constraint in (\ref{eq: constraint OmegaW}) and therefore,
	\begin{align*}
	|T_{1}| & \leq\left\Vert \hat{\boldsymbol{\xi}}_{A}^{\top}\left(\mathbb{I}_{p-1}-\hat{\boldsymbol{\Omega}}_{\mathbf{W}}b_{n}^{-1}\sum_{i\in H_{4}}\mathbf{W}_{i}\mathbf{W}_{i}^{\top}\right)\right\Vert _{\infty}\left\Vert \hat \pib-\boldsymbol{\pi}\right\Vert _{1}\\
	& \overset{(i)}{\leq}24\sqrt{b_{n}^{-1}\log p}M^{3}M_{2}\left\Vert \hat \pib-\boldsymbol{\pi}\right\Vert _{1},
	\end{align*}
	where $(i)$ follows by the constraint in (\ref{eq: constraint OmegaW}).
	By the bound in Lemma \ref{lem: lasso bound}, we have that
	\begin{equation}
	\mathbb{P}\left(|T_{1}|>6408\sqrt{b_{n}^{-1}\log p}M^{4}M_{2}s\lambda_{\boldsymbol{\pi}}\ \text{and}\ \mathcal{B}\right)\leq4\exp\left(-3b_{n}/3040\right)+2/p^{2}.\label{eq: pi feasibility 6}
	\end{equation}
	
	\textbf{Step 2:} bound $T_{2}$
	
	Let $\mathcal{F}$ be the $\sigma$-algebra generated by $\{(y_{i},\mathbf{W}_{i},Z_{i})\}_{i\in H_{1}\bigcup H_{3}}$
	and $\{\mathbf{W}_{i}\}_{i\in H_{4}}$. Notice that $\{v_{i}\}_{i\in H_{4}}$
	is independent of both $\{\mathbf{W}_{i}\}_{i\in H_{4}}$ and $\{(y_{i},\mathbf{W}_{i},Z_{i})\}_{i\in H_{1}\bigcup H_{3}}$.
	Hence, $\{v_{i}\}_{i\in H_{4}}$ is independent of $\mathcal{F}$.
	On the other hand, notice that $\{\hat{\boldsymbol{\xi}}_{A}^{\top}\hat{\boldsymbol{\Omega}}_{\mathbf{W}}\mathbf{W}_{i}\}_{i\in H_{4}}$
	is $\mathcal{F}$-measurable. Since $\{v_{i}\}_{i\in H_{4}}$ is i.i.d
	$\mathcal{N}(0,\sigma_{\mathbf{V}}^{2})$, we have that conditional
	on $\mathcal{F}$, $T_{2}$ is Gaussian with mean zero and variance
	$$\hat{\boldsymbol{\xi}}_{A}^{\top}\hat{\boldsymbol{\Omega}}_{\mathbf{W}}\left(b_{n}^{-2}\sum_{i\in H_{4}}\mathbf{W}_{i}\mathbf{W}_{i}^{\top}\right)\hat{\boldsymbol{\Omega}}_{\mathbf{W}}\hat{\boldsymbol{\xi}}_{A}.$$
	By the elementary bound of $\mathbb{P}(|X|>t\sigma)\leq2\exp(-t^{2}/2)$
	for $X\sim\mathcal{N}(0,\sigma^{2})$, we have that for any $t>0$,
	\[
	\mathbb{P}\left(|T_{2}|>t\sqrt{\hat{\boldsymbol{\xi}}_{A}^{\top}\hat{\boldsymbol{\Omega}}_{\mathbf{W}}\left(b_{n}^{-2}\sum_{i\in H_{4}}\mathbf{W}_{i}\mathbf{W}_{i}^{\top}\right)\hat{\boldsymbol{\Omega}}_{\mathbf{W}}\hat{\boldsymbol{\xi}}_{A}}\mid\mathcal{F}\right)\leq2\exp\left(-t^{2}/2\right).
	\]
	
	We notice that, on the event $\mathcal{B}$, $\hat{\boldsymbol{\Omega}}_{\mathbf{W}}$
	satisfies the constraint in (\ref{eq: constraint OmegaW}) and thus
	\[
	\hat{\boldsymbol{\xi}}_{A}^{\top}\hat{\boldsymbol{\Omega}}_{\mathbf{W}}\bigl(b_{n}^{-2}\sum_{i\in H_{4}}\mathbf{W}_{i}\mathbf{W}_{i}^{\top}\bigl)\hat{\boldsymbol{\Omega}}_{\mathbf{W}}\hat{\boldsymbol{\xi}}_{A}\leq32M^{5}M_{2}^{2}b_{n}^{-1}.
	\]
	
	It follows that for any $t>0$,
	\[
	\mathbb{P}\left(|T_{2}|>4tb_{n}^{-1/2}M^{2}M_{2}\sqrt{2M}\mid\mathcal{F}\right)\leq2\exp\left(-t^{2}/2\right).
	\]
	
	We take $t=\sqrt{2\log(100/\alpha)}$ and obtain
	\begin{equation}
	\mathbb{P}\left(|T_{2}|>8b_{n}^{-1/2}M^{2}M_{2}\sqrt{M\log(100/\alpha)}\ \text{and}\ \mathcal{B}\right)\leq0.02\alpha.\label{eq: pi feasibility 7}
	\end{equation}
	
	Now we combine (\ref{eq: pi feasibility 5}), (\ref{eq: pi feasibility 6})
	and (\ref{eq: pi feasibility 7}), obtaining
	\begin{align*}
	\mathbb{P}\left(|T_{1}|+|T_{2}|>\eta_{\boldsymbol{\pi}}\right) & \leq10/p+2\exp(-b_{n}/18)+0.02\alpha+4\exp\left(-3b_{n}/3040\right)+2/p^{2}\\
	& <12/p+0.02\alpha+6\exp\left(-3b_{n}/3040\right).
	\end{align*}
	
	Since $\hat{\boldsymbol{\xi}}_{A}^{\top}{{\tilde \pib}_A}-\hat{\boldsymbol{\xi}}_{A}^{\top}\boldsymbol{\pi}_{A}=T_{1}+T_{2}$,
	we have proved that (\ref{eq: feasibility pi 1}) holds with probability
	at least $1-12/p-0.02\alpha-6\exp\left(-3b_{n}/3040\right)$. By recalling
	(\ref{eq: feasibility pi 3}) and (\ref{eq: feasibility pi 4}), we
	complete the proof.
\end{proof}

\begin{proof}[Proof of Lemma \ref{lem: dantzig pi}]
	Let $\boldsymbol{\delta}={\breve \pib}-\boldsymbol{\pi}$
	and $\hat{\boldsymbol{\Sigma}}_{\mathbf{W}}=b_{n}^{-1}\sum_{i\in H_{4}}\mathbf{W}_{i}\mathbf{W}_{i}^{\top}$.
	Let $J_{0}=\text{supp}(\boldsymbol{\pi})$. Define $\mathcal{B}$
	to be the event that $\boldsymbol{\pi}$ satisfies the constraint
	in (\ref{eq: pi half estimator}) and $\kappa(s)\geq0.24(1-\tau)^{2}M^{-1}$,
	where $\kappa(s)$ is defined in Lemma \ref{lem: RE condition} and
	$\tau\in(0,1)$ is a constant to be determined later.
	
	On the event $\mathcal{B}$, we have that $\|{\breve \pib}\|_{1}\leq\|\boldsymbol{\pi}\|_{1}$,
	which means $\|\boldsymbol{\pi}+\boldsymbol{\delta}_{J_{0}}\|_{1}+\|\boldsymbol{\delta}_{J_{0}^{c}}\|_{1}\leq\|\boldsymbol{\pi}\|_{1}$.
	Hence, on the event $\mathcal{B}$,
	\begin{equation}
	\|\boldsymbol{\delta}_{J_{0}^{c}}\|_{1}\leq\|\boldsymbol{\delta}_{J_{0}}\|_{1}.\label{eq: dantzig pi 2}
	\end{equation}
	
	Also observe that on the event $\mathcal{B}$, $\|b_{n}^{-1}\sum_{i\in H_{4}}\mathbf{W}_{i}Z_{i}-\hat{\boldsymbol{\Sigma}}_{\mathbf{W}}\boldsymbol{\pi}\|_{\infty}\leq\lambda_{\boldsymbol{\pi}}/4$
	and $\|b_{n}^{-1}\sum_{i\in H_{4}}\mathbf{W}_{i}Z_{i}-\hat{\boldsymbol{\Sigma}}_{\mathbf{W}}{\breve \pib}\|_{\infty}\leq\lambda_{\boldsymbol{\pi}}/4$,
	which means
	\[
	\|\hat{\boldsymbol{\Sigma}}\boldsymbol{\delta}\|_{\infty}\leq\lambda_{\boldsymbol{\pi}}/2.
	\]
	
	Therefore, on the event $\mathcal{B}$,
	\begin{align*}
	\boldsymbol{\delta}^{\top}\hat{\boldsymbol{\Sigma}}\boldsymbol{\delta}
	& \leq\|\boldsymbol{\delta}\|_{1}\|\hat{\boldsymbol{\Sigma}}_{\mathbf{W}}\boldsymbol{\delta}\|_{\infty}\leq0.5\lambda_{\boldsymbol{\pi}}\|\boldsymbol{\delta}\|_{1}
	\\
	&=0.5\lambda_{\boldsymbol{\pi}}(\|\boldsymbol{\delta}_{J_{0}}\|_{1}+\|\boldsymbol{\delta}_{J_{0}^{c}}\|_{1})\overset{(i)}{\leq}\lambda_{\boldsymbol{\pi}}\|\boldsymbol{\delta}_{J_{0}}\|_{1}\leq\lambda_{\boldsymbol{\pi}}\sqrt{s}\|\boldsymbol{\delta}_{J_{0}}\|_{2},\label{eq: dantzig pi 4}
	\end{align*}
	where $(i)$ follows by (\ref{eq: dantzig pi 2}).
	
	On the other hand, we can lower bound $\boldsymbol{\delta}^{\top}\hat{\boldsymbol{\Sigma}}\boldsymbol{\delta}$
	via the restricted eigenvalue condition. By (\ref{eq: dantzig pi 2}),
	we have that on the event $\mathcal{B}$, $\|\boldsymbol{\delta}_{J_{0}^{c}}\|_{1}\leq\|\boldsymbol{\delta}_{J_{0}}\|_{1}\leq3\|\boldsymbol{\delta}_{J_{0}}\|_{1}$.
	Thus, we have that
	\[
	\boldsymbol{\delta}^{\top}\hat{\boldsymbol{\Sigma}}\boldsymbol{\delta}\geq\kappa(s)\|\boldsymbol{\delta}_{J_{0}}\|_{2}^{2}\geq0.24(1-\tau)^{2}M^{-1}\|\boldsymbol{\delta}_{J_{0}}\|_{2}^{2}.
	\]
	
	Now we combine the above two displays and obtain that on the event
	$\mathcal{B}$,
	\[
	\|\boldsymbol{\delta}_{J_{0}}\|_{2}\leq\frac{M\lambda_{\boldsymbol{\pi}}\sqrt{s}}{0.24(1-\tau)^{2}}.
	\]
	
	Therefore, (\ref{eq: dantzig pi 2}) implies that on the event $\mathcal{B}$,
	\[
	\|\boldsymbol{\delta}\|_{1}\leq2\|\boldsymbol{\delta}_{J_{0}}\|_{1}\leq2\sqrt{s}\|\boldsymbol{\delta}_{J_{0}}\|_{2}\leq\frac{2M\lambda_{\boldsymbol{\pi}}s}{0.24(1-\tau)^{2}}.
	\]
	
	Notice that by Lemmas \ref{lem: RE condition} and \ref{lem: feasibility pi},
	\begin{align*}
	\mathbb{P}(\mathcal{B})
	& \geq1-14/p-0.02\alpha-6\exp\left(-3b_{n}/3040\right)
	\\
	& -2\exp(-M^{-2}b_{n}/44)-4\exp(-\tau^{2}b_{n}/570).
	\end{align*}
	
	Hence, the desired result follows by choosing $\tau=3/4$.
\end{proof}

\subsection{Proof of auxiliary results used  in proving Theorem \ref{THM:10}} \label{sec:M}

\begin{proof}[Proof of Lemma \ref{lem: sufficient for non adaptivity}]
	Clearly, we always have $L(\Theta_{1},\Theta)\leq L(\Theta,\Theta)$.
	We only need to show the other direction. Let $c>0$ be a constant
	such that 
	\[
	c\inf_{CI\in\mathcal{C}_{\alpha}(\Theta)}\sup_{\theta\in\Theta}\mathbb{E}_{\theta}\mbox{\rm diam}(CI)\leq\inf_{CI\in\mathcal{C}_{\alpha}(\Theta)}\inf_{\theta\in\Theta}\mathbb{E}_{\theta}\mbox{\rm diam}(CI).
	\]
	
	Notice that 
	\begin{align*}
	L(\Theta_{1},\Theta) & =\inf_{CI\in\mathcal{C}_{\alpha}(\Theta)}\sup_{\theta\in\Theta_{1}}\mathbb{E}_{\theta}\mbox{\rm diam}(CI)\\
	& \geq\inf_{CI\in\mathcal{C}_{\alpha}(\Theta)}\inf_{\theta\in\Theta_{1}}\mathbb{E}_{\theta}\mbox{\rm diam}(CI)\\
	& \geq\inf_{CI\in\mathcal{C}_{\alpha}(\Theta)}\inf_{\theta\in\Theta}\mathbb{E}_{\theta}\mbox{\rm diam}(CI)\\
	& \geq c\inf_{CI\in\mathcal{C}_{\alpha}(\Theta)}\sup_{\theta\in\Theta}\mathbb{E}_{\theta}\mbox{\rm diam}(CI)=cL(\Theta,\Theta).
	\end{align*}
	
	The proof is complete. 
\end{proof}

\begin{proof}[Proof of Lemma \ref{lem: necessary for non adaptivity}]
	Clearly, we have 
	\[
	\inf_{CI\in\mathcal{C}_{\alpha}(\Theta)}\sup_{\theta\in\Theta}\mathbb{E}_{\theta}\mbox{\rm diam}(CI)\geq\inf_{CI\in\mathcal{C}_{\alpha}(\Theta)}\inf_{\theta\in\Theta}\mathbb{E}_{\theta}\mbox{\rm diam}(CI).
	\]
	
	We only need to show the other direction. Let $CI_{*}\in\mathcal{C}_{\alpha}(\Theta)$
	and $\theta_{*}\in\Theta$ be such that 
	\[
	\inf_{CI\in\mathcal{C}_{\alpha}(\Theta)}\inf_{\theta\in\Theta}\mathbb{E}_{\theta}\mbox{\rm diam}(CI)\geq0.9\mathbb{E}_{\theta_{*}}\mbox{\rm diam}(CI_{*}).
	\]
	
	Now define $\Theta_{1}=\{\theta_{*}\}$. Clearly, 
	\[
	\mathbb{E}_{\theta_{*}}\mbox{\rm diam}(CI_{*})=\sup_{\theta\in\Theta_{1}}\mathbb{E}_{\theta}\mbox{\rm diam}(CI_{*})\geq\inf_{CI\in\mathcal{C}_{\alpha}(\Theta)}\sup_{\theta\in\Theta_{1}}\mathbb{E}_{\theta}\mbox{\rm diam}(CI).
	\]
	
	By the assumption of $cL(\Theta,\Theta)\leq L(\Theta_{1},\Theta)$,
	we have 
	\[
	\mathbb{E}_{\theta_{*}}\mbox{\rm diam}(CI_{*})\geq c\inf_{CI\in\mathcal{C}_{\alpha}(\Theta)}\sup_{\theta\in\Theta}\mathbb{E}_{\theta}\mbox{\rm diam}(CI).
	\]
	
	Hence, 
	$$\inf_{CI\in\mathcal{C}_{\alpha}(\Theta)}\inf_{\theta\in\Theta}\mathbb{E}_{\theta}\mbox{\rm diam}(CI)\geq0.9c\inf_{CI\in\mathcal{C}_{\alpha}(\Theta)}\sup_{\theta\in\Theta}\mathbb{E}_{\theta}\mbox{\rm diam}(CI).$$
	The proof is complete. 
\end{proof}

\subsection{Proof of auxiliary results used in proving Theorem \ref{THM:12}}

\begin{proof}[Proof of Lemma \ref{lem: scaling para space}]
	Due to length of the work we comment that the result above is quite easy to verify. We leave the details to the reader.  \end{proof}

\begin{proof}[Proof of Lemma \ref{lem: obs equi}]
	If $(\mathbf{y},\mathbf{Z},\mathbf{W})\sim(\theta\odot D)$ with 
	$$\theta=(\beta,\boldsymbol{\gamma},\boldsymbol{\Sigma},\sigma)\in\widetilde{\Theta}_{N_{1},N_{2}}(s),$$
	then 
	$$\mathbf{y}=\mathbf{Z}\beta D+\mathbf{W}\boldsymbol{\gamma}D+\varepsilon$$
	with $\varepsilon\sim\mathcal{N}_{n}(0,\mathbb{I}_{n}(\sigma D)^{2})$
	and rows of $[\mathbf{Z},\mathbf{W}]$ being i.i.d $N(0,\boldsymbol{\Sigma})$.
	Now we divide both sides by $D$, obtaining 
	$$\mathbf{y}D^{-1}=\mathbf{Z}\beta+\mathbf{W}\boldsymbol{\gamma}+\tilde{\varepsilon}$$
	with $\tilde{\varepsilon}=\varepsilon D^{-1}$. Notice that $\tilde{\varepsilon}\sim\mathcal{N}_{n}(0,\mathbb{I}_{n}\sigma^{2})$
	and rows of $[\mathbf{Z},\mathbf{W}]$ being i.i.d $N(0,\boldsymbol{\Sigma})$.
	Thus, $(\mathbf{y},\mathbf{Z},\mathbf{W})\sim\theta$. This shows
	the ``only if'' direction. The ``if'' direction follows by an
	analogous argument. 
\end{proof}

\begin{proof}[Proof of Lemma \ref{lem: scale minimax}]
	Here, for notational simplicity, we use $|\cdot|$ to denote ${\rm diam}(\cdot)$. Fix any $\eta>0$. By the definition of infimum, there exists $T_{*}\in\mathcal{C}_{\alpha}(\widetilde{\Theta}_{N_{1},N_{2}}(s))$
	satisfying 
	\begin{align}
	{\mathbb A}(s,N_{1},N_{2}) &=\inf_{T\in\mathcal{C}_{\alpha}(\widetilde{\Theta}_{N_{1},N_{2}}(s))}\sup_{\theta\in\widetilde{\Theta}_{N_{1},N_{2}}(s)}\mathbb{E}_{\theta}|T(\mathbf{y},\mathbf{Z},\mathbf{W})| \nonumber
	\\
	&\geq\sup_{\theta\in\widetilde{\Theta}_{N_{1},N_{2}}(s)}\mathbb{E}_{\theta}|T_{*}(\mathbf{y},\mathbf{Z},\mathbf{W})|-\eta.\label{eq: scale minimax eq 4}
	\end{align}
	Define $\tilde{T}$ by 
	$$\tilde{T}(\mathbf{y},\mathbf{Z},\mathbf{W})=D\cdot T_{*}(\mathbf{y}D^{-1},\mathbf{Z},\mathbf{W}).$$
	
	For an arbitrary $\theta_{0}=(\beta_{0},\boldsymbol{\gamma}_{0},\boldsymbol{\Sigma},\sigma_{0})\in\widetilde{\Theta}_{DN_{1},DN_{2}}(s)$,
	we define $\theta_{1}=(\beta_{1},\boldsymbol{\gamma}_{1},\boldsymbol{\Sigma},\sigma_{1})$
	by $$\beta_{1}=\beta_{0}D^{-1}, \ \boldsymbol{\gamma}_{1}=\boldsymbol{\gamma}_{0}D^{-1}, \ \mbox{and }
	\sigma_{1}=\sigma_{0}D^{-1}.$$ 
	
	Notice that $\theta_{0}=\theta_{1}\odot D$.
	Notice that 
	$$|\tilde{T}(\mathbf{y},\mathbf{Z},\mathbf{W})|=D|T_{*}(\mathbf{y}D^{-1},\mathbf{Z},\mathbf{W})|.$$
	\vskip 50pt
	Therefore, 
	\begin{align}
	&\sup_{\theta_{1}\in\Theta(s,DN_{1},DN_{2})}\mathbb{E}_{(\mathbf{y},\mathbf{Z},\mathbf{W})\sim\theta_{1}}|\tilde{T}(\mathbf{y},\mathbf{Z},\mathbf{W})| 
	\nonumber \\
	& \qquad\qquad=D\sup_{\theta_{1}\in\widetilde{\Theta}_{DN_{1},DN_{2}}(s)}\mathbb{E}_{(\mathbf{y},\mathbf{Z},\mathbf{W})\sim\theta_{1}}|T_{*}(\mathbf{y}D^{-1},\mathbf{Z},\mathbf{W})|\nonumber \\
	&\qquad\qquad \overset{(i)}{=}D\sup_{\theta\in\widetilde{\Theta}_{N_{1},N_{2}}(s)}\mathbb{E}_{(\mathbf{y},\mathbf{Z},\mathbf{W})\sim(\theta\odot D)}|T_{*}(\mathbf{y}D^{-1},\mathbf{Z},\mathbf{W})|\nonumber \\
	&\qquad\qquad \overset{(ii)}{=}D\sup_{\theta\in\widetilde{\Theta}_{N_{1},N_{2}}(s)}\mathbb{E}_{(\mathbf{y}D^{-1},\mathbf{Z},\mathbf{W})\sim\theta}|T_{*}(\mathbf{y}D^{-1},\mathbf{Z},\mathbf{W})|\nonumber \\
	&\qquad\qquad \overset{(iii)}{\leq}D({\mathbb A}(s,N_{1},N_{2})+\eta),\label{eq: scale minimax eq 5}
	\end{align}
	where $(i)$ follows by Lemma \ref{lem: scaling para space}, $(ii)$
	follows by Lemma \ref{lem: obs equi} and $(iii)$ follows by (\ref{eq: scale minimax eq 4}).
	
	Now we show $\tilde{T}\in\mathcal{C}_{\alpha}(\Theta(s,DN_{1},DN_{2}))$. Notice
	that 
	\begin{align*}
	\mathbb{P}_{(\mathbf{y},\mathbf{Z},\mathbf{W})\sim\theta_{0}}(\beta_{0}\in\tilde{T}(\mathbf{y},\mathbf{Z},\mathbf{W})) & =\mathbb{P}_{(\mathbf{y},\mathbf{Z},\mathbf{W})\sim\theta_{0}}(\beta_{1}D\in D\cdot T_{*}(\mathbf{y}D^{-1},\mathbf{Z},\mathbf{W}))\\
	& =\mathbb{P}_{(\mathbf{y},\mathbf{Z},\mathbf{W})\sim\theta_{0}}(\beta_{1}\in T_{*}(\mathbf{y}D^{-1},\mathbf{Z},\mathbf{W}))\\
	& =\mathbb{P}_{(\mathbf{y},\mathbf{Z},\mathbf{W})\sim(\theta_{1}\odot D)}(\beta_{1}\in T_{*}(\mathbf{y}D^{-1},\mathbf{Z},\mathbf{W}))\\
	& \overset{(i)}{=}\mathbb{P}_{(\mathbf{y}D^{-1},\mathbf{Z},\mathbf{W})\sim\theta_{1}}(\beta_{1}\in T_{*}(\mathbf{y}D^{-1},\mathbf{Z},\mathbf{W}))\\
	& \overset{(ii)}{\geq}1-\alpha,
	\end{align*}
	where $(i)$ follows by Lemma \ref{lem: obs equi} and $(ii)$ follows
	by $T_{*}\in\mathcal{C}_{\alpha}(\widetilde{\Theta}_{N_{1},N_{2}}(s))$ and
	$\theta_{1}\in\widetilde{\Theta}_{N_{1},N_{2}}(s)$. Hence, $\tilde{T}\in\mathcal{C}_{\alpha}(\Theta(s,DN_{1},DN_{2}))$
	and 
	\[
	\sup_{\theta_{1}\in\Theta(s,DN_{1},DN_{2})}\mathbb{E}_{(\mathbf{y},\mathbf{Z},\mathbf{W})\sim\theta_{1}}|\tilde{T}(\mathbf{y},\mathbf{Z},\mathbf{W})|\geq {\mathbb A}(s,DN_{1},DN_{2}).
	\]
	
	By (\ref{eq: scale minimax eq 5}), it follows that 
	$$D({\mathbb A}(s,N_{1},N_{2})+\eta)\geq {\mathbb A}(s,DN_{1},DN_{2}).$$
	Since $\eta>0$ is arbitrary, we have $D{\mathbb A}(s,N_{1},N_{2})\geq {\mathbb A}(s,DN_{1},DN_{2})$. 
\end{proof}

\subsection{Proof of auxiliary results used in proving Theorem \ref{THM:11}}

\begin{proof}[Proof of Lemma \ref{lem: bnd 1}]
	Let $\lambda_{\min}(\cdot)$ denote the minimal eigenvalue. Then
	\[
	\PP\left(\Zb^{\top}(\Wb\Wb^{\top})^{-1}\Zb>a\right)\leq\PP\left(\lambda_{\max}[(\Wb\Wb^{\top})^{-1}]\|\Zb\|_{2}^{2}>a\right)=\PP\left(\|\Zb\|_{2}^{2}>\lambda_{\min}(\Wb\Wb^{\top})a\right).
	\]
	
	By Corollary 5.35 of \citet{vershynin2010introduction}, we have that
	\[
	\PP\left(\sqrt{\lambda_{\min}(\Wb\Wb^{\top})}<\sqrt{2n}-\sqrt{n}-0.1\sqrt{n}\right)\leq2\exp(-0.01n/2).
	\]
	
	Since $\sqrt{2}-1-0.1>0.3$, we have that 
	\begin{align*}
	\PP\left(\Zb^{\top}(\Wb\Wb^{\top})^{-1}\Zb>a\right) & \leq\PP\left(\|\Zb\|_{2}^{2}>\lambda_{\min}(\Wb\Wb^{\top})a\right)\\
	& \leq2\exp(-0.01n/2)+\PP\left(\|\Zb\|_{2}^{2}>0.09na\right)\\
	& \leq2\exp(-0.01n/2)+\frac{\EE\|\Zb\|_{2}^{2}}{0.09na}\\
	& =2\exp(-0.01n/2)+\frac{n}{0.09na}\\
	& <2\exp(-0.01n/2)+\frac{12}{a}.
	\end{align*}
\end{proof}

\begin{proof}[Proof of Lemma \ref{lem: bnd 2}]
	We first notice that 
	\[
	\EE\|\xib\|_{2}^{2}=\EE\xib^{\top}\xib=\EE\trace(\xib\xib^{\top})=\trace(\EE\xib\xib^{\top})=\trace(\Sigmab).
	\]
	
	Then the desired result follows by Markov's inequality
	\[
	\PP(\|\xib\|_{2}>x)=\PP(\|\xib\|_{2}^{2}>x^{2})\leq x^{-2}\EE\|\xib\|_{2}^{2}.
	\]
\end{proof}

\begin{proof}[Proof of Lemma \ref{lem: noiseless}]
	We use an argument that is inspired by the proof of Proposition 1
	of \citet{carpentier2019}. Let $C_{1}>0$ be a constant to be chosen
	later. Let $\mu_{n}(\cdot)$ denote the probability measure of the
	Gaussian distribution $\Ncal(0,\II_{p-1}C_{1}^{2}(p-1)^{-1})$. Recall
	for the parameter $\theta=(\beta,\gammab,\Sigmab,\sigma)\in\Theta_{*}(r)$,
	we have $\sigma=0$ and $\Sigmab=\II_{p}$. Thus, we can write $\yb=\Zb\beta+\Wb\gammab$
	with $\gammab\in\RR^{p-1}$, where entries of $\Zb\in\RR^{n}$ and $\Wb\in\RR^{n\times (p-1)}$ are
	i.i.d standard normal random variables.
	
	Since $p\geq2n+1$, we can without loss of generality set $p=2n+1$
	and hence $\Wb\in\RR^{n\times 2n}$. If $p>2n+1$, then we can simply
	apply this distribution to the first $(2n+1)$ elements of $\gammab$
	and leave the other $(p-2n-1)$ elements to be zero; since doing so
	would create additional unnecessary notations without really changing
	the argument, we work with $p=2n+1$ for notational simplicity. Define
	two probability measures
	\[
	\PP_{[A]}=\int_{\RR^{p-1}}\PP_{(0,\gammab,\II_{p},0)}d\mu_{n}(\gammab)
	\]
	and 
	\[
	\PP_{[B]}=\int_{\RR^{p-1}}\PP_{(r_{n},\gammab,\II_{p},0)}d\mu_{n}(\gammab),
	\]
	where $r_{n}>0$ is a sequence to be determined. We define the event
	\[
	\Acal=\left\{ \Zb^{\top}(\Wb\Wb^{\top})^{-1}\Zb\leq C_{2}\right\} ,
	\]
	where $C_{2}>0$ is a constant to be determined. For a fixed $(\Wb,\Zb)$,
	$\yb$ follows $\Ncal(0,\Wb\Wb^{\top}C_{1}^{2}(p-1)^{-1})$ under
	$\PP_{[A]}$ and follows $\Ncal(\Zb r_{n},\Wb\Wb^{\top}C_{1}^{2}(p-1)^{-1})$
	under $\PP_{[B]}$.
	
	Let $\Bcal=\{v\in\RR^{p-1}:\ \|v\|_{2}\leq1\}$. Let $\tmu_{n}(\cdot)$
	be the truncated Gaussian measure on $\Bcal$, i.e., $\tmu_{n}(C)=\mu_{n}(C\bigcap\Bcal)/\mu_{n}(\Bcal)$
	for any set $C$. Define 
	\[
	\PP_{\tA}=\int_{\RR^{p-1}}\PP_{(0,\gammab,\II_{p},0)}d\tmu_{n}(\gammab).
	\]
	and
	\[
	\PP_{\tB}=\int_{\RR^{p-1}}\PP_{(r_{n},\gammab,\II_{p},0)}d\tmu_{n}(\gammab).
	\]
	
	The rest of proof proceeds in three steps in which we bound (1) difference
	between $\PP_{[A]}$ and $\PP_{[B]}$, (2) difference between $\PP_{[A]}$
	and $\PP_{\tA}$ and (3) difference between $\PP_{[B]}$ and $\PP_{\tB}$.
	
	\textbf{Step 1:} bound the difference between $\PP_{[A]}$ and $\PP_{[B]}$
	
	Let $\tgammab$ be a random vector that is independent of $(\yb,\Wb,\Zb)$
	and has the distribution $\mu_{n}$. Let $\mathcal{Y}\times\mathcal{W}\times\mathcal{Z}$
	be the support of $(\yb,\Wb,\Zb)$. We notice that 
	\begin{align*}
	\EE_{\PP_{[B]}}\psi(\yb,\Wb,\Zb) & =\int_{\mathcal{Y\times\mathcal{W}\times\mathcal{Z}}}\psi(\yb,\Wb,\Zb)d\PP_{[B]}(\yb,\Wb,\Zb)\\
	& =\int_{\mathcal{Y\times\mathcal{W}\times\mathcal{Z}}}\psi(\yb,\Wb,\Zb)\left(\int_{\RR^{n}}d\PP_{(r_{n},\gammab,\II_{p},0)}(\yb,\Wb,\Zb)d\mu_{n}(\gammab)\right)\\
	& \overset{(i)}{=}\int_{\RR^{p-1}}\int_{\mathcal{Y\times\mathcal{W}\times\mathcal{Z}}}\psi(\yb,\Wb,\Zb)d\PP_{(r_{n},\gammab,\II_{p},0)}(\yb,\Wb,\Zb)d\mu_{n}(\gammab)\\
	& =\int_{\RR^{p-1}}\int_{\mathcal{\mathcal{W}\times\mathcal{Z}}}\psi(\Zb r_{n}+\Wb\gammab,\Wb,\Zb)d\PP_{(r_{n},\gammab,\II_{p},0)}(\Wb,\Zb)d\mu_{n}(\gammab)\\
	& =\int_{\mathcal{\mathcal{W}\times\mathcal{Z}}}\int_{\RR^{p-1}}\psi(\Zb r_{n}+\Wb\gammab,\Wb,\Zb)d\mu_{n}(\gammab)d\PP_{(r_{n},\gammab,\II_{p},0)}(\Wb,\Zb)\\
	& =\EE\psi(\Zb r_{n}+\Wb\tgammab,\Wb,\Zb)
	\end{align*}
	with $\EE$ being expectation over random elements $\Wb$, $\Zb$
	and $\tgammab$, where $(i)$ and $(ii)$ follow by Fubini's theorem
	(since $\psi(\Zb r_{n}+\Wb\gammab,\Wb,\Zb)d\PP_{(r_{n},\gammab,\II_{p},0)}(\Wb,\Zb)$
	is integrable). Here, notice that $\Wb$, $\Zb$ and $\tgammab$ are
	mutually independent, where entries of $\Wb$ and $\Zb$ follow the
	standard normal distribution.
	
	Similarly, we have 
	\[
	\EE_{\PP_{[A]}}\psi(\yb,\Wb,\Zb)=\EE\psi(\Wb\tgammab,\Wb,\Zb).
	\]
	
	Let $Q_{(w,z;r_{n})}(\cdot)$ denote the distribution 
	$$\Ncal(zr_{n},ww^{\top}C_{1}^{2}(p-1)^{-1}).$$
	Then we have 
	\begin{align*}
	\EE\psi(\Wb\tgammab,\Wb,\Zb)&=\EE\left(\int_{\mathcal{Y}}\psi(y,\Wb,\Zb)Q_{(\Wb,\Zb;0)}(dy)\right)\\&=\EE\left(\int_{\mathcal{Y}}\psi(y,\Wb,\Zb)Q_{(\Wb,\Zb;r_{n})}(dy)\frac{Q_{(\Wb,\Zb;0)}(dy)}{Q_{(\Wb,\Zb;r_{n})}(dy)}\right).
	\end{align*}
	
	Moreover,
	
	\begin{align*}
	& \left|\EE_{\PP_{[A]}}\psi-\EE_{\PP_{[B]}}\psi\right|\\
	& \qquad  =\left|\EE[\psi(\Zb r_{n}+\Wb\tgammab,\Wb,\Zb)-\psi(\Wb\tgammab,\Wb,\Zb)]\right|\\
	& \qquad   =\left|\EE\left\{ \EE\left[\psi(\Zb r_{n}+\Wb\tgammab,\Wb,\Zb)-\psi(\Wb\tgammab,\Wb,\Zb)\mid\Wb,\Zb\right]\right\} \right|\\
	& \qquad   =\left|\EE\left\{ \boldone_{\Acal}\times\EE\left[\psi(\Zb r_{n}+\Wb\tgammab,\Wb,\Zb)-\psi(\Wb\tgammab,\Wb,\Zb)\mid\Wb,\Zb\right]\right\} \right|\\
	& \qquad   \qquad+\left|\EE\left\{ \boldone_{\Acal^{c}}\times\EE\left[\psi(\Zb r_{n}+\Wb\tgammab,\Wb,\Zb)-\psi(\Wb\tgammab,\Wb,\Zb)\mid\Wb,\Zb\right]\right\} \right|\\
	&\qquad   \leq\left|\EE\left\{ \boldone_{\Acal}\times\EE\left[\psi(\Zb r_{n}+\Wb\tgammab,\Wb,\Zb)-\psi(\Wb\tgammab,\Wb,\Zb)\mid\Wb,\Zb\right]\right\} \right|+\PP(\Acal^{c})\\
	& \qquad   \leq\EE\left|\boldone_{\Acal}\times\EE\left[\psi(\Zb r_{n}+\Wb\tgammab,\Wb,\Zb)-\psi(\Wb\tgammab,\Wb,\Zb)\mid\Wb,\Zb\right]\right|+\PP(\Acal^{c})\\
	& \qquad   =\EE\left|\boldone_{\Acal}\times\EE\left[\psi(\Wb\tgammab,\Wb,\Zb)\left(\frac{dQ_{(\Wb,\Zb;r_{n})}}{dQ_{(\Wb,\Zb;0)}}(\Wb\tgammab)-1\right)\mid\Wb,\Zb\right]\right|+\PP(\Acal^{c})\\
	& \qquad   \leq\EE\left|\boldone_{\Acal}\times\EE\left[\left|\frac{dQ_{(\Wb,\Zb;r_{n})}}{dQ_{(\Wb,\Zb;0)}}(\Wb\tgammab)-1\right|\mid\Wb,\Zb\right]\right|+\PP(\Acal^{c})\\
	& \qquad   =\EE\left(\boldone_{\Acal}\times\TV(Q_{(\Wb,\Zb;r_{n})},Q_{(\Wb,\Zb;0)})\right)+\PP(\Acal^{c})\\
	& \qquad   \overset{(i)}{\leq}\EE\left(\boldone_{\Acal}\times\sqrt{\KL(Q_{(\Wb,\Zb;r_{n})},Q_{(\Wb,\Zb;0)})/2}\right)+\PP(\Acal^{c}),
	\end{align*}
	where $(i)$ follows by the first Pinsker's inequality (Lemma 2.5
	of \citet{tsybakov2009introduction}). By Lemma \ref{lem: KL computation},
	we have 
	\[
	\KL(Q_{(\Wb,\Zb;r_{n})},Q_{(\Wb,\Zb;0)})=\frac{1}{2}r_{n}^{2}\Zb^{\top}\left[\Wb\Wb^{\top}C_{1}^{2}n^{-1}\right]^{-1}\Zb=\frac{nr_{n}^{2}}{2C_{1}^{2}}\Zb^{\top}(\Wb\Wb^{\top})^{-1}\Zb.
	\]
	
	Thus, 
	\[
	\boldone_{\Acal}\times\KL(Q_{(\Wb,\Zb;r_{n})},Q_{(\Wb,\Zb;0)})=\frac{nr_{n}^{2}}{2C_{1}^{2}}\Zb^{\top}(\Wb\Wb^{\top})^{-1}\Zb\times\boldone_{\Acal}\leq\frac{nr_{n}^{2}C_{2}}{2C_{1}^{2}}
	\]
	and 
	\[
	\left|\EE_{\PP_{[A]}}\psi-\EE_{\PP_{[B]}}\psi\right|\leq\frac{n^{1/2}r_{n}\sqrt{C_{2}}}{2C_{1}}+\PP(\Acal^{c}).
	\]
	
	Fix an arbitrary $\alpha>0$. By Lemma \ref{lem: bnd 1}, there exists
	a constant $C_{2}$ depending only on $\alpha$ such that $\PP(\Acal^{c})\leq\alpha/4$.
	Then we take $r_{n}=n^{-1/2}C_{2}^{-1/2}C_{1}\alpha/2$ and obtain
	that 
	\begin{equation}
	\left|\EE_{\PP_{[A]}}\psi-\EE_{\PP_{[B]}}\psi\right|\leq\alpha/2.\label{eq: lem noiseless  eq 9}
	\end{equation}
	
	\textbf{Step 2:} bound the difference between $\PP_{[B]}$ and $\PP_{\tB}$.
	
	Recall from Step 1 that
	\begin{align*}
	\EE_{\PP_{[B]}}\psi(\yb,\Wb,\Zb) & =\int_{\mathcal{\mathcal{W}\times\mathcal{Z}}}\int_{\RR^{p-1}}\psi(\Zb r_{n}+\Wb\gammab,\Wb,\Zb)d\mu_{n}(\gammab)d\PP_{(r_{n},\gammab,\II_{p},0)}(\Wb,\Zb)\\
	& =\int_{\RR^{p-1}}\phi(\gammab)d\mu_{n}(\gammab),
	\end{align*}
	where $\phi(\gammab)=\int_{\mathcal{\mathcal{W}\times\mathcal{Z}}}\psi(\Zb r_{n}+\Wb\gammab,\Wb,\Zb)d\PP_{(r_{n},\gammab,\II_{p},0)}(\Wb,\Zb)$.
	Similarly, we have 
	\[
	\EE_{\PP_{\tB}}\psi(\yb,\Wb,\Zb)=\int_{\RR^{p-1}}\phi(\gammab)d\tmu_{n}(\gammab)=\frac{1}{\mu_{n}(\Bcal)}\int_{\Bcal}\phi(\gammab)d\mu_{n}(\gammab).
	\]
	
	We observe 
	\begin{align*}
	& \left|\EE_{\PP_{[B]}}\psi(\yb,\Wb,\Zb)-\EE_{\PP_{\tB}}\psi(\yb,\Wb,\Zb)\right|\\
	& =\left|\int_{\RR^{p-1}}\phi(\gammab)d\mu_{n}(\gammab)-\frac{1}{\mu_{n}(\Bcal)}\int_{\Bcal}\phi(\gammab)d\mu_{n}(\gammab)\right|\\
	& =\left|\int_{\Bcal}\phi(\gammab)d\mu_{n}(\gammab)+\int_{\Bcal^{c}}\phi(\gammab)d\mu_{n}(\gammab)-\frac{1}{\mu_{n}(\Bcal)}\int_{\Bcal}\phi(\gammab)d\mu_{n}(\gammab)\right|\\
	& \leq\left|\int_{\Bcal}\phi(\gammab)d\mu_{n}(\gammab)-\frac{1}{\mu_{n}(\Bcal)}\int_{\Bcal}\phi(\gammab)d\mu_{n}(\gammab)\right|+\left|\int_{\Bcal^{c}}\phi(\gammab)d\mu_{n}(\gammab)\right|\\
	& = \left|1-\frac{1}{\mu_{n}(\Bcal)}\right|\times\left|\int_{\Bcal}\phi(\gammab)d\mu_{n}(\gammab)\right|+\left|\int_{\Bcal^{c}}\phi(\gammab)d\mu_{n}(\gammab)\right|\\
	& \overset{(i)}{\leq}\left|1-\frac{1}{\mu_{n}(\Bcal)}\right|+\mu_{n}(\Bcal^{c})=\frac{\mu_{n}(\Bcal^{c})}{1-\mu_{n}(\Bcal^{c})}+\mu_{n}(\Bcal^{c}),
	\end{align*}
	where $(i)$ follows by $|\phi(\gammab)|\leq1$ (since $|\psi|\leq1$).
	By Lemma \ref{lem: bnd 2}, 
	\[
	\mu_{n}(\Bcal^{c})\leq\trace\left(\II_{p-1}C_{1}^{2}(p-1)^{-1}\right)=C_{1}^{2}.
	\]
	
	Now we choose $C_{1}=\sqrt{\alpha/12}$. This means that $\mu_{n}(\Bcal^{c})\leq\alpha/12$.
	Hence, $\mu_{n}(\Bcal^{c})<1/2$.
	\begin{align}
	& \left|\EE_{\PP_{[B]}}\psi(\yb,\Wb,\Zb)-\EE_{\PP_{\tB}}\psi(\yb,\Wb,\Zb)\right|\nonumber \\
	& \leq\frac{\mu_{n}(\Bcal^{c})}{1-\mu_{n}(\Bcal^{c})}+\mu_{n}(\Bcal^{c})\leq2\mu_{n}(\Bcal^{c})+\mu_{n}(\Bcal^{c})\leq\alpha/4.\label{eq: lem noiseless eq 10}
	\end{align}
	
	\textbf{Step 3:} bound the difference between $\PP_{[A]}$ and $\PP_{\tA}$.
	
	Similarly to Step 2, we can show that
	\begin{equation}
	\left|\EE_{\PP_{[A]}}\psi(\yb,\Wb,\Zb)-\EE_{\PP_{\tA}}\psi(\yb,\Wb,\Zb)\right|\leq\alpha/4.\label{eq: lem noiseless eq 11}
	\end{equation}
	
	Now we combine (\ref{eq: lem noiseless  eq 9}), (\ref{eq: lem noiseless eq 10})
	and (\ref{eq: lem noiseless eq 11}), obtaining
	\[
	\left|\EE_{\PP_{\tA}}\psi(\yb,\Wb,\Zb)-\EE_{\PP_{\tB}}\psi(\yb,\Wb,\Zb)\right|\leq\alpha.
	\]
	
	Since $\sup_{\theta\in\Theta_{*}(0)}\EE_{\theta}\psi\leq\alpha$ and
	$\PP_{\tA}$ is by definition a mixed of distributions in $\Theta_{*}(0)$, we have
	$\EE_{\PP_{\tA}}\psi(\yb,\Wb,\Zb)\leq\alpha$, which means 
	\[
	\EE_{\PP_{\tB}}\psi(\yb,\Wb,\Zb)\leq2\alpha.
	\]
	
	Notice that $\PP_{\tB}$ is a mixture of distributions in $\Theta_{*}(r_{n})$,
	we have that 
	$$\inf_{\theta\in\Theta_{*}(r_{n})}\EE_{\theta}\psi\leq2\alpha.$$
	The proof is complete since 
	$$r_{n}=n^{-1/2}C_{2}^{-1/2}C_{1}\alpha/2$$
	with $C_{1},C_{2}$ depending only on $\alpha$.
\end{proof}

\section{Comparison of priors}
To provide a comparison of the priors, we outline an adaptation of the prior from \cite{cai2017confidence} and compare with our prior for the proof of minimax lower bound. This comparison illustrates the main differences. (We thank an anonymous reviewer for suggesting this.)

A simple adaptation of the prior considered in \cite{cai2017confidence} under 
our notation: $\mathbf y=\mathbf Z \beta+ \mathbf W\boldsymbol{\gamma}+\boldsymbol \varepsilon$
and $\boldsymbol{\Sigma}=\begin{pmatrix}\boldsymbol{\pi}^{\top}\boldsymbol{\pi}+\sigma_{\mathbf{V}}^{2} & \boldsymbol{\pi}^{\top}\\
\boldsymbol{\pi} & I_{p-1}
\end{pmatrix}$. Let the parameter be indexed by $(\beta,\boldsymbol{\gamma},\boldsymbol{\pi},\sigma_{\mathbf{V}},\sigma_{\varepsilon})$. 

The priors used by \cite{cai2017confidence} in Equation (7.13) on
page 636 therein can be adapted (switching ) as follows. Given $(\beta_{*},\boldsymbol{\gamma}_{*},0,1,\sigma_{0})$,
their prior is 
\begin{align*}
\beta & =\beta_{*}\\
\boldsymbol{\gamma} & =\boldsymbol{\gamma}_{*}+c_{1}\sqrt{\frac{\log(p/m^{2})}{n}}\boldsymbol{\delta}\\
\boldsymbol{\pi} & =c_{2}\sqrt{\frac{\log(p/m^{2})}{n}}\boldsymbol{\delta}\\
\sigma_{\mathbf{V}} & =\sqrt{1-c_{2}^{2}\frac{m\log(p/m^{2})}{n}}\\
\sigma_{\varepsilon} & =\sigma_{0}
\end{align*}
where $\boldsymbol{\delta}$ is from the uniform distribution from
the set $\mathcal{M}=\{\boldsymbol{v}\in\{0,1\}^{p-1}:\ \|\boldsymbol{v}\|_{0}=m\}$.
Here, $\|(\beta_{*},\boldsymbol{\gamma}_{*}^{\top})^{\top}\|_{0}=m$
and $c_{1}>0$ is a constant. 

Here is our prior in Definition 3 from Appendix A. Given $(\beta_{*},\boldsymbol{\gamma}_{*},\boldsymbol{\pi}_{*},\sigma_{\mathbf{V},*},\sigma_{\varepsilon,*})$
with $\|\boldsymbol{\pi}_{*}\|_{0}=s$ (with $m\asymp s$), we define
\begin{align*}
\beta & =\beta_{*}-h\qquad{\rm with}\qquad h=\frac{ds\log p}{n}\\
\boldsymbol{\gamma} & =\boldsymbol{\gamma}_{*}+h\boldsymbol{\pi}+r(1-h)\sigma_{\varepsilon,*}\sqrt{2d\log(p)/n}\boldsymbol{\delta}=\boldsymbol{\gamma}_{*}+\frac{ds\log p}{n}\boldsymbol{\pi}_{*}+\sigma_{\mathbf{V},*}\sqrt{\frac{2d\log p}{n}}\boldsymbol{\delta}\\
\boldsymbol{\pi} & =\boldsymbol{\pi}_{*}+\sigma_{\mathbf{V},*}\sqrt{\frac{2d\log p}{n}}\boldsymbol{\delta}\\
\sigma_{\mathbf{V}} & =\sigma_{\mathbf{V},*}\sqrt{1-\frac{ds\log p}{n}}\\
\sigma_{\varepsilon} & =\sigma_{\varepsilon,*}
\end{align*}
where $d>0$ is a constant and $r=\sigma_{\mathbf{V},*}/\sigma_{\varepsilon,*}$.

From the above comparison, the difference between our prior and that in \cite {cai2017confidence} is not simply that  
$\boldsymbol{\gamma}$ and $\boldsymbol{\pi}$
are switched.  Notice
that in our prior, the construction of $\boldsymbol{\gamma}$ depends
on $\boldsymbol{\pi}_{*}$, whereas in \cite{cai2017confidence},
$\boldsymbol{\pi}_{*}$ is set to be zero. A priori, it is not obvious
whether there exists a construction of $\boldsymbol{\gamma}$ under
nonzero $\boldsymbol{\pi}_{*}$ such that the calculation in our Appendix A would go through. From this perspective, the prior
of \cite{cai2017confidence} is just a special case of our construction. For the \textit{uniform} non-testability result to hold, 
we need to build the prior around a general point  $(\beta_{*},\boldsymbol{\gamma}_{*},\boldsymbol{\pi}_{*},\sigma_{\mathbf{V},*},\sigma_{\varepsilon,*})$.

\end{document}